\documentclass{amsart}
\usepackage{amssymb, amsmath,mathtools,bm}
\usepackage{mathrsfs}
\usepackage{amscd} 
\usepackage{bbm}
\usepackage{verbatim}
\usepackage{stmaryrd}
\usepackage[dvipsnames]{xcolor}
\usepackage{a4wide}
\usepackage{enumitem}
\usepackage[backend=biber,maxbibnames=199,doi=false,isbn=false,url=false]{biblatex}
\usepackage[colorlinks,linkcolor={blue},citecolor={blue},urlcolor={purple},]{hyperref}

\usepackage[colorinlistoftodos,prependcaption,textsize=tiny]{todonotes}

\newcommand{\dd}{\,\mathrm{d}}
\newcommand{\ddd}{\mathrm{d}}
\newcommand{\e}{\mathrm{e}}

\renewcommand{\MR}{\mathrm{MR}}
\newcommand{\UH}{\mathscr{L}_2(U,H)}
\newcommand{\one}{\mathbbm{1}}
\newcommand{\<}{\langle}
\newcommand{\?}{\rangle}

\newcommand{\nn}{|\!|\!|}

\newcommand{\into}{\hookrightarrow}

\DeclareMathOperator*{\esssup}{\textup{ess\,sup}}

\newcommand{\loc}{\mathrm{loc}}
\newcommand{\ceqq}{\coloneqq}
\newcommand{\eqqc}{\eqqcolon}
\newcommand{\col}{\colon}

\renewcommand{\O}{\mathcal{O}} 
\newcommand{\bphieps}{\Psi^{\eps}}

\newcommand{\pr}{\mathbb{P}}

\newcommand{\Ls}{\mathbb{L}}

\newcommand{\Hs}{\mathbb{H}}

\newcommand{\Tor}{\mathbb{T}}

\renewcommand{\div}{\mathrm{div}}

\renewcommand{\Tor}{\mathbb{T}^2}

\theoremstyle{plain}
\newtheorem{theorem}{Theorem}[section]
\theoremstyle{remark}
\newtheorem{remark}[theorem]{Remark}
\newtheorem{example}[theorem]{Example}
\theoremstyle{plain}
\newtheorem{corollary}[theorem]{Corollary}
\newtheorem{lemma}[theorem]{Lemma}
\newtheorem{proposition}[theorem]{Proposition}
\newtheorem{definition}[theorem]{Definition}

\newtheorem{assumption}[theorem]{Assumption}

\numberwithin{equation}{section}

\def\N{{\mathbb N}}

\def\R{{\mathbb R}}

\newcommand{\F}{\mathcal{F}}
\newcommand{\G}{\mathcal{G}}
\renewcommand{\P}{\mathbb{P}}

\newcommand{\om}{\omega}
\newcommand{\Om}{\Omega}
\newcommand{\eps}{\varepsilon}

\usepackage{stmaryrd}

\allowdisplaybreaks

\begin{document}

\author[Esm\'ee Theewis]{Esm\'ee Theewis}
\address[E. Theewis]{Delft Institute of Applied Mathematics\\
Delft University of Technology \\ P.O. Box 5031\\ 2600 GA Delft\\The
Netherlands}
\email{e.s.theewis@tudelft.nl}


\date{October 1, 2026}

\title[Large deviations beyond the coercive case]{Large deviations for stochastic evolution equations beyond the coercive case}  



\begin{abstract} 
Using the weak convergence approach, we prove the small-noise large deviation principle (LDP) for stochastic evolution equations in an $L^2$-setting. As the coefficients are allowed to be non-coercive, our framework encompasses a much broader scope than variational settings. In place of coercivity, we require only well-posedness of the stochastic evolution equation and two concrete, verifiable a priori estimates. 

Furthermore, we accommodate drift nonlinearities satisfying a modified criticality condition and we allow for vanishing drift perturbations. The latter permits the inclusion of It\^o--Stratonovich correction terms, enabling the treatment of both noise interpretations.  

We apply our results to a reaction-diffusion system that lacks coercivity. In separate joint work, the present framework has been used to establish an LDP for the 3D primitive equations with full transport noise, further demonstrating its versatility. 

Finally, we show that even in the coercive case, our framework yields new LDP results for equations with critical nonlinearities that rely on our modified criticality condition, including the stochastic 2D Allen--Cahn equation in the weak setting.
\end{abstract}

\maketitle
\addtocontents{toc}{\protect\setcounter{tocdepth}{2}}

\section{Introduction}  

We consider an abstract stochastic evolution equation with small $\eps>0$:
\begin{equation}
\label{eq:see_intro}
\left\{
\begin{aligned}
&\ddd Y^\varepsilon + A(t,Y^\varepsilon)\dd t=  \sqrt{\varepsilon} \, B(t,Y^\varepsilon)\dd W,\\
&Y^\varepsilon(0)=x,    
\end{aligned}
\right.
\end{equation} 
where $W$ is a $U$-cylindrical Brownian motion. 
Over the last two decades, large deviations have been studied intensively for SPDEs in the variational setting  \cite{liu09,hongliliu21,kumarmohan22,pan24,TV24}, using the weak convergence approach developed in \cite{budhidupuis01}. So far, these works have pushed the boundary in reducing the assumptions on the pair of coefficients $(A,B)$ in \eqref{eq:see_intro}. In the most recent works, the focus has been on weakening   monotonicity assumptions \cite{pan24,TV24}, on relaxing  growth assumptions for $B$ \cite{pan24,TV24} (e.g.\ allowing for gradient noise), and on handling non-compact embeddings in the underlying Gelfand triple \cite{kumarmohan22,TV24} (allowing for unbounded spatial  domains in applications). 
 
A common assumption in the works mentioned above -- and one that lies at the core of the `variational setting' --   is \emph{coercivity of $(A,B)$}. If $(V,H,V^*)$ is the Gelfand triple in the variational setting and  $\langle \cdot ,\cdot \rangle$ denotes the duality pairing between  $V^*$ and $V$, then 
the coercivity condition is typically formulated as:  
For all $T>0$, there exist $\theta, M>0$ and $\phi\in L^2(0,T)$ such that 
\begin{equation}\label{eq:def coercive}
\<A(t,v),v\?-\tfrac{1}{2}\|B(t,v)\|_{\mathscr{L}_2(U,H)}^2 \geq \theta\|v\|_V^2-M\|v\|_H^2-|\phi(t)|^2,\qquad v\in V,\, t\in[0,T]. 
\end{equation} 
The coercivity condition yields a priori $L^2(\Om;C(0,T;H)\cap L^2(0,T;V))$-estimates for solutions to \eqref{eq:see_intro} (see e.g.\ \cite{AV22variational,liurockner15}), which can be used to prove global well-posedness for \eqref{eq:see_intro}. 

However, there exist many SPDEs for which coercivity fails but which can nevertheless be formulated in a similar setting. 
In this paper, we address a new type of generalization: we treat \textbf{non-coercive} pairs $(A,B)$, while using the solution notion from the variational setting. We assume that there is a Gelfand triple $(V,H,V^*)$ of Hilbert spaces, that the initial condition $x$ belongs to $H$, and that the solution $Y^\eps$ belongs to  $C([0,T];H)\cap L^2(0,T;V)$ a.s. Because solutions are required to have this $L^2(0,T;V)$ regularity, we refer to this framework as an `$L^2$-setting'.   
Prominent classes of non-coercive SPDEs tractable in an $L^2$-setting include reaction-diffusion systems and 3D fluid models.  
We refer to Example \ref{ex:applications} for a further discussion including references.   
Since coercivity fails in these examples, well-posedness of \eqref{eq:see_intro} needs to be proved in alternative ways, and generally, $L^2(\Om;C(0,T;H)\cap L^2(0,T;V))$-estimates are unavailable. 
For instance, it may happen that energy bounds merely hold in probability:
\begin{equation}\label{eq:weak energy}
  \lim_{\gamma\to\infty}\P(\|Y^\eps\|_{C(0,T;H)\cap L^2(0,T;V)}>\gamma)=0. 
\end{equation}
This does not pose a problem for applying the results presented in this paper. Indeed, in our main result, Theorem \ref{th:LDP general}, we prove that the following conditions are already sufficient for the LDP:
\begin{enumerate}[label=\textup{\roman*.},ref=\textup{\roman*.}]
  \item\label{1} a structural assumption on $(A,B)$ (see Assumption \ref{ass:critvarsettinglocal}) 
  \item\label{2} well-posedness of \eqref{eq:see_intro}
  \item\label{3} a deterministic a priori estimate  for the skeleton equation associated to \eqref{eq:see_intro}
  \item\label{4} uniform boundedness \emph{in probability} for a family of SPDE solutions associated to \eqref{eq:see_intro} 
\end{enumerate}
We emphasize that in \ref{1}, we do not require monotonicity, coercivity, nor a compact embedding in the Gelfand triple. 
The advantage of the criteria above is that  one merely  has to establish  PDE estimates and probabilistic estimates   \emph{specific} to the SPDE of interest. This is significantly easier than   proving an LDP  via the weak convergence approach from scratch, which would require proving difficult continuity properties of the solution maps related to  the equations mentioned in \ref{3} and \ref{4}.  
Furthermore, we provide a proof of well-posedness for the skeleton equation mentioned in \ref{3}, which is non-trivial when the SPDE itself admits only weak energy bounds of the form \eqref{eq:weak energy}. 

Conceptually, the conditions \ref{1}--\ref{4}\ also  clarify  which steps in a weak convergence approach proof are specific for the equation, and which steps can be done at a highly abstract level.

In our proof, we extend the LDP result from \cite{TV24}, which was based on the weak convergence approach. Apart from handling non-coercive coefficients,  we incorporate two further extensions:  
\begin{itemize} 
  \item  more flexible growth conditions on the nonlinearities, see Assumption \ref{ass:critvarsettinglocal} and Remark \ref{rem:difference AB}    
  \item an additional $\eps$-dependent term in the drift, in particular suited for the treatment of Stratonovich noise, see \eqref{eq:SPDE perturbed} and  Theorem \ref{th:LDP perturbed2}  
\end{itemize}

These extensions enable us to cover notable new examples where coercivity is absent.  
We present an LDP for the stochastic 2D Brusselator model with transport noise in Section \ref{sec: brusselator}, which involves a critical nonlinearity. Moreover, the framework developed here served as the basis for \cite{AT25}, where the LDP for the stochastic 3D primitive equations with Stratonovich (or It\^o) transport noise is established. Further possible  applications are discussed in Example \ref{ex:applications}.  

Finally, we also specialize our results to the coercive case. In fact, if $(A,B)$ satisfies the coercivity \eqref{eq:def coercive}, then \ref{2}--\ref{4}\ become redundant, thus \eqref{eq:def coercive} and \ref{1}\ alone are sufficient for the LDP, see Corollaries  \ref{cor:LDPcoercive}, \ref{cor:LDPcoerciveperturbed} and \ref{cor:LDPcoerciveperturbed1}. This then also leads to the LDP for several models with borderline nonlinearities that were not covered yet in the available variational LDP literature. Examples are the 2D Allen--Cahn equation  and the stochastic 2D Burger's equation in the analytically weak setting, with gradient noise (see Example \ref{ex:applications hat F}). Moreover, we can extend applications of the LDP in \cite{TV24} from It\^o to Stratonovich noise (see Remark \ref{ex:stratonovich coercive}).

\section{Main results}\label{sec:main results} 

Throughout this paper, we assume that $(V,H,V^*)$ is a Gelfand triple of real separable Hilbert spaces and $U$ is a real separable Hilbert space. We let $\mathscr{L}_2(U,H)$ denote the space of Hilbert-Schmidt operators from $U$ to $H$ and write $\nn\cdot\nn_H\ceqq \|\cdot\|_{\mathscr{L}_2(U,H)}$. Moreover, we will work with complex interpolation spaces  $V_\theta\ceqq [V^*,V]_\theta$ for $\theta\in[0,1]$. We  write  $\|\cdot\|_{\theta}\ceqq \|\cdot\|_{V_\theta}$ and will be using the following interpolation estimate:
\begin{equation}\label{eq: interpol est}
  \|v\|_{\theta}\leq \|v\|_{V^*}^{1-\theta}\|v\|_{V}^\theta, \qquad\qquad v\in V. 
\end{equation}

In this section we explain our framework and assumptions, and state our main LDP results: Theorem \ref{th:LDP general}  for (possibly) non-coercive equations and Corollary \ref{cor:LDPcoercive}  for coercive equations. Lists of  applications can be found in Examples \ref{ex:applications} and \ref{ex:applications hat F}.

Furthermore, in Subsection \ref{ss:perturbed main results}, we discuss a perturbed version of \eqref{eq:SPDE} containing an additional vanishing, $\eps$-dependent drift term. This framework covers  in particular small-noise equations written in Stratonovich form, which arise frequently in stochastic fluid dynamics. The corresponding LDP results are stated and proved in Section~\ref{sec:perturb}, together with further remarks on their applications to SPDEs with Stratonovich
noise.

\subsection{The non-coercive case}\label{ss:mainresultnoncoercive}
For $\eps>0$, we consider the following  stochastic evolution  equation: 
\begin{equation}\label{eq:SPDE}
   \begin{cases}
  &\ddd Y^\eps(t)=-A(t,Y^\eps(t))\dd t+\sqrt{\eps}B(t,Y^\eps(t))\dd W(t), \quad t\in[0,T], \\
  &Y^\eps(0)=x,
\end{cases}
\end{equation}  
where $W$ is a $U$-cylindrical Brownian motion and $x\in H$. 
We will be working with  drift and diffusion coefficients of the following form:
\begin{equation*} 
A(t,v)=A_0(t,v)v-\hat{F}(t,v)-\hat{f} \, \text{ and } \, B(t,v)=B_0(t,v)v+G(t,v)+g. 
\end{equation*}
Here, $A$ and $B$ are quasilinear or semilinear mappings, $\hat{F}$ and $G$ are nonlinear terms, and $\hat{f}$ and $g$ are inhomogeneities. 
The difference between $A$, $B$ defined above and $A$, $B$ from \cite[Asm.\ 2.2]{TV24} are our nonlinearity $\hat{F}$  and inhomogeneity $\hat{f}$ in $A$, which  generalize the terms $F$ and $f$ in \cite[Asm.\ 2.2]{TV24} (see Remark \ref{rem:difference AB} below). The $\hat{F}$-term satisfies modified, more flexible growth and Lipschitz conditions compared to $F$, which can be used to include larger classes of (critical) nonlinearities. Examples of such cases will be given in Example \ref{ex:applications hat F}.

Let us start introducing the assumptions that we use for $A$ and $B$. As a basis, we work with Assumption \ref{ass:critvarsettinglocal} below, which ensures local well-posedness for \eqref{eq:SPDE}, as we argue in Remark \ref{rem: SPDE loc well-posed}. For global well-posedness and an LDP,  one usually also requires coercivity \eqref{eq:def coercive} for the full pair $(A,B)$ (see \cite[(2.7), Th.\ 2.6]{TV24}), which we do \emph{not} require below.  
Later in this section, we introduce new and much more general conditions (Assumption \ref{ass:coer replace}) to replace the coercivity condition.  
 
\begin{assumption}[Local well-posedness conditions -- extended critical variational setting]\noindent\phantomsection\label{ass:critvarsettinglocal} 

\noindent
$A$ and $B$ are defined by 
\begin{equation}\label{eq:defAB}
A(t,v)=A_0(t,v)v-\hat{F}(t,v)-\hat{f}, \quad B(t,v)=B_0(t,v)v+G(t,v)+g, \qquad t\in\R_+,v\in V,
\end{equation}
and the following holds: 
\begin{enumerate}[label=\textit{(\arabic*)},ref=\textit{(\arabic*)}]
\item\label{it:AB mble}  $A_0\colon \R_+ \times H\to\mathscr{L}(V,V^*)$, $B_0\colon \R_+ \times H \to \mathscr{L}(V,\UH)) 
$  and $G\colon \R_+\times V\to\UH$ are Borel measurable and  $g\in L^2_\loc(\R_+;\UH)$. 

 \item\label{it:coercivelinear}  
  For all $n,T>0$, there exist  $\theta_{n,T},M_{n,T}>0$ such that for all $t\in[0,T],u\in H,v\in V$ with $\|u\|_H\leq n$:
  \begin{flalign*}
  &\qquad\qquad \<A_0(t,u)v,v\?-\tfrac{1}{2}\nn B_0(t,u)v\nn_H^2\geq \theta_{n,T}\|v\|_V^2-M_{n,T}\|v\|_H^2. \qquad &\text{\emph{(coercivity of the linear part)}}&
  \end{flalign*}
  
\item \label{it:growth AB} There exist an $m\in\N$ and $\rho_i\geq 0$, $\beta_i\in(\frac{1}{2},1)$ for $i\in\{1,\ldots, m\}$,  such that 
\begin{flalign}\label{eq: subcrit cond}
&\qquad\qquad (1+\rho_i)(2\beta_i-1)\leq 1,& \text{\emph{((sub)criticality)}}&
\end{flalign}
and for all  $n,T>0$ there exists a constant $C_{n,T}$ such that for all $t\in[0,T]$ and $u,v,w\in V$ with $\|u\|_H,\|v\|_H\leq n$, we have
\begin{align*}
\|A_0(t,u)w\|_{V^*}&\leq C_{n,T} \|w\|_V,\\
\|A_0(t,u)w-A_0(t,v)w\|_{V^*}&\leq C_{n,T}\|u-v\|_H\|w\|_V,\\
\nn B_0(t,u)w\nn_{H}&\leq C_{n,T} \|w\|_V,\\
\nn B_0(t,u)w-B_0(t,v)w\nn_{H}&\leq C_{n,T}\|u-v\|_H\|w\|_V,\\ 
\nn G(t,u)\nn_{H}&\leq C_{n,T}\textstyle{\sum_{i=1}^{m}} (1+\|u\|_{\beta_i}^{\rho_i+1}),\\
\nn G(t,u)-G(t,v)\nn_{H}&\leq C_{n,T}\textstyle{\sum_{i=1}^{m}} (1+\|u\|_{\beta_i}^{\rho_i}+\|v\|_{\beta_i}^{\rho_i})\|u-v\|_{\beta_i}.
\end{align*}
\item\label{it:hat F}  
$\hat{F}=\sum_{i=1}^{m}\hat{F}_i$ with $\hat{F}_i\colon\R_+\times V\to V_{\alpha_i}$ Borel measurable and $\alpha_i\in[0,\frac12]$,   
and there exist $\hat{\rho}_i>0$, $\hat{\beta}_i\in(\frac12,1]$ for $i\in\{1,\ldots, m\}$, such that
\begin{flalign}\label{eq: subcrit cond alpha}
&\qquad\qquad
  (1+\hat{\rho}_i)(2\hat{\beta}_i-1)\leq 1+2\alpha_i, & \text{\emph{($\alpha$-(sub)criticality)}}&
\end{flalign}
and for all $n,T>0$, $t\in[0,T]$ and $u,v\in V$ with $\|u\|_H,\|v\|_H\leq n$, 
\begin{align*} 
\|\hat{F}_i(t,u)\|_{\alpha_i}&\leq C_{n,T}  (1+\|u\|_{\hat{\beta}_i}^{\hat{\rho}_i+1}),\\
\|\hat{F}_i(t,u)-\hat{F}_i(t,v)\|_{\alpha_i}&\leq C_{n,T}(1+\|u\|_{\hat{\beta}_i}^{\hat{\rho}_i}+\|v\|_{\hat{\beta}_i}^{\hat{\rho}_i})\|u-v\|_{\hat{\beta}_i}.
\end{align*}
\item $\hat{f}=\sum_{i=1}^m \hat{f}_i$ with $\hat{f}_i\in L^{2/(1+2\alpha_i)}_{\loc}(\R_+;V_{\alpha_i})$. 
\end{enumerate}
\end{assumption}

\begin{remark}\label{rem:difference AB}
Assumption \ref{ass:critvarsettinglocal} coincides with \cite[Asm.\ 2.2]{TV24}, except for the fact that the nonlinearity $F$ and inhomogeneity $f$ appearing in \cite[Asm.\ 2.2]{TV24} are replaced by $\hat{F}$ and  $\hat{f}$ respectively. 
In fact, the case $\alpha_i=0$ for all $i$ coincides exactly with \cite[Asm.\ 2.2]{TV24}, noting that in the latter, without loss of generality, one can assume that $\rho_i>0$ for all $i$.  
Thus the terms $\hat{F}$ and $\hat{f}$ are more flexible, since the $\alpha$-(sub)criticality condition \eqref{eq: subcrit cond alpha} allows one to use non-zero $\alpha_i$ as well. In that case, the extra regularity of $\hat{F}_i$ (being $V_{\alpha_i}$-valued instead of $V_*$-valued) is traded for larger admissible $\hat{\beta}_i$ and $\hat{\rho}_i$ in the growth and local Lipschitz conditions for $\hat{F}_i$. 
\end{remark}

\begin{remark}\label{rem: SPDE loc well-posed}
If $(A,B)$  satisfies Assumption \ref{ass:critvarsettinglocal}, then for every $\eps\in[0,1]$, local well-posedness of  \eqref{eq:SPDE} holds  and there exists a unique maximal solution, due to \cite[Th.\ 5.5]{BGV}. Here we use that $(A,\sqrt{\eps}B)$ satisfies Assumption  \ref{ass:critvarsettinglocal} for  $\eps\in[0,1]$, hence it satisfies  \cite[Asm.\ 5.1]{BGV}, taking into account \cite[Rem.\ 5.2(3)]{BGV}.   

However, one cannot obtain global well-posedness from \cite[Th.\ 3.4, Th.\ 3.5]{AV22variational} or \cite[\S7]{BGV} under only Assumption \ref{ass:critvarsettinglocal}, as we do not have coercivity of $(A,B)$. Therefore, global well-posedness will be imposed separately by Assumption \ref{ass:coer replace}\ref{it:1}. 
\end{remark}

An LDP is abstractly defined as follows.

\begin{definition}\label{def: LDP}
Let $\mathcal{E}$ be a Polish space, let $(\Om,\F,\P)$ be a probability space and let $(Y^\eps)_{\eps\in(0,\eps_0)}$ be a collection of $\mathcal{E}$-valued random variables on $(\Om,\F,\P)$, for some $\eps_0>0$. Let $I\colon \mathcal{E}\to[0,\infty]$ be a function. Then $(Y^{\eps})_{\eps\in(0,\eps_0)}$  \emph{satisfies the large deviation principle (LDP) on $\mathcal{E}$} with rate function $I\colon \mathcal{E}\to[0,\infty]$ if
\begin{enumerate}[label=\emph{(\roman*)},ref=\textup{(\roman*)}]
    \item $I$ has compact sublevel sets,
    \item for all open $E\subset \mathcal{E}$: $\liminf_{\eps\downarrow0}\eps\log \P(Y^{\eps}\in E)\geq -\inf_{z\in E}I(z)$,
    \item for all closed $E\subset \mathcal{E}$: $\limsup_{\eps\downarrow0}\eps\log \P(Y^{\eps}\in E)\leq -\inf_{z\in E}I(z)$.
\end{enumerate} 
\end{definition} 

In our setting, we will use the above definition with for $\mathcal{E}$ the \emph{maximal regularity space} 
\begin{equation}\label{eq: def MR space}
    \MR(0,T)\coloneqq C([0,T];H)\cap L^2(0,T;V), \quad \|\cdot\|_{\MR(0,T)}\coloneqq \|\cdot\|_{C([0,T];H)}+\|\cdot\|_{L^2(0,T;V)}
\end{equation}
for $T>0$, which is a Banach space. From now on, the space $\MR(0,T)$ is always defined as above.

Our first main result, Theorem \ref{th:LDP general}, provides sufficient conditions under which the LDP holds for solutions $(Y^\eps)$ to \eqref{eq:SPDE}. To establish this result, we need to study two equations related to \eqref{eq:SPDE}. 
The skeleton equation is the deterministic evolution equation given by 
\begin{equation}\label{eq:skeleton}
   \begin{cases}
  &\ddd u^\psi(t)=-A(t,u^\psi(t))\dd t+B(t,u^\psi(t))\psi(t)\dd t, \quad t\in[0,T], \\
  &u^\psi(0)=x,
\end{cases}
\end{equation} 
for given $\psi\in L^2(0,T;U)$ and $x\in H$. 
Moreover, we consider the following stochastic evolution equation:
\begin{equation}\label{eq:SPDE tilted}
   \begin{cases}
  &\ddd X^\eps(t)=-A(t,X^\eps(t))\dd t+B(t,X^\eps(t))\bphieps(t)\dd t+\sqrt{\eps}B(t,X^\eps(t))\dd W, \quad t\in[0,T], \\
  &X^\eps(0)=x,
\end{cases}
\end{equation} 
where $\bphieps$ is  a predictable stochastic process taking a.s.\ values in $L^2(0,T;U)$. 

The following solution notion will be used for equations \eqref{eq:SPDE}, \eqref{eq:skeleton} and \eqref{eq:SPDE tilted}. 

\begin{definition}\label{def:sol} 
Let $A\colon[0,T]\times V\to V^*$, $B\colon[0,T]\times V\to \UH$ and let $x\in H$. Let $W$ be a $U$-cylindrical Brownian motion on a filtered probability space $(\Om,\mathcal{F},(\mathcal{F}_t)_{t\geq 0},\P)$. Consider
\begin{equation}\label{eq:SPDE plain}
     \begin{cases}
   &\ddd u(t)=-A(t,u(t))\dd t+B(t,u(t))\dd W(t), \quad t\in[0,T], \\
   &u(0)=x,
\end{cases}
\end{equation}
We say that a strongly progressively measurable process $u\colon [0,T]\times \Om\to V$ is a \emph{strong solution} to \eqref{eq:SPDE plain} if a.s. (recall \eqref{eq: def MR space}): 
\[
u\in\MR(0,T),\; A(\cdot,u(\cdot))\in L^2(0,T;V^*)+L^1(0,T;H),\; B(\cdot,u(\cdot))\in L^2(0,T;\UH),
\]
and a.s.
\begin{align}\label{eq:strong sol}
  u(t)=x-\int_0^t A(s,u(s))\dd s+\int_0^t B(s,u(s))\dd W(s) \: \text{ in }V^* \text{ for all }t\in[0,T].
\end{align}

If $B=0$, we write $u'(t)=-A(t,u(t))$ instead of $\dd u(t)=-A(t,u(t))\dd t$ in \eqref{eq:SPDE plain} and we call $u\in \MR(0,T)$ a strong solution  if $A(\cdot,u(\cdot))\in L^2(0,T;V^*)+L^1(0,T;H)$ and \eqref{eq:strong sol} holds.
\end{definition}

We now formulate a sufficient assumption that will make up for the missing coercivity of $(A,B)$.

\begin{assumption}\label{ass:coer replace} 
$(A,B)$  satisfies  Assumption  \ref{ass:critvarsettinglocal} and  $x\in H$. 
In addition, the following conditions are satisfied for some $\eps_0>0$:
\begin{enumerate}[label=\textit{(\Roman*)},ref=\textit{(\Roman*)}]
  \item\label{it:1} For every $U$-cylindrical Brownian motion $W$ on a filtered probability space $(\Om,\mathcal{F},(\mathcal{F}_t)_{t\geq 0},\P)$ and for every $\eps\in (0,\eps_0)$ and $T>0$, \eqref{eq:SPDE} has a unique strong solution $Y^\eps$ in the sense of Definition \ref{def:sol}. 
  \item\label{it:2} For every ${T}>0$ and $\psi\in L^2(0,{T};U)$: if   $u^{\psi}$ is a strong solution to \eqref{eq:skeleton}, then it satisfies 
  \begin{equation*} 
   \|u^\psi\|_{\MR(0,{T})}\leq C_{x}({T},\|\psi\|_{L^2(0,{T};U)}),
  \end{equation*}
   with $C_{x}\col \R_+\times \R_+\to\R_+$  non-decreasing in both components.  
\item\label{it:3} For every $T,K>0$ and for every collection $(\bphieps)_{\eps\in(0,\eps_0)}$ of predictable stochastic processes $\bphieps\col [0,T]\times \Om\to U$ with   $\|\bphieps\|_{L^2(0,T;U)}\leq K$ a.s., the following holds: 
if $(X^\eps)_{\eps\in(0,\eps_0)}$ is a family such that for all $\eps\in(0,\eps_0)$, $X^\eps$ is a strong solution to \eqref{eq:SPDE tilted}, then 
\begin{equation*} 
    \lim_{\gamma\to\infty}\sup_{\eps\in(0,\eps_0)}\P(\|X^\eps\|_{\MR(0,T)}>\gamma)=0. 
\end{equation*}
\end{enumerate} 
\end{assumption}

\begin{remark}
Neither existence nor uniqueness for $u^\psi$ and $X^\eps$ is required in \ref{it:2} and \ref{it:3}: only a priori estimates are needed. However, under Assumptions \ref{ass:critvarsettinglocal} and \ref{ass:coer replace},  one does have existence and uniqueness for $u^\psi$ and $X^\eps$, as will be argued in Lemma \ref{lem:consequences ass I-III}. 
\end{remark}

As explained in the Introduction, there are many models that can be described in an $L^2$-setting, but lack coercivity, so variational frameworks cannot be used. In such cases, one may still be able to check the above assumption. Examples of such models are given in Example \ref{ex:applications}. 
Let us already mention that in applications where well-posedness of the SPDE \eqref{eq:SPDE}  has not been proved yet, it is often practical to use  the `non-blow up condition' of  Lemma \ref{lem:suff non-blow up} to verify Assumption \ref{ass:coer replace}\ref{it:1}.

Our first main LDP result is as follows. 

\begin{theorem}[LDP]\label{th:LDP general}
    Let Assumption \ref{ass:coer replace} hold and let $T>0$. Then the family $(Y^\eps)_{\eps\in(0,\eps_0)}$ of solutions to \eqref{eq:SPDE} satisfies the large deviation principle on $L^2(0,T;V)\cap C([0,T];H)$ with rate function  $I\colon L^2(0,T;V)\cap C([0,T];H)\to [0,+\infty]$ defined by
\begin{equation*} 
I(z)=\frac{1}{2}\inf\Big\{\textstyle{\int_0^T}\|\psi(s)\|_U^2\dd s : \psi\in L^2(0,T;U), \, z=u^{\psi}\Big\},
\end{equation*}
where $\inf\varnothing\coloneqq +\infty$ and $u^\psi$ is the unique strong solution to \eqref{eq:skeleton}. 
\end{theorem} 

\begin{example}\label{ex:applications}
Two applications of Theorem \ref{th:LDP general} for equations with non-coercive coefficients have already been worked out in detail: 
 \begin{itemize}
  \item The 2D Brusselator model with transport noise: see Subsection \ref{ss:brusselator}. 
 \item The 3D primitive equations  with transport noise: see \cite[Th.\ 2.7]{AT25}. Additionally, the Stratonovich noise case \cite[Th.\ 2.9]{AT25} is  proved using Theorem \ref{th:LDP perturbed2} from Section \ref{sec:perturb}.  
 \end{itemize}
We emphasize that in the first-mentioned example, it is crucial to use a non-zero $\alpha_i$ for the  $\hat{F}$-term. 
 
Further possible applications of Theorem \ref{th:LDP general} and Theorem \ref{th:LDP perturbed2}   include the following: 
 \begin{itemize}  
     \item The Lotka--Volterra equations with transport noise in 1--4 dimensions, see \cite[Th.\ 4.5]{AVreac24}.  
     \item The 3D primitive equations with non-isothermal turbulent pressure and transport noise, as studied in \cite{AHHS25nonisothermal}. 
    \item More generally: for settings  with higher regularity (i.e.\ smoother spaces $V$ and $H$), coercivity of $(A,B)$ often gets lost, while well-posedness can sometimes still be proved and Theorem \ref{th:LDP general} and Theorem \ref{th:LDP perturbed2} are applicable. 
        At an abstract level, such settings are precisely considered in \cite[Th.\ 6.7, Rem.\ 6.8]{AV25survey}. 
        
        Concrete examples are the abstract fluid model with upgraded regularity from \cite[Th.\ 7.12]{AV25survey},  the stochastic 1D or 2D Cahn--Hilliard equation mentioned in \cite[\S7.2, Rem.\ 7.7]{AV25survey} and the reaction–diffusion systems found in \cite[(5.1), Prop.\ 5.1, Rem.\ 5.2]{boschhupkes25}.  
 \end{itemize}
 For the models above, one can verify that Assumption \ref{ass:critvarsettinglocal} is satisfied, using the given references. Additionally,  global well-posedness of the corresponding SPDE has already been obtained, thus Assumption \ref{ass:coer replace}\ref{it:1} is satisfied as well. 
To obtain the LDP, one only needs to check Assumption \ref{ass:coer replace}\ref{it:2}\ref{it:3} and then apply Theorem \ref{th:LDP general}. 
\end{example}

The following lemma is an immediate consequence of the blow-up criteria proved in \cite{AV22variational,BGV}. Recall Remark \ref{rem: SPDE loc well-posed}. For the notion of maximal solution, we refer to \cite[Def.\ 4.4]{AV22nonlinear1}.  

\begin{lemma}\label{lem:suff non-blow up}
  Let $(A,B)$  satisfy  Assumption  \ref{ass:critvarsettinglocal}, let  $x\in H$ and $\eps_0\in(0,1]$. Suppose that for every $\eps\in(0,\eps_0)$ and maximal solution $(u,\sigma)$ to \eqref{eq:SPDE}, we have  
  \begin{equation}\label{eq:suff non-blow up}
     \sup_{t\in[0,\sigma\wedge T)}\|u(t)\|_H+\textstyle{\int_0^{\sigma\wedge T}\|u\|_V^2\dd s}<\infty \; {\text{ a.s., \quad for all $T>0$.}}
  \end{equation}
  Then  Assumption \ref{ass:coer replace}\ref{it:1} is valid. 
\end{lemma}
\begin{proof}
 From \eqref{eq:suff non-blow up} it follows that 
 \[
 \P(\sigma<T)=\P(\sigma<T, \sup_{t\in[0,\sigma)}\|u(t)\|_H+\textstyle{\int_0^{\sigma}\|u\|_V^2\dd s}<\infty).
 \] 
 By the blow-up criterion of \cite[Th.\ 6.4]{BGV} (and Remark \ref{rem: SPDE loc well-posed}, $\eps\in(0,\eps_0)\subset[0,1]$), the probability on the right-hand side above is zero when $(u,\sigma)$ is the maximal solution. We conclude that $\sigma \geq T$ a.s., and taking the intersection over $T\in\N$, we conclude that $\sigma=\infty$ a.s., thus the maximal solution is globally defined and guarantees the well-posedness stated in Assumption \ref{ass:coer replace}\ref{it:1}. 
\end{proof} 
  
\subsection{The coercive case}\label{ss:mainresultcoercive}

Next, let us review the case of coercive coefficients. In fact, if $(A,B)$ is coercive, then most of the assumptions in Theorems \ref{th:LDP general} and \ref{th:LDP perturbed2} are satisfied automatically. Using these results, we thus obtain more direct LDP results.  Since the $\alpha$-(sub)criticality condition \eqref{eq: subcrit cond alpha} for $\hat{F}$ improves that of \cite{TV24}, we obtain several new applications, see Example \ref{ex:applications hat F}. 
 
\begin{corollary}[LDP -- coercive case]\label{cor:LDPcoercive}
  Let $(A,B)$ satisfy Assumption \ref{ass:critvarsettinglocal},  and suppose that $(A,B)$ is coercive in the sense of \eqref{eq:def coercive}. 
  
  Then for every $\eps\in[0,1]$, \eqref{eq:SPDE} has a unique strong solution $Y^\eps$, and the family $(Y^\eps)_{\eps\in(0,\frac12)}$ satisfies the large deviation principle on $L^2(0,T;V)\cap C([0,T];H)$ with the rate function mentioned in Theorem \ref{th:LDP general}. 
\end{corollary}
\begin{proof}
Well-posedness follows from  \cite[Th.\ 3.5]{AV22variational} applied with $(A,\sqrt{\eps}B)$, which is also coercive for all $\eps\in[0,1]$ (and see Remark \ref{rem: SPDE loc well-posed}). Here, we note that \cite[Th.\ 3.5]{AV22variational} extends to the case where $A$ contains the additional terms $\hat{F}$ and $\hat{f}$ (see \cite[Th.\ 7.1, Rem.\ 7.2]{BGV}). 
For the LDP, by Theorem \ref{th:LDP general}, it now suffices to verify Assumption \ref{ass:coer replace}\ref{it:2} and \ref{it:3}.  
Moreover, \ref{it:2} holds by the estimates   leading to \cite[(3.31)]{TV24}. Indeed, these estimates rely only on the coercivity \eqref{eq:def coercive} of $(A,B)$ and imply that, whenever $u\in\MR(0,T)$ is a strong solution to \eqref{eq:skeleton}:  
\begin{equation*} 
  \|u\|_{\MR(0,T)}\leq (2+ {\theta}^{-1})^{\frac{1}{2}}\big(\|x\|_H+\sqrt{2}\|\phi\|_{L^2(0,T)} \big)\exp[ MT+\tfrac{1}{2}\|\psi\|_{L^2(0,T;U)}^2],
\end{equation*}
where $\theta, M>0$ and $\phi\in L^2(0,T)$ are as in \eqref{eq:def coercive}.  
Lastly, \ref{it:3} follows from the proof of \cite[Lem.\ 4.11]{TV24}, noting that the latter again relies only on the coercivity \eqref{eq:def coercive} of $(A,B)$. 
\end{proof}

The above result for the coercive case (i.e.\ a variational setting) leads to interesting new applications  that were not covered yet in the  variational LDP literature. 

\begin{example}\label{ex:applications hat F}
If we have coercivity of $(A,B)$, and in Assumption \ref{ass:critvarsettinglocal} we have $\alpha_i=0$ for all $i$, then Corollary \ref{cor:LDPcoercive} coincides with \cite[Th.\ 2.6]{TV24}, recalling Remark \ref{rem:difference AB}. 
Yet, using non-zero $\alpha_i$,   Corollary \ref{cor:LDPcoercive} yields the LDP for new equations with critical nonlinearities for which the LDP was unavailable before. 
In particular, we immediately obtain the LDP for all examples given in \cite[\S8]{BGV} when we take Gaussian noise. These include: 
    \begin{itemize}
 \item 2D Allen--Cahn equation with transport noise, analytically weak setting: Subsection \ref{ss:allen}
 \item a generalized 2D Burger's equation with transport noise, see \cite[Ex.\ 8.3]{BGV}
 \item stochastic Kuramoto–Sivashinsky (fourth order) equation, see \cite[\S8.3]{BGV}
 \end{itemize} 
 \end{example}

\subsection{Vanishing perturbations and Stratonovich noise}\label{ss:perturbed main results}

In addition to the LDP results of Subsections \ref{ss:mainresultnoncoercive} and \ref{ss:mainresultcoercive}, we establish analogous results for a perturbed version of the stochastic evolution equation \eqref{eq:SPDE}. The precise assumptions, statements, and proofs are 
deferred to Section~\ref{sec:perturb}. In this subsection, we give a brief preview, explain the motivation for considering such perturbations, and summarize the main conclusions.

We consider the following perturbed  stochastic evolution equation: 
\begin{equation}\label{eq:SPDE perturbed}
   \begin{cases}
  &\dd \tilde{Y}^\eps(t)=-A(t,\tilde{Y}^\eps(t))\dd t+\eta(\eps)\tilde{A}(t,\tilde{Y}^\eps(t))\dd t+\sqrt{\eps}B(t,\tilde{Y}^\eps(t))\dd W(t), \quad t\in[0,T], \\
  &\tilde{Y}^\eps(0)=x.
\end{cases}
\end{equation}  
Compared with \eqref{eq:SPDE}, this equation contains an additional drift term $\eta(\eps)\tilde{A}$,  which is assumed to vanish, in the sense that $\lim_{\eps\downarrow0}\eta(\eps)=0$.

A primary motivation for studying  \eqref{eq:SPDE perturbed} comes from equations driven by Stratonovich noise. Indeed, if the $\sqrt{\eps}B\dd W$-term in \eqref{eq:SPDE} is interpreted in the
Stratonovich sense, conversion to It\^o form often produces an equation of the form \eqref{eq:SPDE perturbed}, with $\eta(\eps)=\eps$ and $\tilde A$ representing the
It\^o--Stratonovich correction.  

Our LDP results for \eqref{eq:SPDE perturbed} consist of Theorem \ref{th:LDP perturbed2} for non-coercive equations and Corollaries  \ref{cor:LDPcoerciveperturbed} and \ref{cor:LDPcoerciveperturbed1} for coercive equations. 
Crucially, under their respective assumptions, these results show that the LDP remains stable under vanishing drift perturbations: provided $\eta(\eps)={o}(1)$, the LDP holds regardless of the decay rate of $\eta(\eps)$ relative to the noise intensity  $\sqrt{\eps}$.  

The rate function is determined by the same skeleton equation \eqref{eq:skeleton} as in the unperturbed case. In particular, the vanishing perturbation does not appear in the skeleton equation.
The associated controlled stochastic equation, however, becomes
\begin{equation}\label{eq:SPDE tilted perturb}
   \begin{cases}
  &\ddd \tilde{X}^\eps(t)=-A(t,\tilde{X}^\eps(t))\dd t+\eta(\eps)\tilde{A}(t,\tilde{X}^\eps(t))\dd t\\
  &\qquad\qquad\quad\qquad+B(t,\tilde{X}^\eps(t))\Psi^\eps(t)\dd t+\sqrt{\eps}B(t,\tilde{X}^\eps(t))\dd W(t), \quad t\in[0,T], \\
  &\tilde{X}^\eps(0)=x.
\end{cases}
\end{equation}   
Thus, the deterministic a priori estimate for the skeleton equation
remains unchanged, whereas the well-posedness and uniform
boundedness-in-probability conditions must be verified for the
perturbed equations \eqref{eq:SPDE perturbed} and
\eqref{eq:SPDE tilted perturb}.

In the semilinear coercive setting, Corollary \ref{cor:LDPcoerciveperturbed} shows that a bounded linear drift perturbation preserves the relevant coercivity properties for all sufficiently small $\eps$. Corollary \ref{cor:LDPcoerciveperturbed1} further exploits the small-noise regime and allows the coercivity assumption on the linear pair $(A_0,B_0)$ to be replaced by coercivity of $(A_0,0)$, together with a weaker coercivity condition on the nonlinear terms. 
These results allow LDPs for several equations with It\^o transport noise to be extended to the corresponding equations with Stratonovich noise,  see Remark \ref{ex:stratonovich coercive}. A detailed application is given in Subsection \ref{ss:NS}, where we establish an LDP for the 2D  Navier--Stokes equation with Stratonovich transport noise on bounded and unbounded spatial domains. 
In the non-coercive setting, a major application of Theorem \ref{th:LDP perturbed2} is the LDP for the   3D primitive equations with Stratonovich transport noise, see Remark \ref{rem:appl strat}.

\section{The skeleton equation with flexible nonlinearities}

In this section we prove local well-posedness of the skeleton equation \eqref{eq:skeleton} and a blow-up criterion. We extend the theory of  \cite[\S3]{TV24} under Assumption \ref{ass:critvarsettinglocal}, which includes more general nonlinearities and inhomogeneities (see Remark \ref{rem:difference AB}). 

\subsection{Linearized problem}
Throughout this subsection, we assume that $m\in\N$ and  $\alpha_i\in[0,\frac12]$ for $i\in\{1,\ldots,m\}$ and study a  linearization of the skeleton equation. We discard the nonlinearities $\hat{F}$ and $G$ in the coefficients $A$ and $B$ (see \eqref{eq:defAB}), and we consider for fixed $w\in L^\infty(0,T;H)$:
\begin{equation}\label{eq: linear skeleton}
\begin{cases}
    &u'(t)+A_0(t,w(t))u(t)-B_0(t,w(t))u(t)\psi(t)=\sum_{i=1}^m\bar{f}_i(t)+\bar{g}(t)\psi(t), \\
    &u(0)=x.
\end{cases}
\end{equation}
Here, we assume that $A_0$ and $B_0$ are as in Assumption \ref{ass:critvarsettinglocal}, and $\bar{f}_{i}\in L^{2/(1+2\alpha_i)}(0,T;V_{\alpha_i})$ and $\bar{g}\in L^2(0,T;\UH)$.
The goal is to prove well-posedness and energy estimates for  \eqref{eq: linear skeleton}.   
Later, it will be used to treat the terms $\hat{F}(u)$ and $\hat{f}$.  

We define the following sum space:
\[
S_T=L^{2}(0,T;V^*)+L^{1}(0,T;H), 
\]
where we note that the spaces appearing in $S_T$ form a Banach interpolation couple with ambient space $L^1(0,T;V^*)$. Then, $S_T$ is a Banach space \cite[Prop.\ C.1.3]{HNVWvolume1} and $S_T\into L^1(0,T;V^*)$. 

In \cite[\S3.1]{TV24}, inhomogeneities in $S_T$ were studied. 
Now, we note that   the Stein--Weiss theorem for complex interpolation \cite[Th.\ 14.3.1]{HNVWvolume3} yields for all $\alpha\in[0,\frac12]$:
\begin{equation}\label{eq: into S}
  L^{2/(1+2\alpha)}(0,T;V_{\alpha})= [L^{2}(0,T;V^*),L^{1}(0,T;H)]_{2\alpha}\into S_T,  
\end{equation}
where the first identification is isometric, and  the embedding into $S_T$ is contractive, by exactness of the complex interpolation method.   
In particular, \eqref{eq: into S} and the theory in \cite[\S3.1]{TV24} allow us to consider sums of inhomogeneities $\bar{f}_i\in L^{2/(1+2\alpha_i)}(0,T;V_{\alpha_i})\subset S_T$. 
 
The following result was proved in \cite[Th.\ 3.4]{TV24}. See \eqref{eq: def MR space} for the definition of $\MR(0,T)$.

\begin{theorem}\label{th: well-posed with M in L^1}
Let $\bar{A}\colon [0,T]\to\mathscr{L}(V,V^*)$ and suppose that
for all $u\in\MR(0,T)$: $  \bar{A}(\cdot)u(\cdot)\in S_T$ and $\|\bar{A}(\cdot)u(\cdot)\|_{S_T}\leq \alpha \|u\|_{\MR(0,T)}$ 
  for some constant $\alpha>0$ independent of $u$.
  Suppose that for some $\theta>0$ and $M\in L^1(0,T)$:  $\<\bar{A}(t)v,v\?\geq \theta \|v\|_V^2-M(t)\|v\|_H^2$ for all $v\in V$.  
  
  Then for any $h\in S_T$, there exists a unique strong solution  $u\in \MR(0,T)$ to
  \begin{equation*}
    \begin{cases}
    u'(t)+\bar{A}(t)u(t)=h(t), \quad t\in[0,T],\\
    u(0)=x
  \end{cases}
  \end{equation*}
 and the following estimate holds:
 \begin{equation*}
\|u\|_{\MR(0,T)}\leq C_{\theta}\exp(2\|M\|_{L^1(0,T)})\big(\|h\|_{S_T}+\|x\|_H\big).
\end{equation*}
\end{theorem}

A consequence of Theorem \ref{th: well-posed with M in L^1} is the following result for \eqref{eq: linear skeleton},  which extends  \cite[Cor.\ 3.5]{TV24}. 

\begin{corollary}\label{cor: skeleton MR linearized crit var setting}
Let $A_0$ and $B_0$ satisfy the conditions concerning $A_0,B_0$ in Assumption \ref{ass:critvarsettinglocal} and let $\psi\in L^2(0,T;U)$.
Let $T>0$ and let $w\in L^\infty(0,T;H)$.
Then, if $\bar{f}_i \in L^{2/(1+2\alpha_i)}(0,T;V_{\alpha_i})$ for $i\in\{1,\ldots,m\}$ and $\bar{g}\in L^2(0,T;\UH)$, there exists a unique strong solution  $u\in \MR(0,T)$ to
\begin{equation}\label{eq:skeletonlinear}
\begin{cases}
    &u'(t)+A_0(t,w(t))u(t)-B_0(t,w(t))u(t)\psi(t)=\sum_{i=1}^m\bar{f}_i(t)+\bar{g}(t)\psi(t), \\
    &u(0)=x,
\end{cases}
\end{equation}
Moreover, for any $\tilde{T}\in[0,T]$ there exists a constant $K_{\tilde{T}}>0$ such that
 \begin{equation}\label{def: maxreg}
 \|u\|_{\MR(0,\tilde{T})}\leq K_{\tilde{T}}\Big(\|x\|_H+\textstyle{\sum_{i=1}^{m}}\|\bar{f}_i\|_{L^{2/(1+2\alpha_i)}(0,\tilde{T};V_{\alpha_i})} +\|\bar{g}\|_{L^2(0,\tilde{T};\UH)}\Big),
 \end{equation} 
and $K_{\tilde{T}}$ is non-decreasing in $\tilde{T}$ and depends further only on $T$, $\|w\|_{L^\infty(0,T;H)}$ and $\|\psi\|_{L^2(0,\tilde{T};U)}$.
\end{corollary}
\begin{proof}
By \cite[Cor.\ 3.5]{TV24},  
 $\bar{A}\colon [0,T]\to\mathscr{L}(V,V^*)$ defined  by $\bar{A}(t)v\coloneqq  A_0(t,w(t))v-B_0(t,w(t))v\psi(t)$ satisfies all conditions of Theorem \ref{th: well-posed with M in L^1}. Moreover, thanks to \eqref{eq: into S} and the Cauchy--Schwarz inequality, $h\ceqq \sum_{i=1}^{m}\bar{f}_i+\bar{g}\psi$ satisfies $h\in S_T$. 
The claim thus follows from Theorem \ref{th: well-posed with M in L^1},  the observation that
\[
\|h\|_{S_{\tilde{T}}}\leq \textstyle{\sum_{i=1}^{m}}\|\bar{f}_i\|_{L^{2/(1+2\alpha_i)}(0,\tilde{T};V_{\alpha_i})} +\|\bar{g}\|_{L^2(0,\tilde{T};\UH)}\|\psi\|_{L^2(0,\tilde{T};U)},
\] 
and by putting $K_{\tilde{T}}\coloneqq C_{\theta_{n,T}}\exp(2\|M_{n,T}+\frac{1}{2}\|\psi\|_U^2\|_{L^1(0,\tilde{T})})(1\vee\|\psi\|_{L^2(0,\tilde{T};U)})$ as in the proof of \cite[Cor.\ 3.5]{TV24}. 
\end{proof}

\subsection{Nonlinear problem} 

We turn to the original nonlinear skeleton equation \eqref{eq:skeleton}. 
For $\psi\in L^2(0,T;U)$,  initial data $v_0\in H$ and a time interval  $[0,\tilde{T}]$, we rewrite the latter equation as
\begin{align}\label{eq: local well-posed}
 & \begin{cases}
    &u'+\breve{A}(\cdot,u)u=\hat{F}(\cdot,u)+\hat{f}+(G(\cdot,u)+g)\psi\qquad \text{ on } [0,\tilde{T}], \\
    &u(0)=v_0,
  \end{cases}\\ &\hspace{-.9cm}\text{ where } \breve{A}(t,u)v\ceqq A_0(t,u)v-B_0(t,u)v\psi. \notag
\end{align}

We will prove the following local well-posedness result analogous to \cite[Th.\ 3.7]{TV24},  now with the more general terms $\hat{F}(u)$ and $\hat{f}$ from Assumption \ref{ass:critvarsettinglocal}. 
As in \cite{TV24}, we define 
\[
\MR(0,T)=L^2(0,T;V)\cap C([0,T];H).
\]

\begin{theorem}[Local well-posedness of the skeleton equation]\label{th: local well posedness skeleton}
Suppose that $(A,B)$ satisfies Assumption \ref{ass:critvarsettinglocal}. 
Let $u_0\in H$ be fixed. Then there exist $\tilde{T},\eps>0$ such that for each $v_0\in B_H(u_0,\eps)$,
there exists a unique strong solution $u_{v_0}\in \MR(0,\tilde{T})$ to
\eqref{eq: local well-posed}. 
Moreover, there exists a constant $C>0$ such that for all $v_0,w_0\in B_H(u_0,\eps)$:
\begin{equation}\label{eq: ct dependence local well-posed}
  \|u_{v_0}-u_{w_0}\|_{\MR(0,\tilde{T})}\leq C\|v_0-w_0\|_H.
\end{equation}
\end{theorem}

To prove Theorem \ref{th: local well posedness skeleton}, we need suitable   estimates for $\hat{F}$.  
The following interpolation estimate is crucial for this end. It follows from \cite[Lem.\ 2.5, (5.48)]{BGV}.  

\begin{lemma}[Critical interpolation estimate]\label{lem:crit interpol}
Let  $\hat{\beta}\in (1/2,1]$ and $T>0$.    
Then for all $u\in \MR(0,T)$, it holds that 
\begin{equation*} 
  \|u\|_{L^{{2}/({2\hat{\beta}-1})}(0,T;V_{\hat{\beta}})}\leq   \|u\|_{L^{\infty}(0,T;H)}^{2-2\hat{\beta}}\|u\|_{L^{2}(0,T;V)}^{2\hat{\beta}-1}\leq \|u\|_{\MR(0,T)}.
\end{equation*}
\end{lemma} 

The next lemma covers estimates similar to \cite[Lem.\ 3.8 (iv)(v)]{TV24} for $\hat{F}$. Note that \ref{it:emb4} readily shows that $\hat{F}(u)\in L^{2/(1+2\alpha)}(0,T;V_{\alpha})$ whenever $u\in\MR(0,T)$. 

\begin{lemma}\label{lem: 3.8 extra F}
 Let   $\hat{F}=\sum_{i=1}^m \hat{F}_i$ and $m,\alpha_i,\hat{\beta}_i,\hat{\rho}_i$ be as in Assumption \ref{ass:critvarsettinglocal}\ref{it:hat F}.  
  Let $n\in\R_+$, $T>0$ and  $\sigma>0$. 
  There exist  constants $M_T,\hat{C}_{n,T},\tilde{C}_{n,T},C_{n,T,\sigma}$ non-decreasing in $T$ and $n$, and constants $\epsilon_i>0$,  such that
  \begin{enumerate}[label=\emph{(\roman*)},ref=\textup{(\roman*)}]
  \item \label{it:emb4} for all $u\in\MR(0,T)$ with $\|u\|_{C([0,T];H)}\leq n$ and $i\in\{1,\ldots,m\}$:
  \begin{equation*}\label{eq: estimate hat F}
  \|\hat{F}_i(u)\|_{L^{2/(1+2\alpha_i)}(0,T;V_{\alpha_i})}  
   \leq \hat{C}_{n,T}(1+\|u\|_{L^2(0,T;V)}^{1+2\alpha_i}), 
  \end{equation*}
  \item \label{it:emb5} for all $u,v\in\MR(0,T)$ with $\|u\|_{C([0,T];H)},\|v\|_{C([0,T];H)}\leq n$ and $i\in\{1,\ldots,m\}$:
  \begin{align*}\label{eq: estimate hat F difference} 
  &\|\hat{F}_i(u)-\hat{F}_i(v)\|_{L^{2/(1+2\alpha_i)}(0,T;V_{\alpha_i})} \\  
  &\qquad\leq \tilde{C}_{n,T} \Big(T^{\epsilon_i}+\|u\|_{L^{{2}/({2\hat{\beta}_i-1})}(0,T;V_{\hat{\beta}_i})}^{\hat{\rho}_i}+\|v\|_{L^{{2}/({2\hat{\beta}_i-1})}(0,T;V_{\hat{\beta}_i})}^{ \hat{\rho}_i}\Big)\|u-v\|_{L^{{2}/({2\hat{\beta}_i-1})}(0,T;V_{\hat{\beta}_i})}, 
  \end{align*}
  \item \label{it:emb6} for all $u,v\in\MR(0,T)$ with $\|u\|_{C([0,T];H)},\|v\|_{C([0,T];H)}\leq n$:
  \begin{align*} 
  \Big|\int_0^T\<\hat{F}(u)-\hat{F}(v),u-v\?\dd t   \Big|
  &\leq \hat{C}_{\sigma,n,T} \int_0^T  (1+\|u\|_V^2+ \|v\|_V^2)\|u-v\|_H^2\dd t+\sigma\|u-v\|_{L^2(0,T;V)}^2. 
  \end{align*} 
  \end{enumerate}
\end{lemma}

\begin{proof}
Note that the constants $C_{n,T}$ in Assumption \ref{ass:critvarsettinglocal}\ref{it:hat F} can be taken non-decreasing in $T$ and $n$ without loss of generality. To omit subscripts, we provide the proof for $m=1$ and write $\alpha\ceqq \alpha_1$, $\hat{\beta}\ceqq\hat{\beta}_1$ and $\hat{\rho}\ceqq\hat{\rho}_1$. Applying the proof below for each $i$ and taking maxima or sums of the corresponding constants yields the case $m>1$. 
 
\ref{it:emb4}: Using subsequently Assumption \ref{ass:critvarsettinglocal}\ref{it:hat F}, H\"older's inequality, Lemma \ref{lem:crit interpol} and Young's inequality, we have
\begin{align*}
  \|\hat{F}(u)\|_{L^{2/(1+2\alpha)}(0,T;V_{\alpha})} 
  &\leq  C_{n,T} \big\|1+\|u\|_{\hat{\beta}}^{1+\hat{\rho}} \big\|_{L^{2/(1+2\alpha)}(0,T)} \\
   &=  {C}_{n,T}\big(T^{(1+2\alpha)/2}+\|u\|_{L^{2(1+\hat{\rho})/(1+2\alpha)}(0,T;V_{\hat{\beta}})}^{1+\hat{\rho}} \big)\\
   &\leq  \tilde{C}_{n,T}\big(1+\|u\|_{L^{2/(2\hat{\beta}-1)}(0,T;V_{\hat{\beta}})}^{1+\hat{\rho}} \big)\\
&\leq  \tilde{C}_{n,T}\big(1+\|u\|_{C([0,T];H)}^{(1+\hat{\rho})(2-2\hat{\beta})} \|u\|_{L^2(0,T;V)}^{(1+\hat{\rho})(2\hat{\beta}-1)} \big)\\
&\leq  \hat{C}_{n,T}\big(1+  \|u\|_{L^2(0,T;V)}^{1+2\alpha} \big),   
\end{align*}
where all constants are non-decreasing in $T$ and $n$. 
 
\ref{it:emb5}: 
Let $q=2(\hat{\rho}+1)/(1+2\alpha)$, $q' =2(\hat{\rho}+1)/((1+2\alpha)\hat{\rho})$ and note that by the (sub)criticality condition, we have $q=\hat{\rho}q'\leq {2}/({2\hat{\beta}-1})$. Combined with Assumption \ref{ass:critvarsettinglocal}\ref{it:hat F}, H\"older's inequality, Lemma \ref{lem:crit interpol} and Young's inequality (using that $2\hat{\beta}-1\in(0,1]$), this gives
\begin{align*}
\|&\hat{F}(u)-\hat{F}(v)\|_{L^{2/(1+2\alpha)}(0,T;V_{\alpha})}\leq C_{n,T} \big\|\big(1+\|u\|_{\hat{\beta}}^{\hat{\rho}}+\|v\|_{\hat{\beta}}^{ \hat{\rho}}\big)\|u-v\|_{\hat{\beta}} \big\|_{L^{2/(1+2\alpha)}(0,T)} \\
&\leq {C}_{n,T} \Big(T^{1/q'}+\|u\|_{L^{\hat{\rho}q'}(0,T;V_{\hat{\beta}})}^{\hat{\rho}}+\|v\|_{L^{\hat{\rho}q'}(0,T;V_{\hat{\beta}})}^{ \hat{\rho}}\Big)\|u-v\|_{L^{q}(0,T;V_{\hat{\beta}})}\\
&\leq \tilde{C}_{n,T} \Big(T^{1/q'}+\|u\|_{L^{{2}/({2\hat{\beta}-1})}(0,T;V_{\hat{\beta}})}^{\hat{\rho}}+\|v\|_{L^{{2}/({2\hat{\beta}-1})}(0,T;V_{\hat{\beta}})}^{ \hat{\rho}}\Big)\|u-v\|_{L^{{2}/({2\hat{\beta}-1})}(0,T;V_{\hat{\beta}})} 
\end{align*} 
where $C_{n,T}$ and $\tilde{C}_{n,T}$ can be chosen non-decreasing in $T$ and $n$. 

\ref{it:emb6}:  
We use the following complex interpolation estimates for $x\in V$, $y\in V_\alpha$ and $\theta\in[1/2,1]$:  
\begin{equation}\label{eq:interpol est alpha pair}
\begin{split}
&\|x\|_{\theta}\leq \|x\|_H^{2-2\theta}\|x\|_V^{2\theta-1}, \\   
&\<y,x\? \leq   \|y\|_{\alpha}\|x\|_{H}^{2\alpha}\|x\|_V^{1-2\alpha}, 
\end{split}
\end{equation}
where we use that $V_\theta=[H,V]_{2\theta-1}$ for $\theta\in[1/2,1]$. 
The   estimates of \eqref{eq:interpol est alpha pair}, Assumption \ref{ass:critvarsettinglocal}\ref{it:hat F} and Young's inequality  with $q= (1+\alpha-\hat{\beta})^{-1}$, $q'= (-\alpha+\hat{\beta})^{-1}$ yield
  \begin{align*} 
   \Big|\int_0^T\<\hat{F}(u)-\hat{F}(v),u-v\?\dd t    \Big|
  &\leq   \int_0^T \|\hat{F}(u)-\hat{F}(v)\|_\alpha\|u-v\|_H^{2\alpha}\|u-v\|_V^{1-2\alpha} \dd t \\
  &\leq  {C}_{n,T} \int_0^T  (1+\|u\|_{\hat{\beta}}^{\hat{\rho}}+ \|v\|_{\hat{\beta}}^{\hat{\rho}})\|u-v\|_H^{2\alpha+2-2\hat{\beta}}\|u-v\|_V^{-2\alpha+2\hat{\beta}} \dd t \\
  &\leq  \tilde{C}_{\sigma,n,T} \int_0^T  (1+\|u\|_{\hat{\beta}}^{\hat{\rho}q}+ \|v\|_{\hat{\beta}}^{\hat{\rho}q}) \|u-v\|_H^{2} \dd t +\sigma\|u-v\|_{L^2(0,T;V)}^2 \\
   &\leq \hat{C}_{\sigma,n,T} \int_0^T  (1+\|u\|_V^2+ \|v\|_V^2)\|u-v\|_H^2\dd t+\sigma\|u-v\|_{L^2(0,T;V)}^2. 
  \end{align*} 
 In the last line we  combine \eqref{eq:interpol est alpha pair} with the fact that $2\hat{\beta}-1\in(0,1]$ and 
$\hat{\rho}q\leq  {2}/({2\hat{\beta}-1})$ by virtue of  \eqref{eq: subcrit cond alpha}, and we use that    $\|u\|_{C([0,T];H)}\leq n$. Also, we note that $q\in(1,\infty)$, which is ensured by  \eqref{eq: subcrit cond alpha} and the fact that $\hat{\rho}>0$, $\alpha\geq 0$ and $\hat{\beta}\leq 1$. 
\end{proof}

\begin{proof}[Proof of Theorem \ref{th: local well posedness skeleton}] 
We follow \cite[\S3.2]{TV24}. Note that $\tilde{A}$ in the latter reference is the same as our $\breve{A}$ in \eqref{eq: local well-posed}. 
  For $v_0\in H$,  let $z_{v_0}\in \MR(0,T)$ denote the strong solution to  
\begin{equation}\label{eq:def z_v_0}
  \begin{cases}
    &z'+\breve{A}(u_0)z=0\quad \text{ on } [0,T], \\
    &z(0)=v_0, 
  \end{cases}
\end{equation}
which exists uniquely due to Corollary \ref{cor: skeleton MR linearized crit var setting}.  
Fix some $T_1\in(0,T]$ such that
\begin{equation*} 
  \|z_{u_0}-u_0\|_{C([0,T_1];H)}\leq 1/3. 
\end{equation*} 
Such $T_1$ exists since $z_{u_0}(0)=u_0$ and $z_{u_0}\in \MR(0,T)$. 
For $v_0\in H$, $r>0$ and $\tilde{T}\in[0,T]$, we use the following complete subspace of $\MR(0,\tilde{T})$: 
\begin{equation*} 
Z_{r,\tilde{T}}(v_0)\coloneqq \{v\in \MR(0,\tilde{T}):v(0)=v_0, \|v-z_{u_0}\|_{\MR(0,\tilde{T})}\leq r\}.
\end{equation*}
By \cite[Lem.\ 3.10]{TV24}, we can fix $\eps_1,r_1>0$ such that for every  $\eps\in(0,\eps_1]$, $r\in(0,r_1]$, $\tilde{T}\in(0,T_1]$: 
\begin{equation}\label{eq:eps1 r1 property}
   v_0\in B_H(u_0,\eps) \text{ and }v\in Z_{r,\tilde{T}}(v_0)\implies \|v-u_0\|_{C([0,\tilde{T}];H)}\leq 1.  
\end{equation}
Now we prove the following: 
\begin{enumerate}[label=\textit{(\roman*)},ref=\textit{(\roman*)}]
  \item \label{it: estimate tilde f v}  For any  $\tilde{T}\in(0,T_1]$, $\eps\in(0,\eps_1]$, $r\in(0,r_1]$, $v_0\in B_H(u_0,\eps)$, $v\in Z_{r,\tilde{T}}(v_0)$:
\begin{align*}
&\textstyle{\sum_{i=1}^m}\| \hat{F}_i(v)+\hat{f}_i\|_{L^{2/(1+2\alpha_i)}(0,\tilde{T};V_{\alpha_i})}\leq \alpha_{T_1}(\tilde{T})+\beta_{T_1}(\tilde{T},r)r,  
\end{align*}
  where $\alpha_{T_1}(\tilde{T}),\beta_{T_1}(\tilde{T},r)\downarrow 0$ as $\tilde{T},r\downarrow0$, and $\alpha_{T_1}(\tilde{T})$ and $\beta_{T_1}(\tilde{T},r)$ are independent of $v_0$ and $v$.
\item\label{it: estimate tilde f v-w} For any $\tilde{T}\in(0,T_1]$, $\eps\in(0,\eps_1]$, $r\in(0,r_1]$, $v_0, w_0\in B_H(u_0,\eps)$, $v\in Z_{r,\tilde{T}}(v_0)$, $w\in Z_{r,\tilde{T}}(w_0)$ and $\sigma>0$:
\begin{align*} 
&\textstyle{\sum_{i=1}^m}\|\hat{F}_i(v)-\hat{F}_i(w)\|_{L^{2/(1+2\alpha_i)}(0,\tilde{T};V_{\alpha_i})} \leq \gamma_{T_1}(\tilde{T},r) \|v-w\|_{\MR(0,\tilde{T})}, 
\end{align*}
where $\gamma_{T_1}(\tilde{T},r)\downarrow 0$ as $\tilde{T},r\downarrow0$, and  $\gamma_{T_1}(\tilde{T},r)$ is independent of $v_0,w_0,v$ and $w$.
\end{enumerate}

Proof of \ref{it: estimate tilde f v}: Fix $n\ceqq 1+\|u_0\|_H$. 
Note that $\|v\|_{C([0,\tilde{T}];H)}\leq \|v-{u_0}\|_{\MR(0,\tilde{T})}+\|u_0\|_H\leq  n$ by \eqref{eq:eps1 r1 property}. 
By Lemmas \ref{lem: 3.8 extra F} and \ref{lem:crit interpol}, and using that $\tilde{C}_{n,\tilde{T}}$ from Lemma \ref{lem: 3.8 extra F} is non-decreasing in $\tilde{T}$,  
\begin{align*}
 & \| \hat{F}_i(v)+\hat{f}_i\|_{L^{2/(1+2\alpha_i)}(0,\tilde{T};V_{\alpha_i})}\\
 &\leq \|\hat{F}_i(v)-\hat{F}_i(z_{u_0})\|_{L^{2/(1+2\alpha_i)}(0,\tilde{T};V_{\alpha_i})}+\|\hat{F}_i(z_{u_0}) \|_{L^{2/(1+2\alpha_i)}(0,\tilde{T};V_{\alpha_i})}+\|\hat{f}_i\|_{L^{2/(1+2\alpha_i)}(0,\tilde{T};V_{\alpha_i})}\\ 
&\leq \tilde{C}_{n,T_1} \Big(\tilde{T}^{\epsilon_i}+\|v\|_{L^{{2}/({2\hat{\beta}_i-1})}(0,\tilde{T};V_{\hat{\beta}_i})}^{\hat{\rho}_i}+\|z_{u_0}\|_{L^{{2}/({2\hat{\beta}_i-1})}(0,\tilde{T};V_{\hat{\beta}_i})}^{ \hat{\rho}_i}\Big)\|v-z_{u_0}\|_{ L^{{2}/({2\hat{\beta}_i-1})}(0,\tilde{T};V_{\hat{\beta}_i})}+\phi_i(\tilde{T})\\
&\leq \tilde{C}_{n,T_1} \Big(\tilde{T}^{\epsilon_i}+\|v-z_{u_0}\|_{\MR(0,\tilde{T})}^{\hat{\rho}_i}+2\|z_{u_0}\|_{L^{{2}/({2\hat{\beta}_i-1})}(0,\tilde{T};V_{\hat{\beta}_i})}^{ \hat{\rho}_i}\Big)\|v-z_{u_0}\|_{\MR(0,\tilde{T})}+\phi_i(\tilde{T})\\ 
  &\leq \tilde{C}_{n,T_1} \Big(\tilde{T}^{\epsilon_i}+ r^{\hat{\rho}_i} +\phi_i(\tilde{T})\Big) r+\phi_i(\tilde{T})
\end{align*} 
with
$
\phi_i(\tilde{T})\ceqq \| \hat{F}_i(z_{u_0}) \|_{L^{2/(1+2\alpha_i)}(0,\tilde{T};V_{\alpha_i})}+\|\hat{f}_i\|_{L^{2/(1+2\alpha_i)}(0,\tilde{T};V_{\alpha_i})}+2\|z_{u_0}\|_{L^{{2}/({2\hat{\beta}_i-1})}(0,\tilde{T};V_{\hat{\beta}_i})}^{ \hat{\rho}_i}.
$ 
Recalling that $\epsilon_i,\hat{\rho}_i>0$ and summing over $i$, the estimate is already of the claimed form. 

Proof of \ref{it: estimate tilde f v-w}: Again we put $n\ceqq 1+\|u_0\|_H$. Similar to the above, by Lemmas \ref{lem: 3.8 extra F} and \ref{lem:crit interpol}:
\begin{align*} 
&\|\hat{F}_i(v)-\hat{F}_i(w)\|_{L^{2/(1+2\alpha_i)}(0,\tilde{T};V_{\alpha_i})} \\
&\leq \tilde{C}_{n,T_1} \Big(\tilde{T}^{\epsilon_i}+\|v-z_{u_0}\|_{\MR(0,\tilde{T})}^{ \hat{\rho}_i}+\|w-z_{u_0}\|_{\MR(0,\tilde{T})}^{ \hat{\rho}_i}+2\|z_{u_0}\|_{L^{{2}/({2\hat{\beta}_i-1})}(0,\tilde{T};V_{\hat{\beta}})}^{\hat{\rho}_i}\Big)\|v-w\|_{\MR(0,\tilde{T})}\\
&\leq \tilde{C}_{n,T_1} \Big(\tilde{T}^{\epsilon_i}+2r^{ \hat{\rho}_i}+2\|z_{u_0}\|_{L^{{2}/({2\hat{\beta}_i-1})}(0,\tilde{T};V_{\hat{\beta}_i})}^{\hat{\rho}_i}\Big)\|v-w\|_{\MR(0,\tilde{T})},
\end{align*}
which is again of the claimed form after summing over $i$. 

Having proved the tools above, we  can now continue the  proof of  
Theorem \ref{th: local well posedness skeleton} in line with the proof of \cite[Th.\ 3.7]{TV24}.  We define 
\[
\breve{F}(v)=\hat{F}(v)+\hat{f}+(G(v)+g)\psi 
\] 
and note that $\breve{F}$ corresponds to $\tilde{F}$ from \cite[\S3.2]{TV24} if we replace $F$ and $f$ in the latter by $\hat{F}$ and $\hat{f}$. 
Then we define $\Psi_{v_0}\colon \MR(0,\tilde{T})\to \MR(0,\tilde{T})$   by $\Psi_{v_0}(v)\coloneqq u$, where $u$ is the strong solution to
\begin{equation}\label{eq:contraction map def}
  \begin{cases}
    &u'+\breve{A}(u_0)u=(\breve{A}(u_0)-\breve{A}(v))v+\breve{F}(v) \quad \text{on } [0,\tilde{T}],\\
    &u(0)=v_0.
  \end{cases}
\end{equation}
Such $u$ exists uniquely thanks to Corollary \ref{cor: skeleton MR linearized crit var setting}, \cite[Lem.\ 3.8]{TV24} and Lemma \ref{lem: 3.8 extra F}. 
Now, for existence of a   strong solution to \eqref{eq: local well-posed}, it suffices to show that $\Psi_{v_0}$ is strictly contractive.

To do so, one can copy the proof of \cite[Th.\ 3.7]{TV24}, noting that compared to \cite[(3.12)]{TV24}, the right-hand side in \eqref{eq:contraction map def} contains the terms  $\hat{F}(v)$ and $\hat{f}$ (instead of $F$ and $f$), and \ref{it: estimate tilde f v} and \ref{it: estimate tilde f v-w} above provide exactly the required estimates to handle these. Thus we obtain existence of solutions and the claimed estimate, for all small enough $r,\tilde{T}$ and $\eps$.  

For uniqueness amongst solutions in $\MR(0,\tilde{T})$, one can copy the last part of the proof of \cite[Th.\ 3.7]{TV24}, replacing $\tilde{F}(v)$ by $\breve{F}(v)$. 
\end{proof}

Having established local well-posedness for \eqref{eq:skeleton}, we will next provide criteria that ensure global well-posedness. To this end, we use the notion of \emph{maximal solution}.

\begin{definition}\label{def: maximal solution}
For $T\in(0,\infty]$, let 
$\MR_{\mathrm{loc}}(0,T)\coloneqq \{u\colon [0,T)\to H: u|_{[0,\tilde{T}]}\in\MR(0,\tilde{T}) \text{ for all } \tilde{T}\in[0,T)\}.$ 
A \emph{maximal solution} to \eqref{eq:skeleton} is a pair $(u_*,T_*)\in \MR_{\mathrm{loc}}(0,T_*) \times (0,\infty]$  such that
\begin{enumerate}[label=(\roman*)]
  \item for all $T\in(0,T_*)$, $u_*|_{[0,T]}$ is a strong solution to \eqref{eq:skeleton},
  \item for any $T>0$ and for any strong solution $u\in\MR(0,T)$ to \eqref{eq:skeleton} it holds that $T\leq T_*$ and $u=u_*$ on $[0,T]$.
\end{enumerate}
\end{definition}

The next statement about a maximal solution and blow-up criterion for the skeleton equation follows from the proof of  \cite[Prop.\ 3.15]{TV24}, with  Theorem \ref{th: local well posedness skeleton} replacing  \cite[Th.\ 3.7]{TV24}.  
\begin{proposition}[Blow-up criterion]\label{prop: blow up criterion}
  Let   $x\in H$ and let $\psi\in L^2_{\mathrm{loc}}(\R_+;U)$. Suppose that $(A,B)$ satisfies Assumption  \ref{ass:critvarsettinglocal}. Then  equation \eqref{eq:skeleton} has a unique maximal solution $(u_*,T_*)$. Moreover, if $T_*<\infty$ and $\sup_{T\in[0,T_*)}\|u_*\|_{L^2(0,T;V)}<\infty$, then $\lim_{t\uparrow T_*}u_*(t)$ does not exist in $H$.
\end{proposition}

\section{Proof of Theorem \ref{th:LDP general}}\label{sec: proofs}
 
In this section, we prove the LDP stated in Theorem~\ref{th:LDP general}.   
We use the weak convergence approach of Budhiraja and Dupuis \cite{budhidupuis01}. 
This approach provides sufficient conditions for an LDP. 
More precisely, we use the following slight adaptation of  \cite[Th.\ 3.2]{matoussi21}. This version was also used in \cite[Th.\ 4.5]{TV24}, where further details can be found. 
In the formulation, we identify $W$ with an associated $U_1$-valued $Q_1$-Wiener process, where $U_1$ is any separable Hilbert space admitting a Hilbert--Schmidt embedding $J\colon U\into U_1$ and $Q_1=JJ^*$, see \cite[(4.2), Prop.\ A.4]{TV24}.  

Define for $K>0$: 
\begin{align}\label{eq:SKAK}
\begin{split}
  S_K&\ceqq \{\psi\in L^2(0,T;U):\|\psi\|_{L^2(0,T;U)}\leq K\}, \\
  \mathcal{A}_K&\ceqq \{\Psi \col [0,T]\times \Om\to U : \Psi \text{ is a predictable process and }\Psi\in S_K \text{ a.s.}\}
\end{split}
\end{align}

\begin{theorem}[Weak convergence approach]\label{th: sufficient LDP criterion}
Let $(\mathcal E,d)$ be a Polish space. Let $(Y^\eps)_{\eps\in(0,\eps_0)}$ be a family of $\mathcal{E}$-valued random variables on $(\Omega,\F,\P)$.
Let $W$ be a  $U$-cylindrical Brownian motion.  
Suppose that for $\eps\in[0,\eps_0)$, there exist measurable maps $\G^\eps\colon C([0,T];U_1)\to\mathcal{E}$ such that
\begin{enumerate}[label=\emph{(\roman*)},ref=\textup{(\roman*)}]
\item \label{it:suf1} $Y^\eps=\G^\eps(W)$ a.s., for all $\eps\in(0,\eps_0)$,
  \item \label{it:suf2} for any $K>0$,  $(\psi_n)\subset S_K$ and $\psi\in S_K$ with $\psi_n\to \psi$ weakly in $L^2(0,T;U)$, it holds that
  \[
  \G^0\big(\textstyle{\int_0^{\cdot}} \psi_n(s)\dd s\big)\to\G^0\big(\textstyle{\int_0^{\cdot}} \psi(s)\dd s\big) \text{ in } \mathcal{E},
  \]
  \item \label{it:suf3} for any $K>0$ and $(\Psi^\eps)_{\eps\in(0,\eps_0)}\subset \mathcal{A}_K$, it holds that
  \[
  d\big(\G^\eps\big(W(\cdot)+\eps^{-\frac12}\textstyle{\int_0^{\cdot}} \Psi^\eps(s)\dd s\big),\G^0\big(\textstyle{\int_0^{\cdot}} \Psi^\eps(s)\dd s\big)\big)\to 0 \text{ in probability,\, as }\eps\downarrow 0.
  \] 
\end{enumerate}
Then, $(Y^\eps)_{\eps\in(0,\eps_0)}$ satisfies the LDP on $\mathcal{E}$ with   rate function $I\col \mathcal{E}\to [0,\infty]$ defined by
\begin{align*}
I(z)\coloneqq \tfrac12\inf\big\{\textstyle{\int_0^{T}}\|\psi(s)\|_U^2\dd s : \psi\in L^2(0,T;U),\, z=\G^0(\textstyle{\int_0^{\cdot}} \psi(s)\dd s)\big\},
\end{align*}
where $\inf\varnothing\coloneqq \infty$. 
\end{theorem}
In our application, we use $\mathcal{E}=\MR(0,T)$ and we use the measurable maps $\G^0,\G^\eps$  constructed as in
\cite[(4.4), Lem.\ 4.6]{TV24}. These satisfy   $$\G^0\big(\textstyle{\int_0^\cdot} \psi(s)\dd s\big)=u^\psi, \quad\G^\eps\big(W(\cdot)+\eps^{-\frac12}\textstyle{\int_0^{\cdot}} \Psi^\eps(s)\dd s\big)=X^\eps,$$  where $u^\psi$ and $X^\eps$ are strong solutions  to \eqref{eq:skeleton} and \eqref{eq:SPDE tilted}, respectively.

As a first preliminary step, we establish some  consequences of Assumption \ref{ass:coer replace} \ref{it:1}--\ref{it:3}. 
\begin{lemma}\label{lem:consequences ass I-III} 
Let $(A,B)$ satisfy  Assumption  \ref{ass:critvarsettinglocal}.   

If Assumption \ref{ass:coer replace}\ref{it:2} holds, then:
\begin{enumerate}[label=\textit{(\Roman**)},ref=\textit{(\Roman**)}]
\setcounter{enumi}{1}
\item\label{it:gen2}    For every $\psi\in L^2(0,T;U)$, \eqref{eq:skeleton} has a unique strong solution $u^{\psi}$, and it satisfies   the energy estimate of Assumption \ref{ass:coer replace}\ref{it:2}. 
\end{enumerate} 

If Assumption \ref{ass:coer replace}\ref{it:1} and \ref{it:3} hold, then: 
\begin{enumerate}[label=\textit{(\Roman**)},ref=\textit{(\Roman**)}]
\setcounter{enumi}{2}
\item\label{it:gen3}      For every $T,K>0$ and for every collection $(\bphieps)_{\eps\in(0,\eps_0)}$ of predictable stochastic processes $\bphieps\col [0,T]\times \Om\to U$ with   $\|\bphieps\|_{L^2(0,T;U)}\leq K$ a.s., the following holds: for every $\eps\in(0,\eps_0)$, \eqref{eq:SPDE tilted} has a unique strong solution $X^\eps$, and 
    $\lim_{\gamma\to\infty}\sup_{\eps\in(0,\eps_0)}\P(\|X^\eps\|_{\MR(0,T)}>\gamma)=0$. 
    In addition,
    \begin{align} 
\lim_{\gamma\to\infty} \sup_{\eps\in(0,\eps_0)}\P(\|B(\cdot,X^\eps(\cdot))\|_{L^2(0,T;\UH)}>\gamma)=0.\label{eq:lim2}
\end{align}      
\end{enumerate}  
\end{lemma}
\begin{proof}
\ref{it:gen2}:  We need to verify existence and uniqueness of strong solutions to \eqref{eq:skeleton} (the energy estimate then   holds by assumption). Theorem \ref{th: local well posedness skeleton} and Proposition \ref{prop: blow up criterion}  readily yield local well-posedness, existence and uniqueness of a maximal solution and a blow-up criterion. 
Moreover,  the energy estimate of  \ref{it:2} implies that for any $T_*\in(0,\infty)$ and  for any strong solution $u$ to \eqref{eq:skeleton} on $[0,T_*]$: 
\begin{equation}\label{eq:no blow up}
    \sup_{T\in(0,T_*)}\|u\|_{\MR(0,T)}<\infty. 
\end{equation}
Now let $(u_*,T_*)$ be the unique maximal solution and suppose that $T_*<\infty$. Then we can consider $\bar{f}_i\ceqq \hat{F}_i(u_*)+\hat{f}_i$ for $i\in\{1,\ldots,m\}$ and $\bar{g}\ceqq G(u_*)+g$. By \cite[Lem.\ 3.8(iv)]{TV24} and Lemma \ref{lem: 3.8 extra F}, combined with \eqref{eq:no blow up}, we have  $\bar{g}\in L^2(0,T_*;\UH)$ and $\bar{f}_i\in L^{2/(1+2\alpha_i)}(0,T_*;V_{\alpha_i})$. Applying  
  Corollary \ref{cor: skeleton MR linearized crit var setting} with these inhomogeneities, we obtain a solution $u\in\MR(0,T_*)$ to the linear equation \eqref{eq:skeletonlinear}. Now,  $u_*$ solves the same equation on $(0,T)$ for every $T<T_*$, thus by  uniqueness in Corollary \ref{cor: skeleton MR linearized crit var setting}, it follows that $u_*=u$ on $[0,T_*)$. In particular, $\lim_{t\uparrow T_*}u_*(t)=\lim_{t\uparrow T_*}u(t)$ exists in $H$ (recall that $u\in \MR(0,T_*)$), contradicting the blow-up criterion from Proposition \ref{prop: blow up criterion}. Thus we must have $T_*=\infty$, i.e.\  global existence and uniqueness hold. 

\ref{it:gen3}: 
The Yamada--Watanabe theorem in \cite{rock08}, together with \ref{it:1},  ensures  that for each $\eps>0$,   there exists a measurable It\^o map $\G^\eps$ such that for any $U$-cylindrical Brownian motion $W$, $\G^\eps(W)$ is the unique strong solution to \eqref{eq:SPDE} driven by $W$. The proof in \cite{rock08} carries over with $C([0,T];H)\cap L^2(0,T;V)$ in place of the solution space $\mathbb{B}=C([0,\infty);H)\cap L^1_\loc([0,\infty);V)$ (alternatively, one can invoke \cite[Th.\ 3.1]{T25}).
Consequently, by Girsanov's theorem, $\G^\eps(W(\cdot)+\eps^{-1/2}\int_0^\cdot \bphieps\dd s)$  is the unique strong solution to \eqref{eq:SPDE tilted}. 

Now, the first limit was already stated in \ref{it:3}, so it remains to prove \eqref{eq:lim2}.  To this end, we use the first limit and the growth bounds for $B_0$ and $G$ in Assumption \ref{ass:critvarsettinglocal}.   Indeed, Assumption \ref{ass:critvarsettinglocal} implies \cite[Lem.\ 3.8 (ii)(iv)]{TV24}, giving with $n(u)\ceqq \|u\|_{\MR(0,T)}$: 
\begin{align*}
\|B(u)\|_{L^2(0,T;\UH)}&\leq \|B_0(u)u\|_{L^2(0,T;\UH)}+\|G(u)\|_{L^2(0,T;\UH)}+\|g\|_{L^2(0,T;\UH)}\\
&\leq C_{n(u),T} n(u)+\tilde{C}_{n(u),T}(1+n(u))+\|g\|_{L^2(0,T;\UH)}. 
\end{align*}
Using that $C_{n ,T}$ and $\tilde{C}_{n,T}$ are non-decreasing in $n$, we obtain for all $N>0$, $\eps\in(0,\eps_0)$ and $\gamma>0$:
\begin{align*}
\P&(\|B(X^\eps)\|_{L^2(0,T;\UH)}>\gamma)\\
&\;\leq 
\P\big(\|B(X^\eps)\|_{L^2(0,T;\UH)}>\gamma, \|X^\eps\|_{\MR(0,T)}\leq N\big) +\P(\|X^\eps\|_{\MR(0,T)}>N)\\
&\;\leq \P\big(C_{N,T} N+\tilde{C}_{N,T}(1+N)+\|g\|_{L^2(0,T;\UH)}>\gamma\big) +\sup_{\hat{\eps}\in(0,\eps_0)}\P(\|X^{\hat{\eps}}\|_{\MR(0,T)}>N).
\end{align*}

Now let $\delta>0$ be arbitrary. By \ref{it:3}, we can fix an $N=N_\delta$ such that the last term is less than $\delta$. Keeping this $N$ fixed, we take the supremum over $\eps\in(0,\eps_0)$ in the preceding estimate and then let $\gamma\to\infty$, giving $$\lim_{\gamma\to\infty}\sup_{{\eps}\in(0,\eps_0)}\P(\|B(X^\eps)\|_{L^2(0,T;\UH)}>\gamma)\leq 0+\delta.$$ 
We conclude that \eqref{eq:lim2} holds. 
\end{proof} 

The next proposition is based on \cite[Prop.\ 4.9]{TV24} and verifies condition \ref{it:suf2} of Theorem \ref{th: sufficient LDP criterion}. We use Assumption \ref{ass:coer replace}\ref{it:2} and Lemma \ref{lem:consequences ass I-III}\ref{it:gen2} to handle the missing coercivity, and  Lemma \ref{lem: 3.8 extra F} to treat the nonlinearity $\hat{F}$. 
 
\begin{proposition}\label{prop:cpt sublevel sets}
Suppose that Assumption \ref{ass:coer replace} holds. Let $T,K> 0$ and $x\in H$. Equip the set $S_K$ defined in \eqref{eq:SKAK} with the weak topology of $L^2(0,T;U)$. 
Then the following map is continuous:  $(S_K,\mathrm{weak})\to \MR(0,T)\colon  \psi\mapsto u^{\psi}$, with $u^\psi$  the unique strong solution to \eqref{eq:skeleton}. 
\end{proposition}
 
\begin{proof} 
Note that $\psi \mapsto u^{\psi}$ is well-defined due to Lemma \ref{lem:consequences ass I-III}\ref{it:gen2}. 
Suppose that $\psi_n \to \psi$ weakly in $L^2(0,T;U)$ and write $w_n\coloneqq u^{\psi_n}-u^\psi$. We show that $w_n\to 0$ in $\MR(0,T)$.
For each $n\in\N$, $w_n$ is a strong solution to
\begin{align*}
\begin{cases}
  &w_n'+\bar{A_0} w_n=f_n +\hat{F}(u^{\psi_n})-\hat{F}(u^\psi)+\big(\bar{B_0} w_n+g_n\big)\psi_n+{b}(\psi_n-\psi), \\
  &w_n(0)=0,
\end{cases}
\end{align*}
where $\bar{A_0}\coloneqq A_0 (u^\psi)$, $\bar{B_0}\coloneqq B_0(u^\psi)$ and we use the notation from the proof of \cite[Prop.\ 4.9]{TV24} in which we put $F=0$, i.e.:
\begin{align*}
f_n&\coloneqq (A_0(u^\psi)-A_0(u^{\psi_n}))u^{\psi_n}\in L^2(0,T;V^*),\\
g_n&\coloneqq -(B_0(u^\psi)-B_0(u^{\psi_n}))u^{\psi_n}+G(u^{\psi_n})-G(u^\psi)\in L^2(0,T;\UH),\\
{b}&\coloneqq B(u^\psi)=B_0(u^\psi) u^\psi+G(u^\psi)+g\in L^2(0,T;\UH).
\end{align*}
Compared to \cite[Prop.\ 4.9]{TV24}, only the part  $\hat{F}(u^{\psi_n})-\hat{F}(u^\psi)$ in the equation for $w_n$ is new. 
Therefore, by the chain rule \cite[(A.2)]{TV24} and defining $I_j^n(t)$ as in  \cite[(4.18)]{TV24}, we have one extra term: 
\begin{align}
 \frac{1}{2} \|w_n(t)\|_H^2 &=  I_1^n(t)+I_2^n(t)+I_3^n(t)+ \int_0^t\<\hat{F}(s,u^{\psi_n}(s))-\hat{F}(s,u^\psi(s)),w_n(s)\?\dd s. \label{eq: gronwall prep w_n}
\end{align}
Using Assumption \ref{ass:coer replace}\ref{it:2} and boundedness of $(\psi_n)$ in $L^2(0,T;U)$, we can define 
\begin{equation}\label{eq: def N}
N\coloneqq  \|u^\psi\|_{\MR(0,T)}+\sup_{n\in\N}\|u^{\psi_n}\|_{\MR(0,T)} <\infty.
\end{equation}  
By \cite[(4.18), (4.20), (4.23), p.\ 27]{TV24}, we have  
\begin{align*}
I_1^n(t)+I_2^n(t)&= \int_0^t  \<-\bar{A_0}w_n(s)+f_n(s),w_n(s)\?+\<(\bar{B_0}w_n(s)+g_n(s))\psi_n(s),w_n(s)\?\dd s\\ 
&\leq -\frac{\theta_{N,T}}{2} \|w_n\|_{L^2(0,t;V)}^2+\int_0^t \big(h_n(s)+M_{N,T}+\frac{1}{2}\|\psi_n(s)\|_U^2\big)\|w_n(s)\|_{H}^2\dd s, 
\end{align*}
where $(h_n)\subset L^1(0,T)$ satisfies $\sup_{n\in\N}\|h_n\|_{L^1(0,T)}<\infty$ and  $\theta_{N,T},M_{N,T}$ are the constants from Assumption \ref{ass:critvarsettinglocal}\ref{it:coercivelinear}. 
Moreover, for every $\sigma>0$, we have thanks to Lemma \ref{lem: 3.8 extra F}\ref{it:emb6}:
\begin{align*}
 \int_0^t\<\hat{F}(u^{\psi_n})-\hat{F}(u^\psi),w_n\?\dd s   
  &\leq \hat{C}_{\sigma,N,T} \int_0^t  (1+\|u^{\psi_n}\|_V^2+ \|u^\psi\|_V^2)\|w_n\|_H^2\dd s+\sigma\|w_n\|_{L^2(0,t;V)}^2. 
\end{align*}
Hence, fixing $\sigma=\theta_{N,T}/4$ and recalling \eqref{eq: gronwall prep w_n}, we obtain 
\begin{equation*}
\|w_n(t)\|_{H}^2\leq -\frac{\theta_{N,T}}{2}\|w_n\|_{L^2(0,t;V)}^2 +2\int_0^t \big(\hat{h}_n(s)+M_{N,T}+\frac{1}{2}\|\psi_n(s)\|_U^2\big)\|w_n(s)\|_{H}^2\dd s+2\sup_{s\in[0,t]}|I_3^n(s)|, 
\end{equation*}
where $\hat{h}_n \ceqq h_n +\hat{C}_{\sigma,N,T}(1+\|u^{\psi_n}\|_V^2+ \|u^\psi\|_V^2)$ also satisfies  $\sup_{n\in\N}\|\hat{h}_n\|_{L^1(0,T)}<\infty$, due to \eqref{eq: def N}. 
Gronwall's inequality thus gives for all $n\in\N$:
\begin{equation}\label{eq: gronwall w_n}
 \|w_n\|_{\MR(0,T)}^2\lesssim_{N,T}  \sup_{s\in[0,T]}|I_3^n(s)|\exp(2\hat{\kappa}),
\end{equation}
with constant $\hat{\kappa}\coloneqq \sup_{n\in\N}(\|\hat{h}_n\|_{L^1(0,T)}+\frac{1}{2}\|\psi_n\|_{L^2(0,T;U)}^2)+M_{N,T}<\infty$. 

Finally, we will apply \cite[Lem.\ 4.8]{TV24} to $I_3^n= \int_0^\cdot \<b(s)(\psi_n(s)-\psi(s)),w_n(s)\?\dd s$ (see \cite[(4.18)]{TV24}) to obtain $\sup_{s\in[0,T]}|I_3^n(s)|\to 0$.  Combined with \eqref{eq: gronwall w_n}, this yields
$w_n\to0$ in $\MR(0,T)$, concluding the proof. 
To apply \cite[Lem.\ 4.8]{TV24}, we must verify  boundedness of $(w_n)\subset C([0,T];H)$ and $(\hat{\alpha}_n)\subset L^1(0,T;V^*)$, where $\hat{\alpha}_n\ceqq \bar{\alpha}_n+\hat{F}(u^{\psi_n})-\hat{F}(u^\psi)$, with  
$\bar{\alpha}_n\coloneqq -\bar{A_0}w_n+f_n+(\bar{B_0}w_n+g_n)\psi_n+{b}(\psi_n-\psi)$. Boundedness of $(w_n)$ holds by \eqref{eq: def N} and boundedness of $(\bar{\alpha}_n)$ in $L^1(0,T;V^*)$ was already shown in the proof of \cite[Prop.\ 4.9]{TV24}. 
Thus, it suffices to observe that  Lemma \ref{lem: 3.8 extra F}\ref{it:emb4} and \eqref{eq: def N} imply that for $i\in\{1,\ldots,m\}$: $\sup_{n\in\N}\|\hat{F}_i(u^{\psi_n})-\hat{F}_i(u^\psi)\|_{L^{2/(1+2\alpha_i)}(0,T;V_{\alpha_i})}\leq 2\hat{C}_{N,T}(1+N^{1+2\alpha_i})<\infty$, and to recall that $L^{2/(1+2\alpha_i)}(0,T;V_{\alpha_i})\into L^1(0,T;V^*)$. 
\end{proof}

The next proposition is the analog of \cite[Prop.\ 4.12]{TV24} and verifies condition \ref{it:suf3} of Theorem \ref{th: sufficient LDP criterion}. We exploit Assumption \ref{ass:coer replace}\ref{it:1} and \ref{it:3}, Lemma \ref{lem:consequences ass I-III}\ref{it:gen3}   and Lemma \ref{lem: 3.8 extra F} to handle the missing coercivity and the extra nonlinearity $\hat{F}$ in the estimates. Note that  Lemma \ref{lem:consequences ass I-III} ensures the unique existence of $X^\eps$ and $u^\eps$ in the formulation below. 

\begin{proposition}\label{prop: X^eps weak conv}
Suppose that Assumption \ref{ass:coer replace} holds and let $T,K>0$ and $x\in H$. For $\eps\in(0,\eps_0)$, let $\bphieps\col [0,T]\times \Om\to U$  be a predictable stochastic process with   $\|\bphieps\|_{L^2(0,T;U)}\leq K$ a.s., and let $X^\eps$ be the strong solution to the stochastic control problem \eqref{eq:SPDE tilted} with $\bphieps$. 
Moreover, suppose that $u^\eps(\om)=u^{\bphieps(\om)}$ for a.e.\ $\om\in\Om$, where $u^{\bphieps(\om)}$  is the strong solution  to \eqref{eq:skeleton} with $\psi=\Psi^\eps(\om)$. 
Then $X^\eps- u^{\eps}\to 0$ in probability in $\MR(0,T)$ as $\eps\downarrow0$.
\end{proposition} 
\begin{proof}
By Assumption \ref{ass:coer replace}\ref{it:2}, we can define 
\begin{align}\label{eq: def esssup N}
N\coloneqq \esssup_{\om\in\Om}\sup_{\eps\in(0,\eps_0)}\|u^\eps(\om)\|_{\MR(0,T)}\leq C_{x}(T,K)<\infty,
\end{align}
and for $\eps\in(0,\eps_0)$ and $n\in\N$, we let 
\[
E_{n,\eps}\coloneqq \{\|X^\eps\|_{\MR(0,T)}\leq n\}\cap \{\|u^{\eps}\|_{\MR(0,T)}\leq N\}.
\]
For all $\eps\in(0,\eps_0)$ and $n\in\N$,  
\begin{align*}
\P(\|X^\eps- u^{\eps}\|_{\MR(0,T)}>\gamma)
&\leq \P(\{\|X^\eps- u^{\eps}\|_{\MR(0,T)}>\gamma\}\cap E_{n,\eps})+\P(E_{n,\eps}^c), 
\end{align*}
and for the last term, we have by Assumption \ref{ass:coer replace}\ref{it:3}  and \eqref{eq: def esssup N}: 
\[
\lim_{n\to\infty}\sup_{\eps\in(0,\eps_0)}\P(E_{n,\eps}^c)= \lim_{n\to\infty}\sup_{\eps\in(0,\eps_0)}\P(\|X^\eps\|_{\MR(0,T)}>n)=0. 
\]
Therefore, to prove the claimed convergence in probability, it suffices to show that for all $\delta>0$ and   large enough $n\in\N$:
\begin{equation}\label{eq: X^eps weak conv to show 2}
  \lim_{\eps\downarrow 0}\P(\{\|X^\eps- u^{\eps}\|_{\MR(0,T)}>\delta\}\cap E_{n,\eps})=0.
\end{equation}
We will prove that the above even holds for every $n\geq N$. Following the proof of \cite[Prop.\ 4.12]{TV24}, 
  the It\^o formula \cite[(A.4)]{TV24} gives for all $t\in[0,T]$:
  \begin{align}
    \|X^\eps(t)-u^{\eps}(t)\|_H^2
    &=2\int_0^t\<-A(s,X^\eps(s))+A(s,u^{\eps}(s)),X^\eps(s)-u^{\eps}(s)\?\dd s\notag\\
    &\qquad+2\int_0^t\<\big(B(s,X^\eps(s))-B(s,u^{\eps}(s))\big)\Psi^\eps(s),X^\eps(s)-u^{\eps}(s)\?\dd s\notag\\
    &\qquad+\eps\int_0^t \nn B(s,X^\eps(s))\nn_H^2\dd s\notag\\
    &\qquad+2\sqrt{\eps}\int_0^t \<X^\eps(s)-u^{\eps}(s),B(s,X^\eps(s))\dd W(s)\?\notag\\
    &\eqqcolon {I}_1^\eps(t)+{I}_2^\eps(t)+{I}_3^\eps(t)+{I}_4^\eps(t). \label{eq: I_i^eps}
    \end{align}
   First, we will derive an estimate of the form
    \begin{equation}\label{eq: to show I^1 I^2}
    {I}_1^\eps(t)+{I}_2^\eps(t)\leq -\frac{\theta_{n,T}}{2}\|X^\eps-u^{\eps}\|_{L^2(0,t;V)}^2+\int_0^t \hat{h}_{n,\eps}(s)\|X^\eps(s)-u^{\eps}(s)\|_H^2 \dd s
    \end{equation}
    that holds a.s.\ on the set $E_{n,\eps}$, for every $t\in[0,T]$ and $\eps\in(0,\eps_0)$. Here,  $\hat{h}_{n,\eps}\in L^1(0,T)$ a.s., with $\hat{\alpha}_n\coloneqq \sup_{\eps\in(0,\eps_0)}\esssup_\Om \|\hat{h}_{n,\eps}\one_{E_{n,\eps}}\|_{L^1(0,T)}<\infty$ for each $n\geq N$. 
    Then, Gr\"onwall's inequality immediately yields that a.e.\ on $E_{n,\eps}$:
    \begin{equation*}
    \|X^\eps-u^{\eps}\|_{\MR(0,T)}^2\lesssim_{n,T}\exp(\hat{\alpha}_n)\big({I}_3^\eps(T)+\sup_{t\in[0,T]}|{I}_4^\eps(t)|\big).
    \end{equation*}  
Therefore, to obtain   \eqref{eq: X^eps weak conv to show 2}, it suffices to prove  \eqref{eq: to show I^1 I^2}, together with the following (see the proof of \cite[Prop.\ 4.12]{TV24} for the detailed reasoning): 
    \begin{align}
    &\lim_{\eps\downarrow 0}\P({I}_3^\eps(T)> \delta)=0\quad \text{ for all } \delta>0, \label{eq: to show I^3}\\
    &\lim_{\eps\downarrow 0}\P(\sup_{t\in[0,T]}|{I}_4^\eps(t)|>\delta)=0\quad \text{ for all } \delta>0. \label{eq: to show I^4}
    \end{align}
    It remains to establish \eqref{eq: to show I^1 I^2}, \eqref{eq: to show I^3} and \eqref{eq: to show I^4}. 

    Let us begin with \eqref{eq: to show I^1 I^2}. Recall that  $(A,B)$ is defined by \eqref{eq:defAB}, which is the same as $(A,B)$ in \cite{TV24} except for $F$ and $f$ therein being replaced by our terms $\hat{F}$ and $\hat{f}$.  
   From \cite[(4.32)--(4.36)]{TV24} (with $F=0$), we can directly conclude that for all $\eps\in(0,\eps_0)$, we have pointwise  on $E_{n,\eps}$:
    \begin{align*}
     {I}_1^\eps(t)+{I}_2^\eps(t)&\leq -\theta_{n,T}\|X^\eps-u^{\eps}\|_{L^2(0,t;V)}^2+\int_0^t h_{n,\eps}(s)\|X^\eps(s)-u^{\eps}(s)\|_H^2\dd s\\
     & \quad+2\int_0^t \<\hat{F}(s,X^\eps(s))-\hat{F}(s,u^\eps(s)),X^\eps(s)-u^\eps(s)\?\dd s,
    \end{align*}
    where $h_{n,\eps}$ is of the form $
    h_{n,\eps} = C_{n,T,\sigma,\tilde{\sigma}}\left(1+\|\Psi^\eps\|_U^2+\|X^\eps\|_V^2+\|u^{\eps}\|_V^2\right)$ 
    for a constant $C_{n,T,\sigma,\tilde{\sigma}}>0$, and $\sigma$ and $\tilde{\sigma}$ depend only on $\theta_{n,T}$. Moreover, by virtue of Lemma \ref{lem: 3.8 extra F}\ref{it:emb6}, we have  for all $\eps\in(0,\eps_0)$, a.s.\ on $E_{n,\eps}$:
    \begin{align*}
    &\int_0^t \<\hat{F}(s,X^\eps(s))-\hat{F}(s,u^\eps(s)),X^\eps(s)-u^\eps(s)\?\dd s\\
    &\qquad\leq \hat{C}_{\frac{\theta_{n,T}}{4},n,T} \int_0^t  (1+\|X^\eps\|_V^2+ \|u^\eps\|_V^2)\|X^\eps-u^\eps\|_H^2\dd s+\frac{\theta_{n,T}}{4}\|X^\eps-u^\eps\|_{L^2(0,t;V)}^2.
    \end{align*}
    Thus, combining the last two estimates, we conclude that \eqref{eq: to show I^1 I^2} holds with $\hat{h}_{n,\eps}\ceqq h_{n,\eps}+2\hat{C}_{\frac{\theta_{n,T}}{4},n,T} (1+\|X^\eps\|_V^2+ \|u^\eps\|_V^2)$. 
    By the $L^2$-boundedness of $\Psi^\eps$ and by definition of $E_{n,\eps}$, we have a.s.\  $\|\hat{h}_{n,\eps}\one_{E_{n,\eps}}\|_{L^1(0,T)}\leq C_{n,T,\sigma,\tilde{\sigma}}(T+K^2+2n^2)+2\hat{C}_{\frac{\theta_{n,T}}{4},n,T} (T+2n^2)$,  for every $\eps\in(0,\eps_0)$. Thus  $\hat{h}_{n,\eps}$ is of the desired form.

    Concerning \eqref{eq: to show I^3},  \eqref{eq:lim2} readily implies for all  $\delta>0$:
    \begin{align*}
    \lim_{\eps\downarrow 0}\P({I}_3^\eps(T)>\delta)
    =\lim_{\eps\downarrow 0}\P( \| B(\cdot,X^\eps(\cdot))\|_{L^2(0,T;\UH)}^2>{\delta}{\eps}^{-1})=0.
    \end{align*}
    
    Lastly, we prove \eqref{eq: to show I^4}.  
    By  \cite[Prop.\ 18.6]{kallenberg21}, \eqref{eq: to show I^4} is equivalent to
    \begin{equation}\label{eq: I^4 sufficient convergence}
    \lim_{\eps\downarrow0}\P\Big(\eps\int_0^T\|\<X^\eps(s)-u^{\eps}(s),B(s,X^\eps(s))(\cdot)\?\|_{\mathscr{L}_2(U,\R)}^2\dd s>\delta\Big)=0 \text{ for all } \delta>0.
    \end{equation}
    By \cite[(4.39), (4.40)]{TV24}, we have   for all $\delta>0$ and $\eps\in(0,\eps_0)$:
    \begin{align*}
    \P\Big(\eps\int_0^T\|\<X^\eps(s)-&u^{\eps}(s),B(s,X^\eps(s))(\cdot)\?\|_{\mathscr{L}_2(U,\R)}^2\dd s>\delta\Big)\notag\\
    &\leq \P\big(\|X^\eps\|_{C([0,T];H)}>\tfrac{1}{2}({\delta}{\eps}^{-1})^{\frac{1}{4}}\big)+\P\big(\|u^{\eps}\|_{C([0,T];H)}>\tfrac{1}{2}({\delta}{\eps}^{-1})^{\frac{1}{4}}\big)\notag\\
    &\qquad+ \P\big(\| B(\cdot,X^\eps(\cdot))\|_{L^2(0,T;\UH)} >({\delta}{\eps}^{-1})^{\frac{1}{4}}\big). 
    \end{align*} 
    Now, combining the above estimate with Assumption \ref{ass:coer replace}\ref{it:3}, \eqref{eq: def esssup N} (recall $n\geq N$) and \eqref{eq:lim2}, we conclude that   \eqref{eq: I^4 sufficient convergence} holds.
\end{proof}

Gathering the results of this section, we now prove the LDP of Theorem \ref{th:LDP general} by verifying the conditions of Theorem \ref{th: sufficient LDP criterion}. 
 
\begin{proof}[\textbf{Proof of Theorem \ref{th:LDP general}}] 
We apply Theorem \ref{th: sufficient LDP criterion} with $\mathcal E=\MR(0,T)$. 
Let $\G^0\colon C([0,T];U_1)\to\MR(0,T)$ be defined as in
\cite[(4.4)]{TV24}, namely, 
\[
\G^0(\gamma)\ceqq
\begin{cases}
u^\psi,
& \text{if } 
\gamma=\int_0^\cdot\psi(s)\dd s
\text{ for some }\psi\in L^2(0,T;U),\\
0,
&\text{otherwise},
\end{cases}
\]
where $u^\psi$ is the unique solution to \eqref{eq:skeleton} provided by Lemma \ref{lem:consequences ass I-III}\ref{it:gen2}. 
By Proposition \ref{prop:cpt sublevel sets} and the argument in \cite[Rem.\ 4.10]{TV24},  $\G^0$ is measurable.

It now suffices to show that the conclusions of \cite[Lem.\ 4.6, Prop.\ 4.9, Prop.\ 4.12]{TV24} are still valid in the current setting, i.e.\  under Assumption \ref{ass:coer replace}. Indeed, these three results are precisely used in the LDP proof \cite[Th.\ 2.6, p. 32]{TV24} to verify conditions \ref{it:suf1}--\ref{it:suf3} of Theorem \ref{th: sufficient LDP criterion}.

\textbf{\cite[Lem.\ 4.6]{TV24}:}  
Global well-posedness of \eqref{eq:SPDE}, which is used in that lemma, is assumed separately through 
Assumption \ref{ass:coer replace}\ref{it:1}.
As explained in the proof of Lemma \ref{lem:consequences ass I-III}\ref{it:gen3}, by Assumption \ref{ass:coer replace}\ref{it:1}, the Yamada--Watanabe theorem and Girsanov's theorem, there exist measurable maps $\G^\eps$ such that $\G^\eps(W)$ is the unique strong solution to \eqref{eq:SPDE} and  $\G^\eps(W(\cdot)+\eps^{-1/2}\int_0^\cdot \bphieps\dd s)$   solves \eqref{eq:SPDE tilted} uniquely, thus the conclusions of \cite[Lem.\ 4.6]{TV24} are still valid.

\textbf{\cite[Prop.\ 4.9]{TV24}:} 
This is covered by Proposition \ref{prop:cpt sublevel sets}.  

\textbf{\cite[Prop.\ 4.12]{TV24}:} The statement is valid by Proposition \ref{prop: X^eps weak conv} 
 and by observing that $X^\eps=\G^\eps(W(\cdot)+\eps^{-1/2}\int_0^\cdot \bphieps\dd s)$ a.s.\  since the strong solution to \eqref{eq:SPDE tilted} is unique.  
\end{proof}

\section{Perturbed LDP results}\label{sec:perturb}

In this section, we state and prove the LDP results for the perturbed equation \eqref{eq:SPDE perturbed} discussed in Subsection \ref{ss:perturbed main results}. We first treat  the (possibly) non-coercive case and then derive two corollaries for semilinear coercive equations.

\subsection{The non-coercive perturbed case}

For the drift perturbation $\eta(\eps)\tilde A$, we impose conditions \ref{it:tildeA} and \ref{it:eta} of the next assumption. For convenience, the same assumption also   collects the 
remaining conditions \ref{it:1pert}--\ref{it:3pert} sufficient for the LDP. These are the perturbed analogues of the well-posedness and uniform boundedness conditions in Assumption \ref{ass:coer replace}.

\begin{assumption}[Perturbed equation]\label{ass:perturb2} 
$(A,B)$  satisfies  Assumption  \ref{ass:critvarsettinglocal} and $x\in H$.  
Moreover, for $m$, $\alpha_i$, $\hat{\beta}_i$, $\hat{\rho}_i$ and $C_{n,T}$ from Assumption \ref{ass:critvarsettinglocal}, we have:  
\begin{enumerate}[label=\textit{(\roman*)},ref=\textit{(\roman*)}] 
\item\label{it:tildeA} $\tilde{A}=A_1+\sum_{i=1}^m\tilde{F}_i$ where  
${A}_1\col \R_+\times V\to V^*$ and $\tilde{F}_i\col \R_+\times V\to V_{\alpha_i}$ are Borel measurable. For all  $n,T>0$, $t\in[0,T]$ and $u,v\in V$ with $\|u\|_H,\|v\|_H\leq n$, it holds that
\begin{align*}
   &\|A_1(t,v)\|_{V^*}\leq C_{n,T}\|v\|_{V},\\ 
&\|\tilde{F}_i(t,u)\|_{\alpha_i}\leq C_{n,T} (1+\|u\|_{\hat{\beta}_i}^{\hat{\rho}_i+1}). 
\end{align*}
\item\label{it:eta} $\eta(\eps)\in\R$ for $\eps>0$ and  $\lim_{\eps\downarrow0}\eta(\eps)=0$. 
\end{enumerate} 
In addition, the following conditions are satisfied for some $\eps_0>0$:
\begin{enumerate}[label=\textit{( $\widetilde{\!\!\text{\Roman*}}$)},ref=\textit{( $\widetilde{\!\!\text{\Roman*}}$)}]
  \item \label{it:1pert} For every $U$-cylindrical Brownian motion $W$ on a filtered probability space $(\Om,\mathcal{F},(\mathcal{F}_t)_{t\geq 0},\P)$  and for every $\eps\in (0,\eps_0)$ and $T>0$, \eqref{eq:SPDE perturbed} has a unique strong solution $\tilde{Y}^\eps$ in the sense of Definition \ref{def:sol}. 
  \item \label{it:2pert} Assumption \ref{ass:coer replace}\ref{it:2} holds (for the same skeleton equation). 
\item \label{it:3pert} For every $T,K>0$ and for every collection $(\bphieps)_{\eps\in(0,\eps_0)}$ of predictable stochastic processes $\bphieps\col [0,T]\times \Om\to U$ with   $\|\bphieps\|_{L^2(0,T;U)}\leq K$ a.s., the following holds: 
if $(\tilde{X}^\eps)_{\eps\in(0,\eps_0)}$ is a family such that for all $\eps\in(0,\eps_0)$, $\tilde{X}^\eps$ is a strong solution to \eqref{eq:SPDE tilted perturb}, then 
\begin{equation*} 
    \lim_{\gamma\to\infty}\sup_{\eps\in(0,\eps_0)}\P(\|\tilde{X}^\eps\|_{\MR(0,T)}>\gamma)=0. 
\end{equation*}
\end{enumerate} 
\end{assumption}

\begin{remark}
  Lemma \ref{lem:suff non-blow up} extends to the setting of \eqref{eq:SPDE perturbed} and can be used to validate Assumption \ref{ass:perturb2}\ref{it:1pert}. That is, if Assumptions  \ref{ass:critvarsettinglocal} and \ref{ass:perturb2}\ref{it:tildeA}\ref{it:eta} hold, then Assumption \ref{ass:perturb2}\ref{it:1pert} is valid if the following non-blow-up is satisfied: 
  \hspace{0.05em}  
   \eqref{eq:suff non-blow up} holds for every maximal solution $(u,\sigma)$ to \eqref{eq:SPDE perturbed} with $\eps\in(0,\eps_0)$.  
\end{remark}

We can now state the non-coercive perturbed LDP result. Its proof is given in Subsection \ref{ss:proof perturb}.

\begin{theorem}[Perturbed LDP]\label{th:LDP perturbed2}  
    Let Assumption \ref{ass:perturb2} hold and let $T>0$.  
    Then the family $(\tilde{Y}^\eps)_{\eps\in(0,\eps_0)}$ of solutions to \eqref{eq:SPDE perturbed} satisfies the large deviation principle on $L^2(0,T;V)\cap C([0,T];H)$ with  the  rate function from Theorem \ref{th:LDP general}, i.e.\   
\begin{equation*} 
I(z)=\frac{1}{2}\inf\Big\{\textstyle{\int_0^T}\|\psi(s)\|_U^2\dd s : \psi\in L^2(0,T;U), \, z=u^{\psi}\Big\},
\end{equation*}
where $\inf\varnothing\coloneqq +\infty$ and $u^\psi$ is the unique strong solution to \eqref{eq:skeleton}. 
\end{theorem}

\begin{remark}\label{rem:appl strat}
Theorem \ref{th:LDP perturbed2} is useful for proving an LDP for non-coercive equations in which the noise is interpreted in Stratonovich sense. In particular, it can be used to prove the LDP for the examples given in Example \ref{ex:applications} with Stratonovich noise. For the  3D primitive equations  with transport noise, this has already been done in \cite[Th.\ 2.9]{AT25}.
\end{remark}

\begin{remark}\label{rem:finite time horizon}
Theorem \ref{th:LDP perturbed2} is also valid for a fixed time horizon $T_1>0$, assuming that Assumption \ref{ass:critvarsettinglocal} and Assumption \ref{ass:perturb2}\ref{it:tildeA} hold on $[0,T_1]$ instead of locally on $\R_+$, that Assumption \ref{ass:perturb2}\ref{it:eta} holds, that Assumption \ref{ass:perturb2}\ref{it:1pert} and \ref{it:3pert} hold with $T=T_1$, and that Assumption \ref{ass:perturb2}\ref{it:2pert} holds for all $T\in(0,T_1]$.  
The conclusion of  Theorem \ref{th:LDP perturbed2} then holds for $T=T_1$. 

Indeed, the proof of Theorem \ref{th:LDP perturbed2} applies with the evident restriction that every auxiliary time horizon is chosen in $(0,T_1]$.  For well-posedness of the skeleton equation on $[0,T_1]$, we note that Lemma \ref{lem:consequences ass I-III}\ref{it:gen2} remains valid since the blow-up criterion of Proposition \ref{prop: blow up criterion} also holds with `$T_*<\infty$' replaced by `$T_*<T_1$'. This follows from the proof of \cite[Prop.\ 3.15]{TV24}. See also \cite{BGV} for the stochastic analog on finite time intervals. 
\end{remark}

\subsection{The coercive perturbed case}

For the next two corollaries, we further specialize to semilinear equations. 
That is, we will assume that $A_0$ and $B_0$ are linear in the sense that 
\begin{equation}\label{eq:semilin}
A_0\col \R_+\to \mathscr{L}(V,V^*), \quad B_0\col \R_+\to \mathscr{L}(V,\UH). 
\end{equation}
In this case, we write $A_0(t)v$ and $B_0(t)v$ instead of $A_0(t,u)v$ and $B_0(t,u)v$. 
Also, we will let $\tilde{A}$ be linear. In particular, $\tilde{A}$ does not contain $\tilde{F}_i$-terms as in Assumption \ref{ass:perturb2}.  In the proof below, this is used to maintain stability of the coercivity \eqref{eq:def coercive} under the perturbation $\eta(\eps)\tilde{A}$. 

\begin{corollary}[Perturbed LDP -- coercive case, semilinear]\label{cor:LDPcoerciveperturbed}
  Let $T>0$. Let $(A,B)$ satisfy Assumption \ref{ass:critvarsettinglocal} on $[0,T]$, and suppose that $(A,B)$ is coercive in the sense of \eqref{eq:def coercive} on $[0,T]$.  
  Assume that $A_0$ and $B_0$ are as in \eqref{eq:semilin}, let $\tilde{A}\col \R_+ \to \mathscr{L}(V,V^*)$ be  Borel measurable and   
  $$
  C_T\ceqq \sup_{t\in[0,T]}\|\tilde{A}(t)\|_{\mathscr{L}(V,V^*)}<\infty.  
  $$
  Let $(\eta(\eps))_{\eps>0}\subset \R$ be such that  $\lim_{\eps\downarrow0}\eta(\eps)=0$. 
  
  Then, there exists an $\eps_0>0$ such that for all $\eps\in(0,\eps_0)$, \eqref{eq:SPDE perturbed} has a unique strong solution $\tilde{Y}^\eps$ on $[0,T]$ and the family $(\tilde{Y}^\eps)_{\eps\in(0,\eps_0)}$ 
    satisfies the large deviation principle on $L^2(0,T;V)\cap C([0,T];H)$, with the same rate function as in Theorem \ref{th:LDP general}. 
\end{corollary} 

\begin{proof} 
Note that Assumption \ref{ass:perturb2}\ref{it:2pert} coincides with Assumption \ref{ass:coer replace}\ref{it:2}, which was already verified in the proof of Corollary \ref{cor:LDPcoercive}. 
Thus, by Remark \ref{rem:finite time horizon}, it remains to prove that Assumption \ref{ass:perturb2}\ref{it:1pert} and \ref{it:3pert} hold for the fixed $T$. Then, \ref{it:1pert} covers the stated well-posedness and the LDP follows from Theorem \ref{th:LDP perturbed2}.  

\ref{it:1pert}: For  $\eps>0$,  
define $A_0^\eps(t)v\ceqq A_0(t)v-\eta(\eps)\tilde{A}(t)v$ and $A^\eps(t,v)\ceqq A(t,v)-\eta(\eps)\tilde{A}(t)v$. 
We show that $(A_0^\eps, B_0)$ and $(A^\eps, B)$ are coercive for all sufficiently small $\eps$. 
 
Since $A_0$ and $B_0$ are as in \eqref{eq:semilin}, Assumption \ref{ass:critvarsettinglocal}\ref{it:coercivelinear} is satisfied with $(\theta_{n,T},M_{n,T})= (\theta_{0,T},M_{0,T})$.  
Choose $\eps_0'>0$ such that $|\eta(\eps)|<\theta_{0,T}/(2C_{T})$ for all $\eps\in(0,\eps_0')$.  
Then, for all $v\in V$, $t\in[0,T]$, and $\eps\in(0,\eps_0')$:
 \begin{align*}
  \<A_0^\eps(t)v,v\?-\tfrac{1}{2}\nn   B_0(t)v\nn_H^2 
  &\geq -|\eta(\eps)|\|\tilde{A}(t)v\|_{V^*}\|v\|_V+\theta_{0,T}\|v\|_V^2-M_{0,T}\|v\|_H^2 \\
  &\geq \tfrac{\theta_{0,T}}{2}\|v\|_V^2-M_{0,T}\|v\|_H^2. 
  \end{align*}
  Similarly, using \eqref{eq:def coercive} and choosing $\eps_0\in(0,\eps_0']$ such that $|\eta(\eps)|<\theta/(2C_{T})$ for all $\eps\in(0,\eps_0)$ yields 
  \begin{align}\label{eq:coerpert}
  \<A^\eps(t,v),v\?-\tfrac{1}{2}\nn   B(t,v)\nn_H^2 
  &\geq -|\eta(\eps)|\|\tilde{A}(t)v\|_{V^*}\|v\|_V +\theta \|v\|_V^2-M \|v\|_H^2-|\phi(t)|^2\notag\\
  &\geq \tfrac{\theta }{2}\|v\|_V^2-M \|v\|_H^2-|\phi(t)|^2.
  \end{align}
It immediately follows that $(A_0^\eps, \sqrt{\eps}B_0)$ and $(A^\eps, \sqrt{\eps}B)$ are also coercive for $\eps\in(0,\eps_0\wedge 1)$, and hence, \eqref{eq:SPDE perturbed} has a unique strong solution on $[0,T]$ by \cite[Th.\ 3.5]{AV22variational} applied to $(A^\eps,\sqrt{\eps}B)$, noting that the latter result also holds for $A^\eps$ containing the terms $\hat{F}$ and $\hat{f}$ and for finite time horizons (see \cite[Th.\ 7.1, Rem.\ 7.2]{BGV}). 
 
\ref{it:3pert}: Using the coercivity estimate \eqref{eq:coerpert} and letting $\eps\in(0,\eps_0\wedge 1/2)$, we can immediately apply the 
  proof of \cite[Lem.\ 4.11]{TV24} with $(A,B)$ replaced by $(A^\eps,B)$ and with $\theta$ replaced by ${\theta}/{2}$.  Redefining $\eps_0$, this concludes the proof. 
\end{proof}

Next, we relax the coercivity assumptions of Corollary \ref{cor:LDPcoerciveperturbed} by exploiting the small-noise regime. The proof uses an additional rescaling argument to recover a suitable small-noise coercivity. Let us emphasize that by taking $\eta\equiv0$, the next corollary also applies immediately to the unperturbed equation \eqref{eq:SPDE}. 

The result will be used in Subsection \ref{ss:NS}, where it allows us to impose the natural parabolicity condition for well-posedness in the Stratonovich formulation, corresponding to coercivity of $(A_0,0)$ rather than of $(A_0,B_0)$, see Remark \ref{rem:NS coer}. 

\begin{corollary}[Perturbed LDP -- perturbed coercive case, semilinear]\label{cor:LDPcoerciveperturbed1}
Let $T>0$.  Let $(A,B)$ satisfy Assumption \ref{ass:critvarsettinglocal} on $[0,T]$ without  \ref{ass:critvarsettinglocal}\ref{it:coercivelinear}. Let 
   $(A_0,0)$ satisfy Assumption \ref{ass:critvarsettinglocal}\ref{it:coercivelinear} on $[0,T]$.  
    Suppose that for some $\delta>0$, the following coercivity holds: there exist $\theta,M>0$ and ${\phi}\in L^2(0,T)$ such that  
   \begin{equation}\label{eq:def coercive delta}
\<A(t,v),v\?-\delta \|G(t,v)\|_{\mathscr{L}_2(U,H)}^2 \geq \theta\|v\|_V^2-M\|v\|_H^2-|{\phi}(t)|^2,\qquad v\in V,\, t\in[0,T]. 
\end{equation}  
Suppose that $A_0$ and $B_0$ are as in \eqref{eq:semilin}, and that  $\tilde{A}\col \R_+ \to \mathscr{L}(V,V^*)$ is Borel measurable and satisfies    $\sup_{t\in[0,T]}\|\tilde{A}(t)\|_{\mathscr{L}(V,V^*)}<\infty$. Let $(\eta(\eps))_{\eps>0}\subset \R$ be such that  $\lim_{\eps\downarrow0}\eta(\eps)=0$. 
  
Then there exists an $\eps_0>0$, such that for all $\eps\in(0,\eps_0)$, \eqref{eq:SPDE perturbed} has a unique strong solution $\tilde{Y}^\eps$ on $[0,T]$ and the family $(\tilde{Y}^\eps)_{\eps\in(0,\eps_0)}$  
satisfies the large deviation principle on $L^2(0,T;V)\cap C([0,T];H)$, with the same  rate function as in Theorem \ref{th:LDP general}. 
\end{corollary}

\begin{proof} 
The result will be derived from Corollary \ref{cor:LDPcoerciveperturbed}. 
First, we prove that for some $\delta_1>0$, $(A,\delta_1 B)$ satisfies the full Assumption \ref{ass:critvarsettinglocal} on $[0,T]$. One has to check only \ref{ass:critvarsettinglocal}\ref{it:coercivelinear} for $(A_0,\delta_1 B_0)$, as the other conditions are satisfied by $(A,B)$, hence by $(A,\delta_1 B)$ for all $\delta_1>0$. Using \eqref{eq:semilin}, coercivity of $(A_0,0)$ (with constants $\theta_{0,T},M_{0,T}$) and the growth bound for $B_0$ in \ref{ass:critvarsettinglocal}\ref{it:growth AB}, we have for all $v\in V$ and $t\in[0,T]$:  
\begin{align*} 
\<A_0(t)v,v\?-\tfrac{1}{2}\nn \delta_1 B_0(t)v\nn_H^2 
&\geq  \theta_{0,T} \|v\|_V^2-M_{0,T} \|v\|_H^2-\tfrac{1}{2}\delta_1^2C_{0,T}^2\|v\|_V^2 
\end{align*} 
Thus $\delta_1\ceqq  \theta_{0,T}^{\frac12}C_{0,T}^{-1}$ suffices. 
  
Next, we show that $(A,\delta_2 B)$ satisfies the coercivity \eqref{eq:def coercive} for some $\delta_2>0$. Recall that $B=B_0+G+g$ by Assumption \ref{ass:critvarsettinglocal}. Using \eqref{eq:def coercive delta}, we have for all $v\in V,$ $t\in[0,T]$, and $\delta_2\leq (\frac23\delta)^{\frac12}$:
\begin{align*} 
   \<A(t,v),v\?-\tfrac{1}{2}\nn \delta_2 B (t,v) \nn_H^2 
  &\geq \<A(t,v),v\?-\tfrac32\delta_2^2\big(\nn  G (t,v) \nn_H^2+ \nn B_0 (t)v\nn_H^2+\nn g(t)\nn_H^2 \big) \\ 
  &\geq \theta \|v\|_V^2-M \|v\|_H^2-|{\phi}(t)|^2-\tfrac32\delta_2^2C_{0,T}^2\|v\|_V^2 -\delta \nn g(t)\nn_H^2. 
\end{align*} 
Thus, putting $\delta_2\ceqq (\frac23\delta)^{\frac12}\wedge (\frac{\theta}{3})^{\frac12}C_{0,T}^{-1}$ and ${\phi}_1(t)\ceqq  (|{\phi}(t)|^2 +\delta \nn g(t)\nn_H^2)^{\frac12}$, it follows that $(A,\delta_2 B)$ satisfies coercivity of the form \eqref{eq:def coercive}. Note that ${\phi}_1\in L^2(0,T)$ thanks to Assumption \ref{ass:critvarsettinglocal}\ref{it:AB mble}. 
  
Now, for $\hat{\delta} \ceqq \delta_1\wedge \delta_2$, $(A,\hat{\delta} B)$ satisfies the full Assumption \ref{ass:critvarsettinglocal} as well as \eqref{eq:def coercive} on $[0,T]$. 
Therefore, by Corollary \ref{cor:LDPcoerciveperturbed} applied to $(\hat{A},\hat{B})\ceqq (A,\hat{\delta}B)$, $\hat{\eta}(\eps)\ceqq \eta(\eps\hat{\delta}^2)$ and with original $\tilde{A}$, it follows that there exists $\hat{\eps}_0>0$ such that for all $\eps\in(0,\hat{\eps}_0)$, the following equation is well-posed:
\begin{equation*} 
\begin{cases}
  &\ddd \hat{Y}^\eps(t)=-A(t,\hat{Y}^\eps(t))\dd t+\eta(\eps\hat{\delta}^2)\tilde{A}(t,\hat{Y}^\eps(t))\dd t+\sqrt{\eps}\hat{\delta}B(t,\hat{Y}^\eps(t))\dd W(t), \quad t\in[0,T], \\
  &\hat{Y}^\eps(0)=x.
\end{cases}
\end{equation*}  
and the solution family $(\hat{Y}^\eps)_{\eps\in(0,\hat{\eps}_0)}$ satisfies the LDP on $L^2(0,T;V)\cap C([0,T];H)$, with rate function $\hat{I}\col L^2(0,T;V)\cap C([0,T];H)\to[0,\infty]$ given by $\hat{I}(z)=\frac{1}{2}\inf \{ \int_0^T \|\psi(s)\|_U^2\dd s : \psi\in L^2(0,T;U), \, z=\hat{u}^{\psi} \},$
where $\inf\varnothing\coloneqq +\infty$ and $\hat{u}^\psi$ is the unique strong solution to \eqref{eq:skeleton} with $B$ replaced by $\hat{B}$. 

Now, performing the substitutions $\eps\mapsto \eps\hat{\delta}^{-2}$ and $\psi\mapsto \hat{\delta}^{-1}\psi$, we conclude that \eqref{eq:SPDE perturbed} has a unique strong solution $\tilde{Y}^\eps$ on $[0,T]$ for all $\eps\in (0,\eps_0)\ceqq (0,\hat{\eps}_0\hat{\delta}^2)$ and the family $(\tilde{Y}^\eps)_{\eps\in(0,\eps_0)}$ satisfies the claimed LDP. To see that the rate function is correct, observe that  $\tilde{Y}^\eps=\hat{Y}^{\eps\hat{\delta}^{-2}}$, and note that $\hat{u}^{\hat{\delta}^{-1}\psi}=u^{\psi}$ implies that $\hat{I}(z)=\hat{\delta}^{-2}I(z)$. 
\end{proof}

\begin{remark}
The proofs of Corollaries \ref{cor:LDPcoerciveperturbed} and \ref{cor:LDPcoerciveperturbed1} are still valid if we replace the assumption that $A_0$ and $B_0$ are as in \eqref{eq:semilin}, by: $(A,B)$ satisfies Assumption \ref{ass:critvarsettinglocal} without  \ref{ass:critvarsettinglocal}\ref{it:coercivelinear} with $n$-independent constants $C_T=C_{n,T}$, and $(A_0,0)$ satisfies Assumption \ref{ass:critvarsettinglocal}\ref{it:coercivelinear}   with $n$-independent constants $(\theta_T,M_T)=(\theta_{n,T},M_{n,T})$. 
\end{remark}

\begin{remark}\label{ex:stratonovich coercive}
Corollaries \ref{cor:LDPcoerciveperturbed} and \ref{cor:LDPcoerciveperturbed1} yield LDPs for many coercive (variational) SPDEs with Stratonovich noise. 
For instance, the LDP for 2D fluid models discussed in \cite[\S5.1]{TV24} with It\^o transport noise can be extended to Stratonovich transport noise. 
A detailed application to the 2D Navier--Stokes equation with Stratonovich transport noise, on bounded and unbounded spatial domains,  is given in Subsection~\ref{ss:NS}. To the best of our knowledge, the resulting LDP is new. 
\end{remark}

\subsection{Proof of Theorem \ref{th:LDP perturbed2}}\label{ss:proof perturb}

In this subsection, we prove Theorem \ref{th:LDP perturbed2}, by making modifications to the proof of Theorem \ref{th:LDP general}. 
First, let us take into account the following remark. 
\begin{remark}\label{rem: tilde YW}
Under Assumptions \ref{ass:perturb2}\ref{it:1pert} and \ref{it:3pert}, the argument in the proof of Lemma \ref{lem:consequences ass I-III}\ref{it:gen3} ensures  unique existence of a strong solution  $\tilde{X}^\eps$ to \eqref{eq:SPDE tilted perturb}, as well as the analog of \eqref{eq:lim2}:
\begin{align} 
\lim_{\gamma\to\infty} \sup_{\eps\in(0,\eps_0)}\P(\|B(\cdot,\tilde{X}^\eps(\cdot))\|_{L^2(0,T;\UH)}>\gamma)=0.\label{eq:lim2 tilde}
\end{align}   
\end{remark}   

The next proposition will be sufficient for proving the LDP for the solutions $\tilde{Y}^\eps$ to \eqref{eq:SPDE perturbed}. 

\begin{proposition}\label{prop:stoch continuity perturb2}
   Suppose that Assumption \ref{ass:perturb2} holds and let $T,K>0$  and $x\in H$. For $\eps\in(0,\eps_0)$, let $\bphieps\col [0,T]\times \Om\to U$  be a predictable stochastic process with   $\|\bphieps\|_{L^2(0,T;U)}\leq K$ a.s., and let   $\tilde{X}^\eps$ be the strong solution to 
  \eqref{eq:SPDE tilted perturb}. 
Moreover, suppose that $u^\eps(\om)=u^{\bphieps(\om)}$ for a.e.\ $\om\in\Om$, where $u^{\bphieps(\om)}$  is the strong solution  to \eqref{eq:skeleton} with $\psi=\bphieps(\om)$. 
Then, $\tilde{X}^\eps-u^\eps\to 0$ in probability in $\MR(0,T)$ as $\eps\downarrow0$. 
\end{proposition}
\begin{proof}
  We adapt the proof of Proposition \ref{prop: X^eps weak conv} to make it suited for  $\tilde{X}^\eps$. Define $N$ by \eqref{eq: def esssup N}, which is justified by \ref{it:2pert}. 
For $\eps\in(0,\eps_0)$ and $n\in\N$,  put 
\[
\tilde{E}_{n,\eps}\coloneqq \{\|\tilde{X}^\eps\|_{\MR(0,T)}\leq n\}\cap \{\|u^{\eps}\|_{\MR(0,T)}\leq N\} 
\]
and note by \ref{it:3pert}, $\lim_{n\to\infty}\sup_{\eps\in(0,\eps_0)}\P(\tilde{E}_{n,\eps}^c)=0$. Therefore, it suffices to prove that for all large enough $n$, we have  $\lim_{\eps\downarrow 0}\P(\{\|\tilde{X}^\eps- u^{\eps}\|_{\MR(0,T)}>\delta\}\cap \tilde{E}_{n,\eps})=0$. 

Fix arbitrary $n\geq N$. We have a.s.\ 
  \begin{align*}
       \|\tilde{X}^\eps(t)-u^{\eps}(t)\|_H^2
    &=2\int_0^t\<-A(s,\tilde{X}^\eps(s))+A(s,u^{\eps}(s)),\tilde{X}^\eps(s)-u^{\eps}(s)\?\dd s\\
    &\qquad+2\int_0^t\<\big(B(s,\tilde{X}^\eps(s))-B(s,u^{\eps}(s))\big)\Psi^\eps(s),\tilde{X}^\eps(s)-u^{\eps}(s)\?\dd s\\
    &\qquad+\eps\int_0^t \nn B(s,\tilde{X}^\eps(s))\nn_H^2\dd s\\
    &\qquad+2\sqrt{\eps}\int_0^t \<\tilde{X}^\eps(s)-u^{\eps}(s),B(s,\tilde{X}^\eps(s))\dd W(s)\?\\
    &\qquad+2\eta(\eps)\int_0^t\big\<\tilde{A}(s,\tilde{X}^\eps(s)),\tilde{X}^\eps(s)-u^{\eps}(s)\big\?\dd s\\
    &\eqqcolon J_1^\eps(t)+J_2^\eps(t)+J_3^\eps(t)+J_4^\eps(t)+J_5^\eps(t).
  \end{align*} 
  On $\tilde{E}_{n,\eps}$,  $J_1^\eps(t)+J_2^\eps(t)$ can be estimated in the same way as $I_1^\eps(t)+I_2^\eps(t)$ on $E_{n,\eps}$ in \eqref{eq: to show I^1 I^2}, replacing only $X^\eps$ by $\tilde{X}^\eps$. Then, Gr\"onwall's inequality applied pointwise on $\tilde{E}_{n,\eps}$ gives 
  \[
    \P(\{\|\tilde{X}^\eps- u^{\eps}\|_{\MR(0,T)}^2>\delta\}\cap \tilde{E}_{n,\eps})\leq \sum_{i=3}^5\P(\{\sup_{t\in[0,T]}|{J}_i^\eps(t)|>c_{n,T}{\delta}\}\cap \tilde{E}_{n,\eps}) 
  \]
  for some constant $c_{n,T}>0$. 
  Moreover, 
  $J_3^\eps(t)$ and $J_4^\eps(t)$ satisfy \eqref{eq: to show I^3} and \eqref{eq: to show I^4} respectively, 
  since  Assumption \ref{ass:perturb2}\ref{it:3pert} replaces Assumption \ref{ass:coer replace}\ref{it:3}  and \eqref{eq:lim2 tilde} replaces \eqref{eq:lim2}. 
   Thus,  it suffices to prove that for our fixed $n\geq N$, we have:
\begin{equation}\label{eq:to show J^5}
 \lim_{\eps\downarrow 0}  \P\big(\big\{\sup_{t\in[0,T]}|J_5^\eps(t)|>c_{n,T}\delta\big\} \cap \tilde{E}_{n,\eps}\big)=0 \quad \text{ for all $\delta>0$. }
\end{equation} 
  
Recalling the structure of $\tilde{A}$ and conditions in Assumption \ref{ass:perturb2}\ref{it:tildeA},  we have on $\tilde{E}_{n,\eps}$:
\begin{align*}
  |J_5^\eps(t)|&\leq  |\eta(\eps)|\int_0^t \|A_1(s,\tilde{X}^\eps(s))\|_{V^*}^2+ \|\tilde{X}^\eps(s)\|_V^2\dd s\\
  &\qquad + 2|\eta(\eps)|\Big|{\sum_{i=1}^m}\int_0^t \<\tilde{F}_i(s,\tilde{X}^\eps(s)),\tilde{X}^\eps(s)-u^{\eps}(s)\big\?\dd s\Big|\\
   &\leq |\eta(\eps)|  (C_{n,T}+1)(1+\|\tilde{X}^\eps\|_{L^2(0,t;V)}^2)\\
   &\qquad+2|\eta(\eps)| {\sum_{i=1}^m}\|\tilde{F}_i(\cdot,\tilde{X}^\eps)\|_{L^{2/(1+2\alpha_i)}(0,T;V_{\alpha_i})} \|\tilde{X}^\eps-u^{\eps}\|_{C([0,T];H)}^{2\alpha_i}\big\|\tilde{X}^\eps-u^{\eps}\big\|_{L^2(0,T;V)}^{1-2\alpha_i}
  \\
  &\leq |\eta(\eps)| (C_{n,T}+1)\big(1+n^2+  2\textstyle{\sum_{i=1}^m}(1+n^{1+2\alpha_i})(2n)^{2\alpha_i}(2n)^{1-2\alpha_i}\big) 
\\ 
&\eqqc |\eta(\eps)| K_{n,T}.
\end{align*}
Above, we applied \eqref{eq:interpol est alpha pair} for $\tilde{F}_i$ and used H\"older's inequality (with $q=2/(1+2\alpha_i)$, $q'=2/(1-2\alpha_i)$) and Lemma \ref{lem: 3.8 extra F}\ref{it:emb4}. We note that these estimates rely only on the assumed growth bounds (not on Lipschitz bounds).  
Taking the supremum over  $t\in[0,T]$ in the last estimate and using that $\lim_{\eps\downarrow0}\eta(\eps)=0$ by Assumption \ref{ass:perturb2}\ref{it:eta}, we find that \eqref{eq:to show J^5} holds  as desired. 
\end{proof}
 
\begin{proof}[\textbf{Proof of Theorem \ref{th:LDP perturbed2}}] 
The rate function and   skeleton equation stated in Theorem \ref{th:LDP perturbed2} are the same as in Theorem \ref{th:LDP general}. The LDP thus follows from Proposition \ref{prop:cpt sublevel sets} and Proposition \ref{prop:stoch continuity perturb2},  
by the   arguments in the proof of Theorem \ref{th:LDP general}. 
\end{proof}

\section{Applications}\label{sec: brusselator}  

Several examples of applications of our results were already mentioned in Examples \ref{ex:applications}, \ref{ex:applications hat F} and Remarks \ref{rem:appl strat}, \ref{ex:stratonovich coercive}. 
Here, in Subsection \ref{ss:brusselator}, we work out an example of a reaction-diffusion system which is not coercive. In addition, we discuss two coercive applications in Subsection \ref{ss:appl coer}.

\subsection{The stochastic 2D Brusselator}\label{ss:brusselator}

We study well-posedness and large deviations for a reaction-diffusion system, the stochastic 2D Brusselator. The Brusselator model was introduced in \cite{prigogine68,prigogine71} to describe chemical autocatalytic reactions and the formation of spatial patterns.  
Notably, this work was a fundamental component of the research that earned Prigogine the 1977 Nobel Prize in Chemistry. 
A special case of the Brusselator model is the Gray--Scott model, see \cite[Ex.\ 4.12]{AVreac24} and the references therein. Thus our results also apply to the corresponding stochastic 2D Gray--Scott model.

Here, similar to \cite[\S 4.3]{AVreac24}, we add stochastic perturbations to the Brusselator model, which represent thermal fluctuations or small-scale turbulence. 
We consider the following system on the two-dimensional torus $\Tor$: 
\begin{equation}\label{eq:brusselator}
     \begin{cases}
   &\ddd u_1-\nabla\cdot(a_1 \nabla u_1)\dd t=[ -u_1u_2^2 +\lambda_1u_1+\lambda_2u_2+\lambda_0]\dd t\\
   &\hspace{5.3cm}+\sqrt{\eps}\sum_{n\geq 1}[(b_{n,1}\cdot\nabla)u_1+g_{n,1}(\cdot,u)]\dd w^n,   \\
   &\ddd u_2-\nabla\cdot(a_2 \nabla u_2)\dd t=  [u_1u_2^2 +\mu_1u_1+\mu_2u_2+\mu_0]\dd t\\
   &\hspace{5.3cm}+\sqrt{\eps}\sum_{n\geq 1}[(b_{n,2}\cdot\nabla)u_2+g_{n,2}(\cdot,u)]\dd w^n,   \\
   &u(0)=(u_{0,1},u_{0,2}),
\end{cases}
\end{equation}
where $(w^n)$ is a sequence of independent real-valued standard Brownian motions. Fixing an orthonormal basis for $\ell^2$, we can identify $(w^n)$ with an $\ell^2$-cylindrical Brownian motion $W=(w^n)$.

For the coefficients, we make the following assumption.  
\begin{assumption}\label{ass:brusselator}   
  \begin{enumerate}[label=\textit{(\arabic*)},ref=\ref{ass:brusselator}\textit{(\arabic*)}]
  \item\label{it:ab} For $i,j,k=1,2$: $a_i^{j,k}\col \R_+\times  \Tor\to \R$ and $b_i^j=(b_{n,i}^j)_{n\geq 1}\col \R_+\times  \Tor\to \ell^2$ are Borel measurable and bounded.   
  \item\label{it:parab} For $i=1,2$: $\nu_i>0$ and for all $t\in\R_+$, $x\in\Tor$ and $\xi\in\R^2$:  
  \[
  \sum_{j,k= 1}^2 \Big(a_i^{j,k}(t,x)-\frac{1}{2}\sum_{n\geq 1}b_{n,i}^{j}(t,x)b_{n,i}^{k}(t,x)\Big)\xi_j\xi_k\geq \nu_i|\xi|^2.
  \]
  \item\label{it:g}  
  For $i=1,2$: $g_i=(g_{n,i})_{n\geq 1}\col \R_+\times \Tor \times \R^2\to \ell^2$ is Borel measurable, 
  $g_i(\cdot,0)\in L^\infty(\R_+\times \Tor;\ell^2)$  
  and for all $t\in\R_+$, $x\in\Tor$ and $y,z\in\R^2$: 
  \[
  \|g_i(t,x,y)-g_i(t,x,z)\|_{\ell^2}\leq C(1+|y|+|z|)|y-z|.
  \]   
  \item\label{it:alphabeta} For $i=0,1,2$: $\lambda_i,\mu_i\col \R_+\times \Tor\to \R$ is  Borel measurable and bounded. 
  \item\label{it:Bgrowth} There exist $M>0$ and $\delta \in(0,\nu_1\wedge\nu_2)$ such that for $i=1,2$, the following holds on $\R_+\times\Tor$, for any fixed $y\in \R^2$: 
      \[
        \frac{1}{4(\nu_i-\delta)}\Big(\sum_{n\geq 1}|b_{n,i}(\cdot)||g_{n,i}(\cdot,y)|\Big)^2
        +\frac{1}{2}\|(g_{n,i}(\cdot,y))_{n\geq 1})\|_{\ell^2}^2\leq N_i(y),
      \]
         where for some $\epsilon\in(0,1]$: 
        \begin{align}\label{eq:Ni}
        \begin{split}
          &N_1(y)\ceqq M[1+|y_1|^2]+(1-\epsilon)|y_1|^2|y_2|^2, \\
          &N_2(y)\ceqq M[1+(1+|y_1|^2)|y_2|^2+|y_1||y_2|^3+|y_1|^4].
          \end{split}
        \end{align} 
  \end{enumerate}
\end{assumption}

\begin{remark}
Assumption \ref{ass:brusselator} is similar to \cite[Asm.\ 2.1]{AVreac24} with $h=3$. In 2D, $h=3$ is not allowed in  \cite[Asm.\ 2.2(ii)]{AVreac24}, but with the use of the flexible $\alpha$-(sub)criticality condition \eqref{eq: subcrit cond alpha} for $\hat{F}$ in Section \ref{sec:main results}, this critical case can be treated, as we will prove below.  Also, our assumption resembles \cite[Asm.\ 8.1(1)(2)(3)]{BGV} with $\rho_1=2$ and $\rho_3=1$, but the latter concerns the scalar case whereas above we have a two-component system, and the boundary conditions are different. 
\end{remark}

The abstract theory will be applied with   
  $U\ceqq \ell^2$ for the noise,  a Gelfand triple $(V,H,V^*)$ defined by 
\begin{equation}\label{eq:Gelfand}
    V=H^1(\Tor;\R^2),\qquad H= L^2(\Tor;\R^2), \qquad V^*= H^{-1}(\Tor;\R^2),
\end{equation} 
and the following coefficients:   
\begin{align}\label{eq: defs coeff}
\begin{split}
&\begin{cases}
 &A_0(t)u\ceqq[-\nabla\cdot(a_1(t,\cdot)\nabla u_1), -\nabla\cdot(a_2(t,\cdot)\nabla u_2)],\\[.2em]
 &B_0(t)u\ceqq\big[((b_{n,1}(t,\cdot)\cdot\nabla)u_1)_n, ((b_{n,2}(t,\cdot)\cdot\nabla)u_2)_n\big],\\[.2em]
 &\hat{F}(t,u)\ceqq F(t,u)+\breve{F}(u), \\[.2em]
 &G(t,u) \ceqq\big[((g_{n,1}(t,\cdot,u))_n, ((g_{n,2}(t,\cdot,u))_n\big],\\[.2em]
& \hat{f}=g=0,
\end{cases}\\& 
\begin{cases}
  &F(t,u) \ceqq[\lambda_1(t,\cdot)u_1+\lambda_2(t,\cdot)u_2
 +\lambda_0(t,\cdot),\mu_1(t,\cdot)u_1+\mu_2(t,\cdot)u_2+\mu_0(t,\cdot)],\\ 
 &\breve{F}(u) \ceqq [-u_1u_2^2, u_1u_2^2].
\end{cases}
\end{split}
\end{align}
The Gelfand triple of \eqref{eq:Gelfand} corresponds to an analytically weak setting. Note that by integration by parts and the periodic boundary conditions, we have for $u,v\in H^1(\Tor;\R^2)$:
\begin{equation}
\langle A_0(t)u,v\rangle 
=\textstyle{\sum_{i=1}^2 }\big(  (a_i(t,\cdot)\nabla u_i),\nabla v_i\big)_{L^2(\Tor;\R^2)},  
\end{equation}
where $\<\cdot,\cdot\?$ is the duality pairing between $V$ and $V^*$ induced by $H$. 

We can already see that the full system \eqref{eq:brusselator} is non-coercive in this Gelfand triple setting, due to the term $u_1u_2^2$  in the equation for $u_2$.  
Mathematically, the latter term gives rise to a super-quadratic term $u_1u_2^3$ in $\<A(u),u\?$, which cannot be controlled in the sense of \eqref{eq:def coercive}. In chemical sense, it encodes autocatalysis (positive feedback). 
Still, we will be able to prove non-blow up by exploiting the dissipative effect of the nonlinearity $-u_1u_2^2$ appearing in the equation for $u_1$.

We define  strong solutions to \eqref{eq:brusselator} analogous to Definition \ref{def:sol}, in which we let $(V,H,V^*)$ be as in  \eqref{eq:Gelfand}. We call $u\col\R_+\times\Om\to H$ a global solution if $u|_{[0,T]}$ is a strong solution on $[0,T]$ for every $T>0$. 
Our result is now as follows. To the best of our knowledge, under the stated assumptions, both the global well-posedness part and the LDP are new. 

\begin{theorem}[Global well-posedness and LDP]\label{th:brusselator}
Let Assumption \ref{ass:brusselator} hold and let $u_0\in L^2(\Tor;\R^2)$. Then for each $\eps\in[0,1]$, the stochastic 2D Brusselator \eqref{eq:brusselator} has a unique strong solution $u^\eps$, with a.s. $u^\eps\in  C([0,T];L^2(\Tor;\R^2))\cap L^2(0,T;H^1(\Tor;\R^2))\eqqc \MR(0,T)$ for every $T>0$.  

Furthermore, $(u^\eps)$ satisfies the large deviation principle on $\MR(0,T)$, and the rate function $I\col \MR(0,T)\to[0,+\infty]$ is given by 
\begin{equation*} 
I(z)=\frac{1}{2}\inf\Big\{\textstyle{\int_0^T}\|\psi(s)\|_{\ell^2}^2\dd s : \psi\in L^2(0,T;\ell^2), \, z=z^{\psi}\Big\},
\end{equation*}
where $\inf\varnothing\coloneqq +\infty$ and $z^\psi $ is the unique strong solution to the skeleton equation \eqref{eq:brusselator skeleton}. 
\end{theorem}
\begin{remark} 
 For the LDP above, it would   suffice to assume the parabolicity of Assumption \ref{it:parab} without the   contribution from the $B_0$-term. Indeed, one can then reduce to $\delta B$ for a small $\delta>0$, following the arguments at the end of \cite[\S6]{AT25} and the proof of Corollary \ref{cor:LDPcoerciveperturbed1}.   However, Assumption \ref{it:parab} as stated is crucial for global well-posedness across all $\eps\in[0,1]$.
\end{remark}
 
The result above will be derived from Theorem \ref{th:LDP general}. 
The conditions appearing in Theorem \ref{th:LDP general} will be verified across the next few lemmas. After that, the proof of Theorem \ref{th:brusselator} is given at the end of this subsection.

We begin by verifying the local well-posedness conditions. Since $A_0$ and $B_0$ are linear, we write $A_0(t)u$ and  $B_0(t)u$, rather than $A_0(t,u)u$ and $B_0(t,u)u$. 
Furthermore, in what follows, we will be using the isometric isomorphism  $\mathscr{L}_2(\ell^2,H)=\mathscr{L}_2(\ell^2,L^2(\Tor;\R^2)) \cong L^2(\Tor;\ell^2(\N;\R^2))$. 

\begin{lemma}[Local well-posedness]\label{lem:bruss loc}
Let $(V,H,V^*)$ and $(A_0,B_0,\hat{F},G,\hat{f},g)$  be as in  \eqref{eq:Gelfand} and  \eqref{eq: defs coeff} and let Assumption \ref{ass:brusselator} hold. Then Assumption \ref{ass:critvarsettinglocal} is satisfied. In particular, local well-posedness and existence of a unique maximal solution holds for \eqref{eq:brusselator} with $\eps\in[0,1]$ and for the skeleton equation \eqref{eq:brusselator skeleton}. 
\end{lemma}

\begin{proof}
The measurability conditions in Assumption \ref{ass:critvarsettinglocal} trivially follow from those in Assumption \ref{ass:brusselator}. Moreover, by Assumption \ref{it:parab}, we have
\begin{align}
\langle A_0&(t)u,u\rangle 
-\tfrac{1}{2}\|B_0(t)u\|_{\mathscr{L}_2(\ell^2,H)}^2 \notag\\
&=\sum_{i=1}^2 \big(  (a_i(t,\cdot)\nabla u_i),\nabla u_i\big)_{L^2(\Tor;\R^2)}
-\sum_{i=1}^2\tfrac{1}{2}\sum_{n\geq 1}\|(b_{n,i}(t,\cdot)\cdot\nabla)u_i\|_{L^2(\Tor)}^2   \notag\\
&=\sum_{i=1}^2\sum_{j,k=1}^{2}\bigg[\int_{\Tor}  a_{i}^{j,k}(t,x)\,\partial_j u_i(x)\,\partial_k u_i(x) -\frac{1}{2} \big( \sum_{n \geq 1} b_{n,i}^{j}(t,x)\, b_{n,i}^{k}(t,x) \big)
\, \partial_j u_i(x)\, \partial_k u_i(x) \dd x \bigg]\notag\\
&\geq \nu_1\|\nabla u_1\|_{L^2(\Tor;\R^2)}^2+\nu_2\|\nabla u_2\|_{L^2(\Tor;\R^2)}^2  \label{eq:linearcoerbrusselator}\\
&\geq (\nu_1\wedge \nu_2)\|u\|_{V}^2-(\nu_1\wedge \nu_2)\|u\|_{H}^2, \notag 
\end{align}  
so Assumption \ref{ass:critvarsettinglocal}\ref{it:coercivelinear} holds. Concerning Assumption \ref{ass:critvarsettinglocal}\ref{it:growth AB}, the growth and Lipschitz estimates for $A_0$ and $B_0$ follow  from the boundedness of $a_i$ and $(b_{n,i})_n$ (Assumption \ref{it:ab}). 

Let us now verify the growth and Lipschitz estimates for $\hat{F}$ and $G$, starting with the most difficult part of $\hat{F}$. 
For $\hat{F}_2\ceqq \breve{F}$, we first note that for all $y,z\in\R^2$:   
\begin{align*} 
|\breve{F}(y)-\breve{F}(z)|\leq 2|y_1y_2^2-z_1z_2^2|\leq 4(|y|^2+|z|^2)|y-z|,
\end{align*}
applying e.g.\  the mean value theorem to $h(y)\ceqq y_1y_2^2$ for the last inequality.  
Therefore, as in \cite[Th.\ 8.2, p.\ 53]{BGV}, we find 
\begin{align*} 
\|\breve{F}(y)-\breve{F}(z)\|_{V_{1/2}}\lesssim\|(|y|^2+|z|^2)|y-z|\|_{L^2(\Tor)}
&\lesssim(\|y\|_{L^6(\Tor;\R^2)}^2+\|z\|_{L^6(\Tor;\R^2)}^2)\|y-z\|_{L^6(\Tor;\R^2)}\\
&\lesssim (\|y\|_{V_{\hat{\beta}}}^2+\|z\|_{V_{\hat{\beta}}}^2)\|y-z\|_{V_{\hat{\beta}}},
\end{align*}
provided that $-2+2\hat{\beta}\geq -\frac13$  (for the Sobolev embedding), noting that  $V_{\hat{\beta}}=H^{-1+2\hat{\beta}}(\Tor)$. Thus we can let $\hat{\beta}_2=5/6$, and then, with $\hat{\rho}_2=2$ and $\alpha_2=1/2$, \eqref{eq: subcrit cond alpha} is satisfied (critical case).  Furthermore, $\breve{F}(0)=0$, so the growth bound follows from the Lipschitz bound. 

Next, we can estimate $G$ and  $\hat{F}_1\ceqq F$ as in \cite[\S5.3, p.\ 993, 994]{AV22variational}. 
Note that  $F(\cdot,0)=(\lambda_0, \mu_0)$ and $G(\cdot,0)$ are bounded by Assumption \ref{ass:brusselator}. Thus the required growth bound for $F$ and $G$ follow from the Lipschitz bounds derived below.  
For $\hat{F}_1= F$, fixing 
\begin{equation}\label{eq:def R}
  {R}\ceqq \max_{i=0,1,2}(\|\lambda_i\|_\infty+\|\mu_i\|_\infty),
\end{equation} 
we have   $\|{F}(t,y)-{F}(t,z)\|_{V^*} \leq {R} \|y-z\|_{H}$.  
Consequently, Assumption \ref{ass:critvarsettinglocal}\ref{it:hat F} holds with $\alpha_1=0$ therein and fixing any $\hat{\beta}_1\in(\frac12,1)$ and $\hat{\rho}_1\in(0,({2\hat{\beta}_1-1})^{-1}-1]$.        
For $G$, we have by Assumption \ref{it:g}, H\"older's inequality and a Sobolev embedding:
\begin{align*} 
\|G(t,y)-G(t,z)\|_{\mathscr{L}_2(\ell^2;H)}&\leq 2C\|(1+|y|+|z|)|y-z|\|_{L^2(\Tor)}\\
&\lesssim\|1+|y|+|z|\|_{L^4(\Tor)}\|y-z\|_{L^4(\Tor;\R^2)}\\
&\lesssim (1+\|y\|_{\beta}+\|z\|_{\beta})\|y-z\|_{\beta},
\end{align*}
provided that $-2+2\beta\geq -\tfrac12$. Choosing $\beta=3/4$ and $\rho=1$,  \eqref{eq: subcrit cond} is satisfied and the Lipschitz bound for $G$ in Assumption \ref{ass:critvarsettinglocal}\ref{it:growth AB} holds. 

The last claim of the lemma follows from Remark \ref{rem: SPDE loc well-posed} and Theorem \ref{th: local well posedness skeleton}.
\end{proof} 

Next, we prove global well-posedness, with the help of Lemma \ref{lem:suff non-blow up}. Here and in the subsequent proofs, we use the letter $C$ with subscripts to denote positive constants depending only on the parameters indicated. We write $A=[A_1, A_2]$, $B=[B_1,B_2]$ for the components of the coefficients, and similarly $A_0=[A_{0,1},A_{0,2}]$, $B_0=[B_{0,1},B_{0,2}]$, $F= [F_1,F_2]$, $\breve{F}=[\breve{F}_1,\breve{F}_2]$, $G= [G_1,G_2]$.

\begin{lemma}[Global well-posedness]\label{lem:glob bruss}
Let Assumption \ref{ass:brusselator} hold. Then  Assumption \ref{ass:coer replace}\ref{it:1} is satisfied for the two-dimensional stochastic Brusselator \eqref{eq:brusselator}. 
\end{lemma}

\begin{proof}
Let  $\eps\in[0,1]$ be arbitrary and let $(u,\sigma)$ be the maximal solution for \eqref{eq:brusselator}.  Let $T>0$ be arbitrary. 

A priori, it is unknown whether $u\in C([0,T];L^2(\Tor;\R^2))\cap L^2(0,T;H^1(\Tor;\R^2))$ a.s., so we use a localization argument to prove the necessary estimates. 
Let $(\sigma_n)$ be a localizing sequence for $(u,\sigma)$ and define for $n\in\N$: 
\[ 
\sigma_n'\ceqq \sigma_n \wedge T\wedge\inf\{t\in[0,\sigma_n]:\|u\|_{L^2(0,t;V)}+\sup_{s\in[0,t]}\|u(s)\|_H\geq n\}. 
\]  
Consider the following linear problem on $[0,T]$:
\begin{equation}\label{eq:brusselator lin ext}
     \begin{cases} 
    &\ddd u^n +A_0(\cdot)u^n  \dd t=[F(\cdot,u)\one_{[0,\sigma_n']}+\breve{F}(u)\one_{[0,\sigma_n']}]\dd t +\sqrt{\eps}[B_0(\cdot)u^n +G(\cdot,u)\one_{[0,\sigma_n']}]\dd W,   \\ 
   &u^n(0)=(u_{0,1},u_{0,2}),
\end{cases} 
\end{equation} 
Existence and uniqueness of a strong solution $u^n$ on $[0,T]$ is guaranteed by the linear result \cite[Th.\ 4.6]{BGV}. We argue that $u=u^n$ on $[0,\sigma_n']$ a.s. 
Let $\tilde w$ be the unique strong solution provided by \cite[Th.\ 4.6]{BGV} to the following equation on $[\sigma_n',T]$:
\[
     \begin{cases} 
    &\ddd \tilde w +A_0(\cdot)\tilde w  \dd t= \sqrt{\eps} B_0(\cdot)\tilde w \dd W,   \\ 
   &\tilde w(\sigma_n')=(u-u^n)(\sigma_n') 
\end{cases} 
\]
Then, $w\ceqq (u-u^n)\one_{[0,\sigma_n']}+\tilde w\one_{(\sigma_n',T]}$ solves the following equation on $[0,T]$:
\[
     \begin{cases} 
    &\ddd w +A_0(\cdot)w  \dd t= \sqrt{\eps} B_0(\cdot)w \dd W,   \\ 
   &w(0)=0.
\end{cases} 
\]
Uniqueness from \cite[Th.\ 4.6]{BGV} (with non-random terminal time $T$) yields $w=0$ a.s. Hence, recalling the definition of $w$, we conclude that  $(u-u^n)\one_{[0,\sigma_n']}=0$ a.s., and using the definition of $\sigma_n'$ and the fact that $(\sigma_n)$ is a localizing sequence, we obtain 
\begin{equation}\label{eq:lim u^n}
 \sigma_n'\uparrow\sigma\wedge T \, \text{ a.s.,  }\qquad\quad u\one_{[0,\sigma_n']}=u^n\one_{[0,\sigma_n']}\, \text{ a.s. for all  }n\in\N. 
\end{equation}

Next, we will prove energy estimates for $u^n$. 
We apply the It\^o formula separately  to each component of \eqref{eq:brusselator lin ext},  and apply the $i$-th part of the linear coercivity estimate \eqref{eq:linearcoerbrusselator}, and use that  $u\one_{[0,\sigma_n']}=u^n\one_{[0,\sigma_n']}$, to obtain  for $i=1,2$, for all $t\in[0,T]$: 
\begin{align}
  \frac12 \|u^n_i(t)\|_{L^2(\Tor)}^2&= \frac12 \|u_{0,i}\|_{L^2(\Tor)}^2+ \int_0^t\<-A_{0,i}(s)u^n(s)+\one_{[0,\sigma_n']}[{F}_i(s,u(s))+\breve{F}_i(u(s))],u^n_i(s)\?\dd s \notag\\
   &\quad+ \sqrt{\eps}\int_0^t\<u^n_i(s),B_{0,i}(s)u^n(s)+\one_{[0,\sigma_n']}G_i(s,u(s))\dd W(s)\?\notag\\
   &\qquad+\frac{\eps}{2}\int_0^t\|B_{0,i}(s)u^n(s)+\one_{[0,\sigma_n']}G_i(s,u(s))\|_{L^2(\Tor;\ell^2)}^2\dd s \notag\\
   &\leq \frac12 \|u_{0,i}\|_{L^2(\Tor)}^2- \nu_i \int_0^t\one_{[0,\sigma_n']}\|\nabla u_i \|_{L^2(\Tor;\R^2)}^2\dd s + \int_0^t\one_{[0,\sigma_n']}\<\breve{F}_i(u(s)),u_i(s)\?\dd s \notag\\ 
   &\quad+  \int_0^t\one_{[0,\sigma_n']}\<F_i(s,u(s)),u_i(s)\? \dd s+\sqrt{\eps}M_i^n(t) \notag\\
   &\qquad+\frac{\eps}{2}\int_0^t\one_{[0,\sigma_n']}\|G_i(s,u(s))\|_{L^2(\Tor;\ell^2)}^2\dd s\notag\\
   &\qquad+\eps\int_0^t\one_{[0,\sigma_n']}\sum_{k\geq 1}\int_{\Tor}g_{k,i}(s,x,u(s,x))\big((b_{k,i}(s,x)\cdot\nabla)u_i(s,x)\big)\dd x\dd s,\label{eq:Ito comp1}
\end{align} 
Here, $M_i^n$ is a continuous local martingale since $B_{0,i}(\cdot)u^n+\one_{[0,\sigma_n']}G_i(\cdot,u)\in L^2(0,T;\mathscr{L}_2(\ell^2;L^2(\Tor)))$ a.s., using that $u\one_{[0,\sigma_n']}\in L^\infty(0,T;L^2(\Tor;\R^2))\cap L^2(0,T;H^1(\Tor;\R^2))$ a.s.\ (since $\sigma_n$ is a localizing sequence) and using the estimates for $B$ in \cite[Lem.\ 3.8]{TV24}, recalling  that $B$ satisfies the growth conditions in Assumption \ref{ass:critvarsettinglocal} thanks to Lemma \ref{lem:bruss loc}. 

Now, first we turn to $i=1$, which has a  negative term  
\[
\int_0^t\one_{[0,\sigma_n']}\<\breve{F}_1(u(s)),u_1(s)\?\dd s= 
- \int_0^t\one_{[0,\sigma_n']}\| u_1 u_2\|_{L^2(\Tor)}^2 \dd s.
\]
This negative term will be used to compensate for other terms.  

Noting that $|F_i(s,y)|\leq {R}(|y_1|+|y_2|+1)$ (recall \eqref{eq:def R}) and applying Young's inequality, we have for $i=1,2$: 
\begin{align*}
\int_0^t\one_{[0,\sigma_n']}
\<F_i(s,u(s)),u_i(s)\?\dd s
&\leq R\int_0^t\one_{[0,\sigma_n']}
\Big(\|u_i\|_{L^2(\Tor)}^2+\|u_1u_2\|_{L^1(\Tor)}+\|u_i\|_{L^1(\Tor)}\Big)\dd s\\
&\leq\int_0^t\one_{[0,\sigma_n']}\Big(\frac{\epsilon}{2}\|u_1u_2\|_{L^2(\Tor)}^2
+C_{\epsilon,R}\|u_i\|_{L^2(\Tor)}^2\Big)\dd s+C_{\epsilon,R}T.
\end{align*}
Next, Young's inequality yields
\begin{align*}
  \int_0^t\one_{[0,\sigma_n']}&\sum_{k\geq 1}\int_{\Tor}g_{k,i}(s,x,u(s,x))\big((b_{k,i}(s,x)\cdot\nabla)u_i(s,x)\big)\dd x\dd s\\
  &\leq \int_0^t\one_{[0,\sigma_n']} \int_{\Tor}(\nu_i-\delta)|\nabla u_i(s,x)|^2+\frac{1}{4(\nu_i-\delta)}\Big(\sum_{k\geq 1}|g_{k,i}(s,x,u(s,x))| |b_{k,i}(s,x)|\Big)^2\dd x\dd s.
\end{align*}
Hence, by Assumption \ref{it:Bgrowth} and since $\eps\leq 1$:
\begin{align*}
&\frac{\eps}{2}\int_0^t\one_{[0,\sigma_n']}\|G_i(s,u(s))\|_{L^2(\Tor;\ell^2)}^2\dd s\\
   &\qquad+\eps\int_0^t \one_{[0,\sigma_n']}\sum_{k\geq 1}\int_{\Tor}g_{k,i}(s,x,u(s,x))\big((b_{k,i}(s,x)\cdot\nabla)u_i(s,x)\big)\dd x\dd s\\
&\leq   \int_0^t\one_{[0,\sigma_n']} \int_{\Tor} N_i(u)\dd x\dd s+(\nu_i-\delta)\int_0^t \one_{[0,\sigma_n']}\|\nabla u_i\|_{L^2(\Tor;\R^2)}^2\dd s.  
\end{align*}
Note that   the definition of $N_1$ in \eqref{eq:Ni} gives
\[
\int_0^t \one_{[0,\sigma_n']}\int_{\Tor} N_1(u)\dd x\dd s\leq MT+ \int_0^t \one_{[0,\sigma_n']}\big(M \|u_1\|_{L^2(\Tor)}^2 +(1-\epsilon)\|u_1u_2\|_{L^2(\Tor)}^2\big)\dd s.
\]
For $i=1$, combining the   estimates above with \eqref{eq:Ito comp1}, and recalling that $\one_{[0,\sigma_n']}u_1=\one_{[0,\sigma_n']}u_1^n$, yields for all $t\in[0,T]$: 
\begin{align}\label{eq:u1 est prep}
   \tfrac12 \|u_1^n(t)\|_{L^2(\Tor)}^2 &\leq \tfrac12 \|u_{0,1}\|_{L^2(\Tor)}^2 -\delta\int_0^t\one_{[0,\sigma_n']}\|\nabla u_1\|_{L^2(\Tor;\R^2)}^2\dd s-\frac{\epsilon}{2}\int_0^t\one_{[0,\sigma_n']}\| u_1 u_2\|_{L^2(\Tor)}^2 \dd s\notag\\
   &\quad+\sqrt{\eps}M_1^n(t)+ \int_0^t \one_{[0,\sigma_n']}(C_{\epsilon,R}+M) \|u_1^n\|_{L^2(\Tor)}^2 \dd s+(C_{\epsilon,R}+M)T.
\end{align}   

The stochastic Gr\"onwall inequality \cite[Cor.\ 5.4b)]{geiss24} now gives  
for all $\gamma>0$ and 
$n\in\N$: 
\begin{align*} 
\begin{split}
   \P\Big(\sup_{t\in [0,T]} \|&u_1^n(t)\|_{L^2(\Tor)}^2 +\textstyle{\int_0^T}\one_{[0,\sigma_n']}\|\nabla u_1\|_{L^2(\Tor;\R^2)}^2\dd s+\textstyle{\int_0^T}\one_{[0,\sigma_n']}\|u_1u_2\|_{L^2(\Tor)}^2\dd s>\gamma\Big)\\ 
   &\leq \gamma^{-1}C_{\delta,\epsilon}(\tfrac12\|u_{0,1}\|_{L^2(\Tor)}^2+(C_{\epsilon,R}+M)T)\exp\big[2(C_{\epsilon,R}+M)T\big], 
\end{split}
   \end{align*}
where $C_{\delta,\epsilon}\ceqq  2+ (\delta^{-1}\vee 2\epsilon^{-1})$.  Combining the estimate above with \eqref{eq:lim u^n}, 
we obtain    
\begin{align}\label{eq:u1 final estimate}
\begin{split}
&\P\big(\sup_{t\in [0,\sigma\wedge T)} \|u_1(t)\|_{L^2(\Tor)}^2+\textstyle{\int_0^{\sigma\wedge T}}\|\nabla u_1\|_{L^2(\Tor;\R^2)}^2\dd s+\textstyle{\int_0^{\sigma\wedge T}} \|u_1u_2\|_{L^2(\Tor)}^2\dd s>\gamma\big) \\
&=\lim_{n\to\infty}\P\big(\sup_{t\in [0,\sigma_n']} \|u_1^n(t)\|_{L^2(\Tor)}^2+\textstyle{\int_0^T}\one_{[0,\sigma_n']}\|\nabla u_1\|_{L^2(\Tor;\R^2)}^2\dd s+\textstyle{\int_0^T}\one_{[0,\sigma_n']}\|u_1u_2\|_{L^2(\Tor)}^2\dd s>\gamma\big) \\
&\leq \sup_{n\in\N}\P\big(\sup_{t\in [0,T]} \|u_1^n(t)\|_{L^2(\Tor)}^2+\textstyle{\int_0^T}\one_{[0,\sigma_n']}\|\nabla u_1\|_{L^2(\Tor;\R^2)}^2\dd s+\textstyle{\int_0^T}\one_{[0,\sigma_n']}\|u_1u_2\|_{L^2(\Tor)}^2\dd s>\gamma\big) \\
&\leq \gamma^{-1}C_{\delta,\epsilon}(\tfrac12\|u_{0,1}\|_{L^2(\Tor)}^2+(C_{\epsilon,R}+M)T)\exp\big[2(C_{\epsilon,R}+M)T \big]\eqqc\gamma^{-1} K_{1,T}.
\end{split}
\end{align}

We proceed by estimating the terms for $u_2$. For the  $\breve{F}_2$-term, we have by H\"older's inequality,  $H^{1/2}(\Tor)\into L^4(\Tor)$ and interpolation:
\begin{align}
\one_{[0,\sigma_n']}\<\breve{F}_2(u(s)),u_2(s)\?
&\leq \one_{[0,\sigma_n']}\int_{\Tor}|u_1 u_2 ^3|\dd x\notag\\
&\leq \one_{[0,\sigma_n']}\|u_1u_2\|_{L^2(\Tor)}\|u_2\|_{L^4(\Tor)}^2\notag\\ 
&\leq \one_{[0,\sigma_n']}\Big(C_\kappa\|u_1u_2\|_{L^2(\Tor)}^2\|u_2\|_{L^2(\Tor)}^2+\kappa\|\nabla u_2\|_{L^2(\Tor;\R^2)}^2+\kappa\|u_2\|_{L^2(\Tor)}^2\Big).  \label{eq:u1u2^3}
\end{align}  

Furthermore, the definition of $N_2$ in \eqref{eq:Ni},  gives
\begin{align*}
\int_0^t \one_{[0,\sigma_n']}&\int_{\Tor} N_2(u)\dd x\dd s \\
   &  \leq MT+ \int_0^t \one_{[0,\sigma_n']} M\Big(\|u_2\|_{L^2(\Tor)}^2+\|u_1u_2\|_{L^2(\Tor)}^2+\|u_1u_2^3\|_{L^1(\Tor)}+\|u_1\|_{L^4(\Tor)}^4\Big)\dd s\\
&\leq  MT+\int_0^t \one_{[0,\sigma_n']}M\Big(\|u_2\|_{L^2(\Tor)}^2+\|u_1u_2\|_{L^2(\Tor)}^2+C_\kappa\|u_2\|_{L^2(\Tor)}^2\|u_1u_2\|_{L^2(\Tor)}^2 \\
   &\qquad\qquad\qquad\qquad\qquad\qquad\qquad\quad +\kappa\|\nabla u_2 \|_{L^2(\Tor;\R^2)}^2+\kappa\|u_2\|_{L^2(\Tor)}^2+\|u_1\|_{L^4(\Tor)}^4\Big)\dd s\\
&\leq  MT+ \int_0^t \one_{[0,\sigma_n']}M\Big((1+\kappa)\|u_2\|_{L^2(\Tor)}^2+(1+C_\kappa\|u_2\|_{L^2(\Tor)}^2)\|u_1u_2\|_{L^2(\Tor)}^2\\
   &\qquad\qquad\qquad\qquad\qquad\qquad\qquad\quad+\kappa\|\nabla u_2 \|_{L^2(\Tor;\R^2)}^2+C\|u_1\|_{L^2(\Tor)}^2\| u_1\|_{H^1(\Tor )}^2\Big)\dd s.
\end{align*} 
where we used \eqref{eq:u1u2^3} to estimate  $\|u_1u_2^3\|_{L^1(\Tor)}$ in the second inequality, and used   $H^{1/2}(\Tor)\into L^4(\Tor)$ and interpolation in the last inequality. 
 
Collecting the   estimates above and applying \eqref{eq:Ito comp1} for $i=2$, and recalling that $u_2\one_{[0,\sigma_n']}=u_2^n\one_{[0,\sigma_n']}$, 
we obtain for all $t\in[0,T]$ a.s.: 
\begin{align}\label{eq:u2 est prep0}
  \tfrac12 \|u_2^n(t)\|_{L^2(\Tor)}^2  &\leq \tfrac12 \|u_{0,2}\|_{L^2(\Tor)}^2 -(\nu_2-(\nu_2-\delta)-\kappa-M\kappa)\int_0^t\one_{[0,\sigma_n']}\|\nabla u_2\|_{L^2(\Tor;\R^2)}^2\dd s\notag\\
   &\;+\sqrt{\eps}M_2^n(t)+ \int_0^t \one_{[0,\sigma_n']}(C_{\epsilon,{R},M,\kappa}+C_{\kappa,M}\|u_1u_2\|_{L^2(\Tor)}^2 ) \|u_2^n\|_{L^2(\Tor)}^2 \dd s\notag\\
   &\;+ \int_0^t\one_{[0,\sigma_n']}((\tfrac{\epsilon}{2}+M)\|u_1u_2\|_{L^2(\Tor)}^2+MC\|u_1\|_{L^2(\Tor)}^2\| u_1\|_{H^1(\Tor)}^2)\dd s+ (C_{\epsilon,R}+M)T.
\end{align}
Now fix $\kappa=\delta/(2+2M)$. 
Then the stochastic Gr\"onwall  inequality from \cite[Cor.\ 5.4b)]{geiss24}  gives for all $\gamma,\Gamma,w>0$ and $n\in\N$: 
\begin{align} \label{eq:u2 est prep} 
\begin{split}
    \P\Big(\sup_{t\in [0,T]} & \|u_2^n(t)\|_{L^2(\Tor)}^2+\textstyle{\int_0^T}\one_{[0,\sigma_n']}\|\nabla u_2\|_{L^2(\Tor;\R^2)}^2\dd s>\gamma\Big)  \\
   &\leq 
   \gamma^{-1}C_{\delta}\e^{\Gamma}\mathbb{E}\Big[w\wedge \big(\tfrac12\|u_{0,2}\|_{L^2(\Tor)}^2+\textstyle{\int_0^T}h_n(s)\dd s+(C_{\epsilon,R}+M)T\big)\Big] \\
   &\qquad +\P\Big(\tfrac12\|u_{0,2}\|_{L^2(\Tor)}^2+\textstyle{\int_0^T}h_n(s)\dd s+(C_{\epsilon,R}+M)T\geq w\Big) \\
   &\qquad +\P\Big(\textstyle{\int_0^T}\one_{[0,\sigma_n']}2\big(C_{\epsilon,{R},M,\kappa}+C_{\kappa,M}\|u_1u_2\|_{L^2(\Tor)}^2 \big) \dd s>\Gamma\Big), 
   \end{split}
\end{align} 
where $C_{\delta}\ceqq 2+ 2\delta^{-1}$ and   
$h_n\ceqq  \one_{[0,\sigma_n']}((\tfrac{\epsilon}{2}+M)\|u_1u_2\|_{L^2(\Tor)}^2+MC\|u_1\|_{L^2(\Tor)}^2\| u_1\|_{H^1(\Tor)}^2).$  

Note that \eqref{eq:u1 final estimate} implies an estimate for not only the first, but also  the second  term in $h_n$:  
\begin{align*}
&\P\big(\textstyle{\int_0^T}\one_{[0,\sigma_n']}\|u_1\|_{L^2(\Tor)}^2\| u_1\|_{H^1(\Tor )}^2\dd s>\gamma\big)\\
&\leq \P\big(\sup_{t\in[0,\sigma\wedge T)}\|u_1(t)\|_{L^2(\Tor)}^2 >\gamma^{1/2}\big)+\P\big(\textstyle{\int_0^{\sigma\wedge T}}\| u_1\|_{H^1(\Tor)}^2 \dd s>\gamma^{1/2}\big)\\
&\leq \P\big(\sup_{t\in[0,\sigma\wedge T)}\|u_1(t)\|_{L^2(\Tor)}^2 >\gamma^{1/2}\big)+\P\big(T\sup_{t\in[0,\sigma\wedge T)}\|u_1(t)\|_{L^2(\Tor)}^2 +\textstyle{\int_0^{\sigma\wedge T}}\|\nabla u_1\|_{L^2(\Tor;\R^2)}^2\dd s>\gamma^{1/2}\big)\\
&\leq \gamma^{-1/2}(1+(T\vee 1))K_{1,T}.    
\end{align*}  
In particular, applying \eqref{eq:u2 est prep} with $\Gamma=\tfrac14\log(\gamma)$ and $w=\gamma^{1/4}$, and combining with \eqref{eq:u1 final estimate}, we conclude that 
\[
\lim_{\gamma\to\infty}\sup_{n\in\N}\P\big(\sup_{t\in [0,T]} \|u_2^n(t)\|_{L^2(\Tor)}^2 +\textstyle{\int_0^T}\one_{[0,\sigma_n']}\|\nabla u_2\|_{L^2(\Tor;\R^2)}^2\dd s>\gamma\big)=0. 
\]
Then, \eqref{eq:lim u^n} and the limit above imply
\begin{align*}
&\lim_{\gamma\to\infty}\P\big(\sup_{t\in [0,\sigma\wedge T)} \|u_2(t)\|_{L^2(\Tor)}^2 +\textstyle{\int_0^{\sigma\wedge T}}\|\nabla u_2\|_{L^2(\Tor;\R^2)}^2\dd s>\gamma\big)\\
&=\lim_{\gamma\to\infty}\lim_{n\to\infty}\P\big(\sup_{t\in [0,\sigma_n']} \|u_2^n(t)\|_{L^2(\Tor)}^2 +\textstyle{\int_0^T}\one_{[0,\sigma_n']}\|\nabla u_2\|_{L^2(\Tor;\R^2)}^2\dd s>\gamma\big)\\
&\leq \lim_{\gamma\to\infty}\sup_{n\in\N}\P\big(\sup_{t\in [0,T]} \|u_2^n(t)\|_{L^2(\Tor)}^2 +\textstyle{\int_0^T}\one_{[0,\sigma_n']}\|\nabla u_2\|_{L^2(\Tor;\R^2)}^2\dd s>\gamma\big)=0.
\end{align*}  
Finally, combining the latter with \eqref{eq:u1 final estimate}, we conclude that for all  $T>0$:  
  \[
     \sup_{t\in[0,\sigma\wedge T)}\|u(t)\|_{L^2(\Tor;\R^2)}+\textstyle{\int_0^{\sigma\wedge T}}\|u\|_{H^1(\Tor;\R^2)}^2\dd s <\infty \quad {\text{ a.s.}}
  \] 
  Therefore,   Lemma \ref{lem:suff non-blow up} yields   $\sigma=\infty$, i.e.\  global well-posedness holds. 
\end{proof}

Next, we  establish the a priori estimate of Assumption \ref{ass:coer replace}\ref{it:2}. 
For $\psi=(\psi^n)\in L^2(0,T;\ell^2)$, the   skeleton equation for the stochastic Brusselator is given by 
\begin{equation}\label{eq:brusselator skeleton}
     \begin{cases}
   & u_1'-\nabla\cdot(a_1\cdot\nabla u_1) =[ -u_1u_2^2 +\lambda_1u_1+\lambda_2u_2+\lambda_0]  \\
   &\hspace{5cm} +\sum_{n\geq 1}[(b_{n,1}\cdot\nabla)u_1+g_{n,1}(\cdot,u)]  \psi^n,   \\
   & u_2'-\nabla\cdot(a_2\cdot\nabla u_2) =  [u_1u_2^2 +\mu_1u_1+\mu_2u_2+\mu_0] \\
   &\hspace{5cm}+\sum_{n\geq 1}[(b_{n,2}\cdot\nabla)u_2+g_{n,2}(\cdot,u)]  \psi^n,   \\
   &u(0)=(u_{0,1},u_{0,2}). 
\end{cases}
\end{equation}

\begin{lemma}[Skeleton equation bound]\label{lem:skeletonbrusselator}
Let $u_0=(u_{0,1},u_{0,2})\in  L^2(\Tor;\R^2)$, let $T>0$ and let $\psi\in L^2(0,T;\ell^2)$. Suppose that $u$ is a strong solution to the  skeleton equation \eqref{eq:brusselator skeleton} on $[0,T]$. Then:  
\begin{equation}\label{eq:apriori bruss skel}
   \|u\|_{L^2(0,T;H^1(\Tor;\R^2))}+\|u\|_{C([0,T]; L^2(\Tor;\R^2))}   \leq C_{(u_{0,1},u_{0,2})}({T},\|\psi\|_{L^2(0,{T};\ell^2)}),
  \end{equation}
   with $C_{(u_{0,1},u_{0,2})}\col \R_+\times \R_+\to\R_+$ a function that is non-decreasing in both components.   
\end{lemma}

 \begin{proof} 
 By the chain rule \cite[(A.2)]{TV24} and Assumption \ref{ass:critvarsettinglocal}\ref{it:coercivelinear}, we have  for $i=1,2$ and $t\in[0,T]$: 
\begin{align*}
  \frac12 \|u_i(t)\|_{L^2(\Tor)}^2&\leq \frac12 \|u_{0,i}\|_{L^2(\Tor)}^2- \int_0^t\<A_i(s,u(s)),u_i(s)\?\dd s +  \int_0^t\<B_i(s,u(s))\psi(s),u_i(s)\?\dd s\notag\\
  &\leq \frac12 \|u_{0,i}\|_{L^2(\Tor)}^2- \int_0^t\<A_i(s,u(s)),u_i(s)\?\dd s \\
  &\qquad+  \frac12\int_0^t\|B_i(s,u(s))\|_{L^2(\Tor;\ell^2)}^2+\|\psi(s)\|_{\ell^2}^2\|u_i(s)\|_{L^2(\Tor)}^2\dd s \notag\\ 
   &\leq \frac12 \|u_{0,i}\|_{L^2(\Tor)}^2- \nu_i \int_0^t\|\nabla u_i \|_{L^2(\Tor;\R^2)}^2\dd s + \int_0^t\<\breve{F}_i(u(s)),u_i(s)\?\dd s \notag\\ 
   &\qquad+  \int_0^t\<{F}_i(s,u(s)),u_i(s)\? \dd s +\frac{1}{2}\int_0^t\|G_i(s,u(s))\|_{L^2(\Tor;\ell^2)}^2\dd s\notag\\
   &\qquad+ \int_0^t\sum_{k\geq 1}\int_{\Tor}g_{k,i}(s,x,u(s,x))\big((b_{k,i}(s,x)\cdot\nabla)u_i(s,x)\big)\dd x\dd s\notag\\
   &\qquad +\frac12\int_0^t \|\psi(s)\|_{\ell^2}^2\|u_i(s)\|_{L^2(\Tor)}^2\dd s.  
\end{align*}
Now, except for the last term, pointwise estimates were already provided in the proof of Lemma \ref{lem:glob bruss} (apply the estimates with $\sigma_n=\sigma_n'=T$). Reusing these estimates, we readily find
\begin{align*} 
\tfrac12 \|u_1(t)\|_{L^2(\Tor)}^2 &\leq \tfrac12 \|u_{0,1}\|_{L^2(\Tor)}^2 -\delta\int_0^t\|\nabla u_1\|_{L^2(\Tor;\R^2)}^2\dd s-\frac{\epsilon}{2}\int_0^t\| u_1 u_2\|_{L^2(\Tor)}^2 \dd s\\
   &\quad+ \int_0^t (C_{\epsilon,R}+M+\tfrac12\|\psi\|_{\ell^2}^2) \|u_1\|_{L^2(\Tor)}^2 \dd s+(C_{\epsilon,R}+M)T, 
\end{align*}
and (again fixing $\kappa=\delta/(2+2M)$) 
 \begin{align*}
  \tfrac12 \|u_2(t)\|_{L^2(\Tor)}^2  &\leq \tfrac12 \|u_{0,2}\|_{L^2(\Tor)}^2 -\tfrac{\delta}{2}\int_0^t\|\nabla u_2\|_{L^2(\Tor;\R^2)}^2\dd s \\
   &\;+ \int_0^t (C_{\epsilon,{R},M,\kappa}+C_{\kappa,M}\|u_1u_2\|_{L^2(\Tor)}^2 +\tfrac12\|\psi\|_{\ell^2}^2) \|u_2\|_{L^2(\Tor )}^2 \dd s\\
   &\;+ \int_0^t (\tfrac{\epsilon}{2}+M)\|u_1u_2\|_{L^2(\Tor)}^2+MC\|u_1\|_{L^2(\Tor)}^2\| u_1\|_{H^1(\Tor )}^2\dd s+ (C_{\epsilon,R}+M)T.
\end{align*} 
Therefore, Gr\"onwall's inequality first gives  
\begin{align*} 
\sup_{t\in[0,T]}&\|u_1(t)\|_{L^2(\Tor)}^2  +\int_0^T \|\nabla u_1\|_{L^2(\Tor;\R^2)}^2\dd s+\int_0^T\| u_1 u_2\|_{L^2(\Tor)}^2 \dd s \\
&\leq C_{\delta,\epsilon}\big(\tfrac12\|u_{0,1}\|_{L^2(\Tor)}^2+(C_{\epsilon,R}+M)T\big)\exp\big[2(C_{\epsilon,R}+M)T+ \|\psi\|_{L^2(0,T;\ell^2)}^2\big] \\
&\eqqc C^{1}(T,\|\psi\|_{L^2(0,T;\ell^2)}^2)
\end{align*} 
with $C_{\delta,\epsilon}\ceqq 2 +(\delta^{-1}\vee 2\epsilon^{-1})$. Now, using \emph{both} of the last two estimates, we obtain also   
\begin{align*} 
\sup_{t\in[0,T]}&\|u_2(t)\|_{L^2(\Tor)}^2  +\int_0^T\|\nabla u_2\|_{L^2(\Tor;\R^2)}^2 \dd s\\
&\leq C_{\delta}\big[\tfrac12\|u_{0,2}\|_{L^2(\Tor)}^2+ (\tfrac{\epsilon}{2}+M)C^{1}(T,\|\psi\|_{L^2(0,T;\ell^2)}^2) + MC (C^{1}(T,\|\psi\|_{L^2(0,T;\ell^2)}^2))^2(1+T)\\
&\qquad \qquad
+(C_{\epsilon,R}+M)T\big]\cdot\exp\big[2C_{\epsilon,{R},M,\kappa}T+2C_{\kappa,M}C^{1}(T,\|\psi\|_{L^2(0,T;\ell^2)}^2)+ \|\psi\|_{L^2(0,T;\ell^2)}^2\big]\\
&\eqqc C^{2}(T,\|\psi\|_{L^2(0,T;\ell^2)}^2)  
\end{align*}
with $C_{\delta}\ceqq 2+ 2\delta^{-1}$. 

Thus,  $C_{(u_{0,1},u_{0,2})}(y,z)$ can be chosen of the form $C(C^1(y,z)+C^2(y,z))^{1/2}$, which is non-decreasing in both $y$ and $z$ (see the formulas above), and such that \eqref{eq:apriori bruss skel} is satisfied.  
 \end{proof}

Finally, we need to establish the boundedness in probability of Assumption \ref{ass:coer replace}\ref{it:3}. 
For  $\bphieps$   a predictable stochastic process taking a.s.\ values in $L^2(0,T;\ell^2)$, the stochastic control problem \eqref{eq:SPDE tilted} is in this case: 
 \begin{equation}\label{eq:brusselator tilt}
     \begin{cases}
    &\ddd u_1-\nabla\cdot(a_1\cdot\nabla u_1) \dd t =[ -u_1u_2^2 +\lambda_1u_1+\lambda_2u_2+\lambda_0]\dd t  \\
    &\hspace{4.2cm}+\sum_{n\geq 1}[(b_{n,1}\cdot\nabla)u_1+g_{n,1}(\cdot,u)]  \psi^n \dd t  \\
   &\hspace{5cm}+\sqrt{\eps}\sum_{n\geq 1}[(b_{n,1}\cdot\nabla)u_1+g_{n,1}(\cdot,u)]\dd w^n,\\
    &\ddd u_2-\nabla\cdot(a_2\cdot\nabla u_2)\dd t  =  [u_1u_2^2 +\mu_1u_1+\mu_2u_2+\mu_0]\dd t \\
    &\hspace{4.2cm} +\sum_{n\geq 1}[(b_{n,2}\cdot\nabla)u_2+g_{n,2}(\cdot,u)]  \psi^n \dd t \\
   & \hspace{5cm}  +\sqrt{\eps}\sum_{n\geq 1}[(b_{n,2}\cdot\nabla)u_2+g_{n,2}(\cdot,u)]\dd w^n,\\
   &u(0)=(u_{0,1},u_{0,2}).
\end{cases}
\end{equation}

\begin{lemma}[Stochastic control problem bound]\label{lem:tiltbrusselator}
Let $u_0=(u_{0,1},u_{0,2})\in  L^2(\Tor;\R^2)$ and let $T>0$. 
Suppose that $(\bphieps)_{\eps>0}$ is a collection of predictable stochastic processes and suppose that $K\geq 0$ is such that for each $\eps>0$, 
\begin{equation}\label{eq:less K} 
    \|\bphieps\|_{L^2(0,T;\ell^2)}\leq K \quad   a.s.  
\end{equation}
Let $u^\eps$ be the strong solution to \eqref{eq:brusselator tilt} on $[0,T]$ corresponding to $\bphieps$. Then it holds that 
\begin{equation}\label{eq:lim tilt bruss} 
\lim_{\gamma\to\infty}\sup_{\eps\in(0,1/2)}\P(\|u^\eps\|_{L^2(0,T;H^1(\Tor;\R^2))}+\|u^\eps\|_{C([0,T]; L^2(\Tor;\R^2))} >\gamma)=0.
\end{equation}
\end{lemma}
 
 \begin{proof}
   By the It\^o formula \cite[(A.4)]{TV24} and Assumption \ref{ass:critvarsettinglocal}\ref{it:coercivelinear}, we have  for $i=1,2$ and for all $\eps\in(0,1/2)$ and $t\in[0,T]$: 
\begin{align*}
  \frac12 \|u_i^\eps(t)\|_{L^2(\Tor)}^2&\leq \frac12 \|u_{0,i} \|_{L^2(\Tor)}^2- \int_0^t\<A_i(s,u^\eps(s)),u_i^\eps(s)\?\dd s +  \int_0^t\<B_i(s,u^\eps(s))\bphieps(s),u_i^\eps(s)\?\dd s\notag\\
  &\qquad+ \sqrt{\eps}\int_0^t\<u_i^\eps(s),B_i(s,u^\eps(s))\dd W(s)\? +\frac{\eps}{2}\int_0^t\|B_i(s,u^\eps(s))\|_{L^2(\Tor;\ell^2)}^2\dd s \notag\\
  &\leq \frac12 \|u_{0,i}\|_{L^2(\Tor)}^2- \int_0^t\<A_i(s,u^\eps(s)),u_i^\eps(s)\?\dd s \\
  &\qquad+ \int_0^t \frac14\|B_i(s,u^\eps(s))\|_{L^2(\Tor;\ell^2)}^2+\|\bphieps(s)\|_{\ell^2}^2\|u_i^\eps(s)\|_{L^2(\Tor)}^2\dd s \notag\\ 
  &\qquad+ \sqrt{\eps} \int_0^t\<u_i^\eps(s),B_i(s,u^\eps(s))\dd W(s)\? +\frac{1}{4}\int_0^t\|B_i(s,u^\eps(s))\|_{L^2(\Tor;\ell^2)}^2\dd s \notag\\
   &\leq \frac12 \|u_{0,i} \|_{L^2(\Tor)}^2- \nu_i \int_0^t\|\nabla u_i^\eps \|_{L^2(\Tor;\R^2)}^2\dd s + \int_0^t\<\breve{F}_i(u^\eps(s)),u_i^\eps(s)\?\dd s \notag\\ 
   &\qquad+  \int_0^t\<{F}_i(s,u^\eps(s)),u_i^\eps(s)\? \dd s +\frac{1}{2}\int_0^t\|G_i(s,u^\eps(s))\|_{L^2(\Tor;\ell^2)}^2\dd s\notag\\
   &\qquad+ \int_0^t\sum_{k\geq 1}\int_{\Tor}g_{k,i}(s,x,u^\eps(s,x))\big((b_{k,i}(s,x)\cdot\nabla)u_i^\eps(s,x)\big)\dd x\dd s\notag\\
   &\qquad + \int_0^t \|\bphieps(s)\|_{\ell^2}^2\|u_i^\eps(s)\|_{L^2(\Tor)}^2\dd s +  M_i^\eps(t),  
\end{align*} 
where $M_i^\eps$ is a continuous local martingale.  
Now we can again use the pointwise estimates from the proof of Lemma \ref{lem:glob bruss} (with $\sigma_n=\sigma_n'=T$), noting that only the  second-last term is new. This readily gives (see \eqref{eq:u1 est prep})
\begin{align*}
   \tfrac12 \|u_1^\eps(t)\|_{L^2(\Tor)}^2 &\leq \tfrac12 \|u_{0,1}\|_{L^2(\Tor)}^2 -\delta\int_0^t\|\nabla u_1^\eps\|_{L^2(\Tor;\R^2)}^2\dd s-\frac{\epsilon}{2}\int_0^t\| u_1^\eps u_2^\eps\|_{L^2(\Tor)}^2 \dd s\\
   &\quad+ M_1^\eps(t)+ \int_0^t (C_{\epsilon,R}+M+\|\bphieps\|_{\ell^2}^2) \|u_1^\eps\|_{L^2(\Tor)}^2 \dd s+(C_{\epsilon,R}+M)T.
\end{align*}  
The stochastic Gr\"onwall inequality \cite[Cor.\ 5.4b)]{geiss24} implies for all $\gamma>0$:    
\begin{align}\label{eq:u1 est tilt}
\begin{split}
   \P\Big(\sup_{t\in [0,T]}&\|u_1^\eps(t)\|_{L^2(\Tor)}^2 +\textstyle{\int_0^T} \|\nabla u_1^\eps\|_{L^2(\Tor;\R^2)}^2\dd s+\textstyle{\int_0^T} \|u_1^\eps u_2^\eps\|_{L^2(\Tor)}^2\dd s>\gamma\Big)\\ 
   &\leq \gamma^{-1}C_{\delta,\epsilon}\big(\tfrac12\|u_{0,1}\|_{L^2(\Tor)}^2+(C_{\epsilon,R}+M)T\big)\exp[2((C_{\epsilon,R}+M)T+K^2)],  
\end{split}
   \end{align} 
where $C_{\delta,\epsilon}\ceqq  2+ (\delta^{-1}\vee 2\epsilon^{-1})$. 
 
Moreover, again using the pointwise estimates from the proof of Lemma \ref{lem:glob bruss} and fixing $\kappa=\delta/(2+2M)$ therein,  we find that  (see \eqref{eq:u2 est prep0})
\begin{align*}
  \tfrac12 \|u_2^\eps(t)\|_{L^2(\Tor)}^2  &\leq \tfrac12 \|u_{0,2}\|_{L^2(\Tor)}^2 -\tfrac{\delta}{2}\int_0^t\|\nabla u_2^\eps\|_{L^2(\Tor;\R^2)}^2\dd s+ M_2^\eps(t)\\
   &\quad+ \int_0^t (C_{\epsilon,{R},M,\kappa}+C_{\kappa,M}\|u_1^\eps u_2^\eps\|_{L^2(\Tor)}^2 +\|\bphieps\|_{\ell^2}^2) \|u_2^\eps\|_{L^2(\Tor )}^2 \dd s\\
   &\quad+ \int_0^t(\tfrac{\epsilon}{2}+M)\|u_1^\eps u_2^\eps\|_{L^2(\Tor)}^2+MC\|u_1^\eps\|_{L^2(\Tor)}^2\|u_1^\eps\|_{H^1(\Tor)}^2\dd s+ (C_{\epsilon,R}+M)T.
\end{align*} 
By the stochastic Gr\"onwall's inequality \cite[Cor.\ 5.4b)]{geiss24} and \eqref{eq:less K}, we obtain for all $\gamma,\Gamma,w>0$:
\begin{align}\label{eq:u2 est tilt}
\begin{split}
\P\Big(\sup_{t\in [0,T]} &\|u_2^\eps(t)\|_{L^2(\Tor)}^2 +\textstyle{\int_0^{T}}\|\nabla u_2^\eps\|_{L^2(\Tor;\R^2)}^2\dd s>\gamma\Big)\\
&\leq  \gamma^{-1}C_{\delta}\e^{\Gamma}\mathbb{E}\Big[w\wedge \big( \tfrac12\|u_{0,2}\|_{L^2(\Tor)}^2+ \textstyle{\int_0^{T}}h^\eps(s)\dd s+(C_{\epsilon,R}+M)T \big)\Big]\\
&\qquad+\P\Big(\tfrac12\|u_{0,2}\|_{L^2(\Tor)}^2+ \textstyle{\int_0^{T}}h^\eps(s)\dd s+(C_{\epsilon,R}+M)T \geq w\Big)\\
&\qquad+\P\Big(2\big(C_{\epsilon,{R},M,\kappa}T+\textstyle{\int_0^{T}}C_{\kappa,M}\|u_1^\eps u_2^\eps\|_{L^2(\Tor)}^2  \dd s+K^2\big)>\Gamma\Big), 
\end{split}
\end{align}
where $C_\delta\ceqq 2+2\delta^{-1}$ and $h^{\eps}\ceqq (\tfrac{\epsilon}{2}+M)\|u_1^\eps u_2^\eps\|_{L^2(\Tor)}^2+MC \|u_1^\eps\|_{L^2(\Tor)}^2\|u_1^\eps\|_{H^1(\Tor)}^2$. 

Using the arguments from the end of the proof of Lemma \ref{lem:glob bruss}, the final estimates in \eqref{eq:u1 est tilt} and \eqref{eq:u2 est tilt} can be estimated uniformly in  $\eps\in(0,1/2)$, noting that the $u_1^\eps$-bound controls both terms in $h^\eps$. Thus, as before, combining these two estimates, taking the supremum over $\eps\in(0,1/2)$  and choosing $\Gamma=\tfrac14\log(\gamma)$ and $w=\gamma^{\frac14}$, we conclude that \eqref{eq:lim tilt bruss} holds.
 \end{proof}
 
\begin{proof}[Proof of Theorem \ref{th:brusselator}]
Assumption \ref{ass:critvarsettinglocal} is satisfied due to Lemma \ref{lem:bruss loc},  Assumption \ref{ass:coer replace}\ref{it:1} holds by Lemma \ref{lem:glob bruss}, Assumption \ref{ass:coer replace}\ref{it:2} holds by Lemma  \ref{lem:skeletonbrusselator} and Assumption \ref{ass:coer replace}\ref{it:3} holds by Lemma \ref{lem:tiltbrusselator}, putting $\eps_0=1/2$. Thus we can apply Theorem \ref{th:LDP general}. 
\end{proof}

\subsection{Coercive applications}\label{ss:appl coer}

\subsubsection{\textbf{2D Allen--Cahn equation with transport and quadratic noise}}\label{ss:allen}
First, let us apply Corollary \ref{cor:LDPcoercive} to an example from the extended critical variational setting \cite[Ex.\ 8.4]{BGV}: the stochastic 2D Allen--Cahn equation in the analytically weak setting. 
To the best of our knowledge, the LDP was not available for this example in 2D until now. Here, we  also include transport noise. 

Consider the following Allen--Cahn equation on a bounded $C^1$-domain $\O\subset \R^2$. 
\begin{equation}\label{eq:allen}
     \begin{cases}
   &\ddd {u}-\Delta {u}\dd t=[ {u}-{u}^3]\dd t+\sqrt{\eps}\sum_{n\geq 1}[(b_n\cdot\nabla){u}+g_n({u})]\dd w^n,   \\ 
   &{u}=0 \textup{ on }\partial\O,\\
   &{u}(0)=u_0,
\end{cases}
\end{equation}
where $(w^n)$ is a sequence of independent real-valued standard Brownian motions. 
We use the Gelfand triple $(V,H,V^*)=(H^1_0(\O),L^2(\O),H^{-1}(\O))$ and define strong solutions according to Definition \ref{def:sol}. Moreover, we assume only the following.

\begin{assumption}\label{ass:allen}

\noindent
\begin{itemize}
\item $(b_n)\in \ell^2$ and $1-\frac12\|(b_n)_n\|_{\ell^2}^2>0$ 
  \item There exist $C_0\geq 0$ and $C_1\in[0,2]$ such that for all $y,z\in \R$: 
  
  \noindent
  $\|g(y)-g(z)\|_{\ell^2} \leq C_0(1+|y|+|z|)|y-z|$ and $\|g(y)\|_{\ell^2}^2\leq C_0+C_1|y|^4$ 
  \end{itemize}
\end{assumption}
Due to \cite[Ex.\ 8.3, Ex.\ 8.4]{BGV}, and noting that the coercivity condition \eqref{eq:def coercive}  (corresponding to $\eta=0$ in \cite{BGV}) allows to include also the endpoint case $C_1=2$, we obtain the following result directly from Corollary \ref{cor:LDPcoercive}. We note that this application relies on the  flexible $\alpha$-(sub)criticality condition \eqref{eq: subcrit cond alpha} for $\hat{F}$, which was not yet included in \cite{TV24} for non-zero $\alpha_i$. 

\begin{theorem}\label{th:allen}
Let Assumption \ref{ass:allen} hold and let $u_0\in L^2(\O)$. 
For all $\eps\in[0,1]$, there exists a unique strong solution $u^\eps$ to \eqref{eq:allen}, and the family $(u^\eps)_{\eps\in(0,\frac12)}$ satisfies the LDP on $\MR(0,T)\ceqq C([0,T];L^2(\O))\cap L^2(0,T;H^1_0(\O))$ with rate function $I\col \MR(0,T)\to[0,\infty]$ given by 
\[
I(z)=\frac12\inf\{\|\psi\|_{L^2(0,T;\ell^2)}^2:\psi \in L^2(0,T;\ell^2), z=u^{\psi} \},
\]
where $\inf\varnothing\ceqq+\infty$ and for $\psi=(\psi^n)\in L^2(0,T;\ell^2)$,  $u^\psi\in  \MR(0,T)$ denotes the unique strong solution to
\begin{equation*} 
     \begin{cases}
   &(u^{\psi})'-\Delta u^{\psi} =[ u^{\psi}-(u^{\psi})^3] +\sum_{n\geq 1}[(b_n\cdot\nabla)u^{\psi}+g_n(u^{\psi})] \psi^n,   \\ 
   &u^{\psi}=0 \textup{ on }\partial\O,\\
   &u^{\psi}(0)=u_0. 
\end{cases}
\end{equation*}
\end{theorem} 

\subsubsection{\textbf{2D Navier--Stokes equation with Stratonovich transport noise}}\label{ss:NS}

Here we use Corollary \ref{cor:LDPcoerciveperturbed1} to derive the Stratonovich noise analog of \cite[Th.\ 5.4]{TV24}. The latter contains the LDP for the 2D Navier--Stokes equation with It\^o transport noise on bounded and unbounded domains, which solved a long-standing open problem, see the Introduction in \cite{TV24}. 

Consider the following Navier--Stokes system with no-slip condition and Stratonovich noise on an open set $\O\subset \R^2$, for small $\eps>0$:
\begin{equation}
\label{eq:NS}
\left\{
\begin{aligned}
&\ddd {u} =\big[\nu \Delta {u} -({u}\cdot \nabla){u} -\nabla p  \big] \,\dd t +\sqrt{\varepsilon} \textstyle{\sum_{n\geq 1}}\big[(b_{n}\cdot\nabla) {u} -\nabla \widetilde{p}_n \big] \circ\dd w^n,
\\
&\div \,{u}=0,
\\& {u}=0 \ \text{on $\partial \O$},
\\& {u}(0,\cdot)=u_0.
\end{aligned}\right.
\end{equation}
Here, $(w^n)$ is a sequence of independent real-valued standard Brownian motions, $\circ$ denotes Stratonovich integration, $u\coloneqq(u^{1},u^{2})\colon[0,\infty)\times \O\to \R^2$ is an unknown velocity field, $p,\widetilde{p}_n\col \R_+\times \O\to \R$ are unknown pressures and 
\begin{equation*}
(b_{n}\cdot\nabla) u\coloneqq\Big(\textstyle{\sum_{j=1}^2} b_n^j \partial_j u^k\Big)_{k=1,2},
\qquad (u\cdot \nabla ) u\coloneqq\Big(\textstyle{\sum_{j=1}^2} u^j \partial_j u^k\Big)_{k=1,2}.
\end{equation*}

\begin{assumption}\label{ass:NS}
$\nu>0$. For  $j\in\{1,2\}$, $(b^j_{n})_{n\geq 1}\col \O\to \ell^2$ is measurable and bounded. For all $n\in\N$, $\div\, b_n=0$ in the sense of distributions on $\O$, where $b_n=(b_n^1,b_n^2)$.  
\end{assumption}

As in \cite[\S 7.3.4]{AV25survey}, we use the Helmholtz projection $\P$  to rewrite \eqref{eq:NS} as an abstract evolution equation of the form \eqref{eq:SPDE perturbed}, putting $U = \ell^2$,  
\[
V \ceqq \Hs^1_0(\O) \ceqq H^1_0(\O;\R^2)\cap \Ls^2(\O),\quad H \ceqq \Ls^2(\O), \quad V^* \ceqq \Hs^{-1}(\O) = (\Hs^1_0(\O))^*,
\]
where $\Ls^2(\O)$ denotes the range of the Helmholtz projection in $L^2(\O;\R^2)$, 
and we use 
\begin{alignat*}{2}
A_0 &\ceqq -\nu \pr\Delta, &\quad   
B_0 u &\ceqq (\pr[(b_{n}\cdot\nabla) u])_{n\geq 1}, \\
\hat{F}(u) &\ceqq- \pr\div[u\otimes u],   &\quad
&G=\hat{f}=g=0,\\
\eta(\eps)&\ceqq \eps,  &\quad
\tilde{A}u&\ceqq 
\P\big[\div(a_b\cdot\nabla u)-\tfrac12\textstyle{\sum_{n\geq 1}}\div\big(b_n\otimes(I-\pr)[(b_n\cdot\nabla)u]\big)\big],
\end{alignat*} 
where  $a_b\ceqq (a_b^{i,j})_{i,j=1}^2\ceqq (\frac12\sum_{n\geq 1}b_n^jb_n^i)_{i,j=1}^2$ and we write $x\otimes y=(x_jy_k)_{j,k=1}^2$ for $x,y\in\R^2$.  
Note that  $\eta(\eps)\tilde{A}$ is the It\^o--Stratonovich correction term for the $\sqrt{\eps}$-noise equation \eqref{eq:NS}, as follows from the derivation in \cite[(1.4), Rem.\ 2.2]{AV24SNS} and since 
$\nabla\widetilde{p}_n=(I-\pr)[(b_n\cdot\nabla)u]$, see \cite[p.\ 43]{AV24SNS}.

Next, let us argue why the conditions of Corollary \ref{cor:LDPcoerciveperturbed1} are satisfied.  By \cite[\S7.3.4]{AV25survey}, $(A_0,B_0,\hat{F})$ defined above satisfy their corresponding conditions of Assumption \ref{ass:critvarsettinglocal} except for \ref{ass:critvarsettinglocal}\ref{it:coercivelinear} (since we do not assume parabolicity involving $(b_n^j)$ as in \cite[Asm.\ 7.13]{AV25survey}). 
Moreover, $(A_0,0)$ clearly satisfies Assumption \ref{ass:critvarsettinglocal}\ref{it:coercivelinear} and from \cite[\S7.3.4]{AV25survey} it follows that $(A,0)$ is coercive, so in particular    \eqref{eq:def coercive delta} holds for any $\delta>0$ (recall that $G=0$).  
Furthermore, $\tilde{A}\in\mathscr{L}(\Hs_0^1(\O),\Hs^{-1}(\O))$, by the boundedness of $b$ from Assumption \ref{ass:NS}. 
Therefore, Corollary \ref{cor:LDPcoerciveperturbed1} yields the following result.

\begin{theorem}\label{th:NS} 
Let Assumption \ref{ass:NS} hold and let $u_0\in \Ls^2(\O)$. Then there exists an $\eps_0>0$ such that  
for all $\eps\in(0,\eps_0)$, 
\eqref{eq:NS} has a unique strong solution $u^\eps$, and 
the family $(u^\eps)_{\eps\in(0,\eps_0)}$ satisfies the LDP on $\MR(0,T)\ceqq C([0,T];\Ls^2(\O))\cap L^2(0,T;\Hs^1_0(\O))$ with rate function $I\col \MR(0,T)\to[0,\infty]$ given by 
\[
I(z)=\frac12\inf\{\|\psi\|_{L^2(0,T;\ell^2)}^2:\psi \in L^2(0,T;\ell^2), z=u^{\psi} \},
\]
where $\inf\varnothing\ceqq+\infty$ and for $\psi=(\psi^n)\in L^2(0,T;\ell^2)$,  $u^\psi\in  \MR(0,T)$ denotes the unique strong solution to 
\begin{equation*} 
\left\{
\begin{aligned}
&(u^{\psi})' =\big[\nu\P \Delta u^{\psi} - \P\div(u^{\psi}\otimes u^{\psi})\big]  
+ \textstyle{\sum_{n\geq 1}} \P[(b_{n}\cdot\nabla) u^{\psi}]\psi^n,\\ 
&u^{\psi}=0 \textup{ on } \partial \O,\\ 
&u^{\psi}(0,\cdot)=u_0.
\end{aligned}\right.
\end{equation*} 
\end{theorem} 

\begin{remark}\label{rem:NS coer}
Although not needed for the LDP, we emphasize that 
  Assumption \ref{ass:NS} ensures that  \eqref{eq:NS} is  well-posed even for every $\eps\geq 0$. In fact, the It\^o--Stratonovich correction term is  helpful: coercivity of $(A_0,0)$ yields coercivity of   $(A_0-\eps \tilde{A},\sqrt{\eps}B_0)$ and $(A-\eps \tilde{A},\sqrt{\eps}B)$ for every $\eps\geq 0$, and global well-posedness follows from \cite[Th.\ 7.10]{AV25survey}. See also  \cite[(A.1), Asm.\ A.1, Prop.\ A.5]{AV24SNS}.  
 
For the claimed coercivity, observe that for all  $u\in \Hs_0^1(\O)$:
\[
-\<\eps\tilde{A}u,u\?=\tfrac{\eps}{2}\textstyle{\sum_{n\geq 1}}\|\P[(b_n\cdot\nabla) u ]\|_{\Ls^2(\O)}^2 =\tfrac12\|\sqrt{\eps}B_0u\|_{\mathscr{L}_2(\ell^2;\Ls^2(\O))}^2.  
\]
Thus we have a cancellation: 
\begin{equation}\label{eq:NSlincoer}
\<(A_0-\eps\tilde{A})u,u\?-\tfrac12\|\sqrt{\eps}B_0u\|_{\mathscr{L}_2(\ell^2;\Ls^2(\O))}^2=\<A_0u,u\?= 
{\nu}\|u\|_{\Hs^1_0(\O)}^2-\nu\|u\|_{\Ls^2(\O)}^2. 
\end{equation}    
Moreover, for $\hat{F}$ we have   
$
\<\hat{F}(u),u\?=(u\otimes u,\nabla u)_{\Ls^2(\O;\R^{2\times 2})}=0  
$ for all   $u\in \Hs_0^1(\O)$  (since $\P$ is self-adjoint, and $u$ is divergence free and has zero trace). 
Adding this to \eqref{eq:NSlincoer}, we conclude that $(A-\eta(\eps)\tilde{A},\sqrt{\eps}B)$ is coercive too.  
\end{remark}

\subsubsection*{Acknowledgement}
I would like to thank Antonio Agresti and Mark Veraar, as well as the anonymous referees, for their helpful comments.

\subsubsection*{Funding} The author is supported by the VICI subsidy VI.C.212.027 of the Netherlands Organisation for Scientific Research (NWO). 

\vspace{-.15cm}
\subsubsection*{Author contributions} This manuscript was written entirely by the author E.T. 

\vspace{-.15cm}
\subsubsection*{Data availability} No datasets were generated or analyzed during this study.

\section*{Declarations}

\subsubsection*{Conflicts of interest} The author has no relevant financial or non-financial interests.

\printbibliography 

\end{document}